%% file: HHHK1revisedperreferee.tex
\documentclass[11pt]{amsart}
\usepackage{amssymb, amscd, amsmath, amsthm, epsf,  latexsym,color} 
\usepackage{hyperref}
\usepackage{pinlabel}
\usepackage{pb-diagram}
\usepackage{soul}
\usepackage{tikz-cd}
\usepackage{float}
\usepackage{graphicx}
\usepackage{comment}

\usepackage{hyperref}

 \usepackage{color}

\numberwithin{equation}{section}

\newtheorem{theorem}{Theorem}[section]
\newtheorem{corollary}[theorem]{Corollary}
\newtheorem{lemma}[theorem]{Lemma}
\newtheorem{proposition}[theorem]{Proposition}

\newtheorem{question}[theorem]{Question}

\theoremstyle{definition}
\newtheorem{definition}[theorem]{Definition}
\newtheorem{example}[theorem]{Example}

\newcommand{\ep}{\epsilon}

\newcommand{\bbi}{{{\bf i}}}

\newcommand{\bbk}{{{\bf k}}}

\newcommand{\FF}{{\mathbb F}}
\newcommand{\ZZ}{{\mathbb Z}}
\newcommand{\RR}{{\mathbb R}}
\newcommand{\CC}{{\mathbb C}}
\newcommand{\TT}{{\mathbb T}}
\newcommand{\PP}{{\mathbb P}}
\newcommand{\cA}{{\mathcal A}}
\newcommand{\cD}{{\mathcal D}}
\newcommand{\cF}{{\mathcal F}}
\newcommand{\cG}{{\mathcal G}}

\newcommand{\cB}{{\mathcal B}}
\newcommand{\cL}{{\mathcal L}}
\newcommand{\cM}{{\mathcal M}}
\newcommand{\Obj}{\operatorname{Obj}}
\newcommand{\Tw}{\operatorname{Tw}}

\newcommand{\Kh}{ \operatorname{Kh}  } 

\newcommand{\RN}[1]{%
  \textup{\uppercase\expandafter{\romannumeral#1}}%
}

\newcommand{\nat}{\natural}

\DeclareMathOperator{\Hom}{hom}

\graphicspath{ {figures/} }

 \renewcommand{\qed}{\hfill$\square$}

\author[M.~Hedden]{Matthew Hedden}
\address{Department of Mathematics, Michigan State University, East Lansing, MI 48824} 
\email{mhedden@math.msu.edu}

\author[C.~Herald]{Christopher M. Herald}
\address{Department of Mathematics and Statistics,   University of Nevada, Reno, Reno, NV 89557} 
\email{herald@unr.edu}

\author[M.~Hogancamp]{Matthew Hogancamp}
\address{Department of Mathematics, Northeastern University, Boston, Massachusetts 02115 } 
\email{m.hogancamp@northeastern.edu}

 \author[P.~Kirk]{Paul Kirk}
\address{Department of Mathematics, Indiana University, Bloomington, IN 47405} 
\email{pkirk@indiana.edu}

\title[Pillowcase  and Khovanov cohomology]{The Fukaya category of the pillowcase, traceless character varieties,  and Khovanov cohomology }

\begin{document}

\date{Sept. 18, 2018}

\begin{abstract} For a diagram of a $2$-stranded tangle in the $3$-ball we define a twisted complex of compact Lagrangians in the triangulated envelope of the Fukaya category of the smooth locus of the pillowcase.  We show that this twisted complex is a functorial invariant of the isotopy class of the tangle, and that it provides a factorization of Bar-Natan's functor from the tangle cobordism category to chain complexes.  In particular, the hom set of our invariant with a particular non-compact Lagrangian associated to the trivial tangle is naturally isomorphic to the reduced Khovanov chain complex of the closure of the tangle.   Our construction comes from the geometry of traceless $SU(2)$ character varieties associated to resolutions of the tangle diagram, and was inspired by  Kronheimer and Mrowka's singular instanton link homology.   \end{abstract}

 \maketitle

\section{Introduction}
This article continues the investigation, set in motion by our earlier articles \cite{HHK1, HHK2, HK}, of the information contained in the traceless $SU(2)$ character varieties associated to a 2-stranded tangle decomposition of the link.  In \cite{HHK2} we considered a tangle decomposition \[(S^3,L)=(D^3, T_0)\cup_{(S^2,4)}(D^3,T_1),\] where  $T_0$ is a trivial 2-tangle in a 3-ball and $T_1$ its complement, and defined  a Lagrangian Floer complex for $L$ from two variants of the immersed traceless character varieties of the complements of the tangles in the decomposition, $R^\nat_\pi(D^3,T_0)$,  $R _\pi(D^3,T_1)$.  This theory, which we called {\em pillowcase homology}, takes place in the smooth part $P^*$  of the traceless character variety of the 4-punctured Conway sphere, $P=R(S^2,4)$, a  variety easily identified with the the quotient of the torus by the elliptic involution.   Our construction is an attempt to provide a symplectic counterpart to Kronheimer and Mrowka's {\em singular instanton link homology} \cite{KM1}.

The present article adopts a  more subtle approach, replacing one of the two Lagrangians, $R _\pi(D^3,T_1)$, by a twisted complex over the Fukaya category of the punctured pillowcase obtained from the cube of resolutions associated to a projection of $T_1$.   Our main result, proved in Sections \ref{functors}, \ref{TwistedSister}, and \ref{mainresults}, 
 states that Bar-Natan's functor from the category of 2-tangles to the category of  chain complexes, defined by sending a 2-tangle to the Khovanov cube associated to the closure of the tangle, factors as a composition of two $A_\infty$-functors through the $A_\infty$-category of twisted complexes over the Fukaya category of the pillowcase.  An immediate consequence of this factorization is the following theorem:

\begin{theorem}\label{thm:main} To a diagram of a $2$-stranded tangle in the $3$-ball, $T\subset D^3$ there is a twisted complex of graded Lagrangians $(L_T,\delta_T)$ lying in the triangulated envelope of the Fukaya category of the punctured pillowcase $\Tw {\rm Fuk}(P^*)$.   The homotopy type of $(L_T,\delta_T)$ is an invariant of the isotopy type of $T$ rel boundary.   
\end{theorem}

\noindent See Sections \ref{TwistedSister} and \ref{mainresults} for more details, particularly regarding the (bi)grading of $(L_T,\delta_T)$. We should also mention that, since we have factored Bar-Natan's functor through   $\Tw {\rm Fuk}(P^*)$, our invariant automatically inherits  functoriality with respect to tangle cobordisms in $D^3\times [0,1]$. 

As a corollary of Theorem \ref{thm:main},  and the fact that Bar-Natan's functor factors through $\Tw {\rm Fuk}(P^*)$,  we realize the (reduced) Khovanov cohomology of a link as the  Lagrangian Floer cohomology of  a pair of   twisted complexes of Lagrangians associated with traceless character varieties.   To state this result precisely,  recall that morphism spaces in $A_\infty$-categories are cochain complexes, and hence the hom space from any fixed Lagrangian to  $(L_T,\delta_T)$   is a cochain complex.  To the trivial tangle $T_0$,  there is an easily identified Lagrangian arc $W_0$  in the pillowcase, arising as the restriction of the traceless character variety of the complement of $T_0$ to the corresponding character variety of the Conway sphere.  We recover reduced Khovanov homology (\cite{Kh, KhTangle}) of the closure of $T$ (using $T_0$ as the complementary tangle) as the cohomology of the space of morphisms between $W_0$ and $(L_T,\delta_T)$:

\begin{corollary} Let $T\subset D^3$ be a 2-stranded tangle, and let $\widehat{T}$ be the link obtained as its $0$-closure by the trivial tangle.  Then we have an isomorphism of   $\ZZ\oplus\ZZ$-bigraded $\FF_2$ vector spaces
\[ \Kh^{\rm red}( \widehat{T}) \cong H^*(\Hom_{}(W_0,(L_T,\delta_T))),\]
where the $\Hom$ space is taken within $\Tw {\rm Fuk}(P^*)$.
\end{corollary}

The corollary provides a symplecto-geometric interpretation of Khovanov homology, which should be compared to  the {\em symplectic Khovanov homology} of Seidel and Smith \cite{SeidelSmith}.  In that theory there is a specific Lagrangian embedding $(S^2)^m\hookrightarrow \mathcal{Y}_{m,t}$, where the manifold  $\mathcal{Y}_{m,t}=\pi^{-1}(t)$  arises as a fiber in a fibration over $\mathrm{Conf}^0_{2m}(\mathbb{C})$, the configuration space of massless collections of $2m$ distinct points.  The fibration is naturally defined by the geometry of the adjoint action of SL$_{2m}(\mathbb{C})$ on $\mathfrak{sl}_{2m}(\mathbb{C})$.   A closed $m$-stranded braid induces a monodromy map on the fiber, and their invariant is the Lagrangian Floer cohomology of the $(S^2)^m$ above with its image under the monodromy.  They conjectured their invariant coincides with Khovanov homology of the closed braid, and this was later confirmed by Abouzaid and Smith \cite{AS1,AS2}.

Our theory is in some sense much simpler, taking place as it does in a fixed Riemann surface where all Floer differentials and higher $A_\infty$-operations  are combinatorial in nature.  It has the deficit, however, of being defined in the spirit of Khovanov homology; namely, our twisted complexes are built from a collection of Lagrangians, each arising as the traceless character variety of a complete resolution of the crossings in a diagram of the given tangle. The size of our twisted complex is therefore exponential in the number of crossings in the tangle diagram.

We arrived at the constructions of the present paper by attempting to iterate an exact triangle in the Fukaya category of the pillowcase  arising from the traceless character varieties of the three 2-stranded tangles defining the unoriented skein relation.  The pillowcase homology of \cite{HHK2} mentioned above  possesses a long exact sequence for the unoriented skein relation as a consequence of this exact triangle in the Fukaya category.  The present work was an attempt to bypass the computational difficulties involved in (i) understanding the Lagrangians associated to traceless character varieties of tangles, and (ii) proving that the pillowcase homology associated to two such Lagrangians is an invariant of the underlying link.  An obvious trade-off is an apparent increase in the complexity of the twisted complex.    Despite this, a general result of Haiden, Katzarkov, and Kontsevich \cite[Theorem 4.3]{HKK} implies that the additional complexity is unnecessary:
\begin{corollary}\label{realization} The twisted complex $(L_T,\delta_T)$ associated to a tangle may be represented by a isotopy class of immersed curves equipped with local systems.  
\end{corollary}
Of course if one were given the collection of immersed curves with local systems associated to $T$, computing its Floer cohomology with $W_0$ -- and hence the Khovanov homology -- would be easy.  Indeed, it would be expressed as a quickly computed function of the geometric intersection number of $T$ with $W_0$ and the ranks of the local systems involved.   It would therefore be very interesting to have a concrete way of producing  curves with local systems from our twisted complexes.  Such a method has been implemented in the case of bordered Heegaard Floer invariants of 3-manifolds with torus boundary by Hanselman-Rasmussen-Watson \cite{HRW}.    In our context,  an algorithm of this form could lead to  more efficient divide and conquer approaches to Khovanov homology computations, and to rank inequalities for the Khovanov homologies of links which differ by tangle substitutions.

An interesting aspect of our work is that, while our construction grew out of an attempt to understand singular instanton homology -- a theory which is the $E_\infty$ page of a spectral sequence beginning at Khovanov homology \cite{KM2} -- we have only recovered Khovanov homology.   A general feature of twisted complexes over an $A_\infty$ category is that their morphism spaces are naturally {\em filtered} cochain complexes.  Our initial expectation was that the spectral sequence associated to the filtration of $\Hom(W_0,(L_T,\delta_T))$ would be isomorphic to Kronheimer and Mrowka's spectral sequence \cite{KM2}.  That our spectral sequence collapses at Khovanov homology seems interesting, and motivates the question:

\begin{question} Can one add higher  order terms to the differential on $(L_T,\delta_T)$ so that, upon pairing with $W_0$, the resulting spectral sequence  is isomorphic to Kronheimer and Mrowka's?\end{question}

A key feature of our construction, and which underlies the collapse of the spectral sequence at Khovanov homology, is the structure of a particular full subcategory $\cL\subset {\rm Fuk}\ P^*$ of the Fukaya category of the punctured pillowcase.  This category is generated by the  two  immersed  Lagrangians corresponding to the traceless character varieties of the two planar 2-tangles in a 3-ball without closed components.  Indeed, while we have stated our results in the introduction in terms of $\Tw {\rm Fuk}\ P^*$, our invariant actually take values in $\Tw \cL$.  It is for this reason that we are able be somewhat careless in our treatment of  ${\rm Fuk} \ P^*$, which should really be a version of the wrapped Fukaya category.  Since the Lagrangians in $\cL$ are compact and we at present only have occasion to pair them with a single non-compact arc, we neglect to account for these subtleties at the moment.

The $A_\infty$-algebra generated by $\cL$ bears striking similarities to Khovanov's arc algebra $H^2$ \cite{KhTangle} and the analogous (isomorphic) algebra appearing in Bar-Natan's work \cite{BN}.   An important technical result for our set-up is a complete computation of the subcategory $\cL$.  We paraphrase it as follows:

\begin{theorem}(Theorem \ref{maincalc} and Proposition \ref{Feasy}) The $A_\infty$-algebra \[ A_\cL := \bigoplus_{L_i,L_j\in\cL} \hom_\cL(L_i,L_j)\]  is 12 dimensional and minimal i.e. $\mu^1\equiv 0$.  As an ungraded algebra, it is isomorphic to Khovanov's arc algebra $H^2$.  Moreover  $\mu^n\equiv 0$ for all $n>3$.  However, there are 24 non-trivial $\mu^3$ operations.
\end{theorem}

\noindent  It is again interesting to compare our work with symplectic Khovanov homology.  A key ingredient  in the proof that symplectic Khovanov homology agrees with Khovanov homology was a formality result for a particular $A_\infty$ subcategory of ${\rm Fuk}\ \mathcal{Y}_{m,t}$ which is quasi-equivalent to the symplectic analogue of $H^n$ \cite{AS1}.  

Our theorem suggests that, unlike the symplectic arc algebra, the category $\cL$ is probably not formal, since it is minimal, yet possesses non-trivial higher operations.
  Despite the non-trivial $\mu^3$ maps, the differentials in the spectral sequence associated to the filtration of $\Hom_{\Tw {\rm Fuk}(P^*)}(W_0,(L_T,\delta_T))$ vanish;  it would be interesting to have a better conceptual explanation of this fact,
 or, alternatively, explicit examples of twisted complexes  $(D,d)$ associated to tangles so that the spectral sequence associated to the filtration of $\Hom_{\Tw {\rm Fuk}(P^*)}((D,d),(L_T,\delta_T))$ has non-trivial higher differentials.

Finally, we compare our work to analogous results in Heegaard Floer theory.  In that setting a $3$-manifold whose boundary is a parametrized surface is assigned, by {\em bordered Floer homology}, an $A_\infty$-module over an $A_\infty$-algebra associated to the surface \cite{LOT}.  Similar structures exist for tangles \cite{Petkova}.  Originally defined using the symplectic field theoretic framework of Lipshitz's cylindrical version of Heegaard Floer theory \cite{Lipshitz}, the bordered theory was recast in terms of $A_\infty$-modules over a particular $A_\infty$-algebra generated by a collection of Lagrangians in the partially wrapped Fukaya category of a symmetric product of the boundary surface (suitably punctured) \cite{auroux2,auroux3,LekiliPerutz}. Important special cases have been worked out in concrete detail; for 3-manifolds with torus boundary by Hanselman-Rasmussen-Watson \cite{HRW}, and for 2-stranded tangles by Zibrowius \cite{Zibrowius}.     These are  formally analogous to our framework, and our invariant should be viewed as a type of bordered Khovanov invariant for a 2-tangle.   It would also be interesting to compare ours to other, more combinatorial constructions, of bordered Khovanov invariants, e.g. \cite{RobertsI, RobertsII,Manion,KhTangle}.   In light of the bordered interpretation, one would expect that the cohomology of the hom spaces between two 2-tangles is isomorphic to the Khovanov homology of their union.   A technical point in our construction inherited from its origins in gauge theory is that the two tangles in our decomposition of a given link are not on  equal footing:  one is endowed with an ``earring" which allows us to define a non-trivial $SO(3)$ bundle and avoid the reducibles.  Concretely, this allows us to work with twisted complexes built from compact objects in the Fukaya category.  For this reason, the correct interpretation of our hom spaces is as follows:
\[ H^*(\Hom((L_V,\delta_V),(L_T,\delta_T)))\cong \Kh^{\rm red}( \overline{V}\cup T)\otimes \mathbb{F}^2\]
\noindent In words, the Khovanov homology of the link arising as the union of a tangle $T$ with the mirror of a tangle $V$ can be recovered, up to tensoring with a particular two dimensional vector space, as the cohomology of the space of morphisms between the twisted complexes we associate to $V$ and $T$, respectively.    We should point out that similar results in these directions, and more precise comparisons with bordered Floer homology, have been obtained independently by Kotelskiy \cite{Kotelskiy,Kotelskiy2}.\\

 \noindent {\bf Organization:}
 In the next section, we briefly review basic notions concerning $A_\infty$-categories, $A_\infty$-functors, and twisted complexes in detail suitable for our purposes.  Section \ref{fig8lag} introduces the pillowcase $P$, its smooth stratum $P^*$,  and the immersed curves $L_i,W_i$. We defer to Appendix \ref{tracelesssect} an explanation of how these arise as traceless flat moduli spaces of tangles, and a proof that the Lagrangian correspondence associated to a $0$-handle cobordism has the effect on traceless moduli spaces of doubling each component or, more algebraically, of taking a tensor product with a 2-dimensional vector space.

 Section \ref{fukayacomp}  summarizes the calculation of the  full $A_\infty$-subcategory $\cL\subset  {\rm Fuk} \ P^*$ of the Fukaya category  of curves in the pillowcase, given in Theorems \ref{maincalc} and \ref{testthm}.  We relegate the rather technical work involved in these calculations to Appendix \ref{APa}; there, we indicate how the calculations are carried out, and in particular why the relevant $\mu^n$ vanish for $n\ge 4$. Section  \ref{sec:gradings} discusses gradings  and  we construct $\ZZ$-gradings in $\cL$ (lifting relative $\ZZ/4$-gradings coming from the gauge theoretic considerations of \cite{HHK2}).
 
 Section \ref{BNsect} introduces a variant of Bar-Natan's \cite{BN} 2-tangle category from  which we produce, in Section  \ref{BNsect}, an $A_\infty$-functor to our Fukaya category $\cL$.  This functor extends to a functor on the triangulated envelope of twisted complexes, and the image of Bar-Natan's twisted complex associated to a tangle is the twisted complex over $\cL$ which we assign to a tangle diagram.
 Section  \ref{BNsect} also produces   functors  from $\cL$ to the dg-category of chain complexes by evaluating against test curves,  such as $W_0$ and $W_1$.  
 
 The properties of these functors and their extensions to  twisted complexes are explored in Sections \ref{functors} and \ref{TwistedSister}, culminating in  Section \ref{mainresults}, which contains our main results, notably exhibiting that the homotopy class of the  twisted complex over $\cL$ which we assign to a 2-tangle is an isotopy invariant of the tangle, and that evaluating against the test curves $W_0$ and $W_1$ yields the $\ZZ\oplus\ZZ$-bigraded  reduced Khovanov chain complexes of the two planar closures of the given tangle.  We illustrate the construction in Section \ref{example}, with an example for a 2-tangle whose closure is the trefoil knot.  
 \\

\noindent{\bf{Acknowledgements:}} The authors are pleased to thank Guillem Cazassus, Artem Kotelskiy, Henry Horton, and Chris Rogers,  for extremely helpful discussions. Kotelskiy, in particular, pointed out Corollary \ref{realization} to us and informed us of his closely related work. The authors also thank the referee for suggestions which greatly improved the exposition.   Matthew Hedden was partially supported by NSF CAREER grant DMS-1150872, NSF DMS-1709016 and an Alfred P. Sloan Research Fellowship during the course of this work. Chris Herald and Paul Kirk were partially supported by Simons Collaboration Grants for Mathematicians.  Matthew Hogancamp was partially supported by NSF DMS-1702274.

\section{$A_\infty$-categories, Additive enlargements and twisted complexes} \label{category}

In this section we briefly review the  notions of  $A_\infty$-categories, their additive enlargements and twisted  envelopes, and of $A_\infty$-modules and $A_\infty$-functors. We follow and refer the reader to Seidel's book \cite{Seidel} for careful constructions and precise definitions; our aim here is simply to collect notation and introduce  terminology to the uninitiated reader.  Since we  restrict our attention throughout the paper to  vector spaces over the field with two elements, which we  will denote by $\FF$,  our discussion of the subject will be eased by an omission of any signs.

An {\em $A_\infty$-category} $\cA$ consists of a collection of objects $X$, graded $\FF$ vector spaces $\Hom_\cA(X_0,X_1)$ for any pair $(X_0,X_1)$ of objects,  and for $n\ge 1$ linear structure maps of degree $2-n$,
$$\mu^n_\cA:\Hom_\cA(X_{n-1},X_{n})\otimes\dots\otimes \Hom_\cA(X_{0},X_{1})\to \Hom_\cA(X_{0},X_{n}),$$ 
which satisfy the $A_\infty$-relations  \cite[Equation (1.2)]{Seidel}.  The first two relations are
$$\mu^1_\cA \left( \mu^1_\cA(a_1)\right) =0 \text{ 
and }
\mu^1_\cA \left( \mu^2_\cA(a_2,a_1)\right) +\mu^2_\cA\left(\mu^1_\cA(a_2),a_1\right)
+\mu^2_\cA\left(a_2, \mu^1_\cA(a_1)\right)=0$$

An $A_\infty$-category is called {\em strictly unital} provided that for each object $X$ there exists an $e_X\in \Hom_\cA(X,X)$ of grading 0 so that 
$\mu^1_\cA(e_X)=0,$ $\mu^2_\cA(e_X,a_1)=a_1, $ $ \mu^2_\cA(a_1, e_X)=a,$ and 
$$\mu^n_\cA(a_n,\dots, a_{k+1}, e_X, a_{k-1}, \dots, a_1)=0,~ n>2$$
for any $a_i$ for which these expressions are defined.  A strictly unital  $A_\infty$-category for which $\mu^n=0$ when $n>2$ is called a differential graded category, or dg-category.  

\medskip

A strictly unital  $A_\infty$-category $\cA$ has a {\em cohomology category} $H(\cA)$; its objects are the  same  as the objects of $\cA$ and its morphisms are given by $$\Hom_{H(\cA)}(X,Y)=\frac{\ker \mu^1_\cA:\Hom_\cA(X,Y)\to \Hom_\cA(X,Y)}{\text{Image }\mu^1_\cA:\Hom_\cA(X,Y)\to \Hom_\cA(X,Y)},$$
with   composition induced by $\mu^2_\cA$.  A useful consequence for the present article is that if $\cA$ is a strictly unital $A_\infty$-category satisfying  $\mu^1_\cA=0$, then $\Hom_{H(\cA)}(X,Y)= \Hom_\cA(X,Y)$ and hence  in $\cA$,  the composition $\mu^2_\cA$ is associative and unital. Thus ignoring  all the $\mu^n$ for $n>2$ yields a category.  

\medskip
 
\begin{example}  A  trivial example of a strictly unital $A_\infty$-category is the category   of  graded vector spaces with all linear maps between these vector spaces, graded by degree, as its morphisms. Take $\mu^2$ to be composition and $\mu^n=0$ for $n\ne 2$.  The identify morphism serves as $e_X$. 
  \end{example}
  
  \begin{example}\label{exofCh} Another rather trivial example, which nonetheless will turn out to be relevant to the present article, is the dg-category of chain complexes, denoted by $Ch$.  The objects of $Ch$ are graded $\FF$ vector spaces $C$ equipped with endomorphisms $d_C: C\to C$ of degree 1, satisfying $d_C \circ d_C=0$.  Given $C_1,C_2\in \Obj(Ch)$, the morphism space $\Hom_{Ch}(C_1, C_2)$ is the  graded vector space  of all linear maps from $C_1$ to $C_2$,  graded by degree.  The structure maps $\mu^r _{Ch} $ are zero except when $r=1,2$.  For $a\in \Hom_{Ch}(C_0,C_1)$, $\mu^1 _{Ch} (a) = d_{C_1}\circ a + a\circ d_{C_0}$, and $\mu^2 _{Ch}$ is just composition of linear maps.    The morphism spaces of the associated homology category are simply the vector spaces generated by chain homotopy classes of chain maps.
  \end{example}

  \medskip
  A  (left)  {\em $A_\infty$-module $\cM$ over an $A_\infty$-category $\cA$} consists of  a graded vector space $\cM(X)$ for each object $X\in \cA$, together with structure maps 
  $$\mu^n_\cM:\Hom_\cA(X_{n-1},X_{n})\otimes\dots\otimes \Hom_\cA(X_{1},X_{2})\otimes \cM(X_{1})\to \cM(X_{n})$$
  which satisfy the obvious analogues of the $A_\infty$-relations (see \cite[Section (1j)]{Seidel}).  For example, one can fix an object $W\in \cA$ and define $\cM(X)=\Hom_\cA(W,X)$.   Taken together, these modules comprise the {\em Yoneda embedding}.

\medskip

An $A_\infty$-{\em functor} $\cF:\cA\to \cB$ consists of an assignment ${\rm Obj}(\cA)\ni X\mapsto \cF(X)\in  {\rm Obj}(\cB)$ together with a sequence of multilinear maps of degree $1-n$
$$\cF^n:\Hom_\cA(X_{n-1},X_n)\otimes\dots\otimes\Hom_\cA(X_{0},X_1)\to \Hom_\cB(\cF(X_{0}),\cF(X_n)),~ n=1,2,\dots$$
satisfying certain compatibility equations with respect to the structure maps \cite[Equation (1.6)]{Seidel}.  Note that an $A_\infty$-module over $\mathcal{A}$ is simply an $A_\infty$-functor from $\mathcal{A}$ to $Ch$.

An $A_\infty$-functor between strictly unital $A_\infty$-categories is called {\em strictly unital} provided that $\cF(e_X)=e_{\cF(X)}$  and 
$$\cF^n(a_n,\dots, a_{k+1}, e_X, a_{k-1}, \dots, a_1)=0, n\ge 2.$$ 
for all objects $X$ of $\cA$.

\medskip

\begin{example}\label{exmorita} Given any $A_\infty$-category $\cA$ and object $A\in \Obj(\cA)$, one can define an $A_\infty$-functor $\cG_A$ from $\cA$ to  $Ch$, by
$$\cG_A:{\rm Obj}(\cA)\to {\rm Obj}(Ch), ~ \cG_A(C)=\big(\Hom_\cA(A,C),\mu^1_\cA\big)$$
and
 \begin{multline}
\label{vsfunctor}
\cG_A^n:\Hom_\cA(C_{n-1}, C_n)\otimes\dots\otimes \Hom_\cA(C_0,C_1)\to \Hom_{Ch} (\Hom_\cA(A,C_0),\Hom_\cA(A,C_n))\\
\cG_A^n(a_n,\dots, a_1)(b)= \mu^{n+1}_\cA(a_n,\dots, a_1, b)\hskip1.2in
\end{multline}
The $A_\infty$-relations satisfied by the higher multiplication operations $\mu^n _{\cA}$ translate into the required conditions for $\cG_A$ to be an $A_\infty$-functor.

 \end{example}

Given an $A_\infty$-category $\cA$,  its {\em additive enlargement} $\Sigma\cA$  is the $A_\infty$-category  with objects consisting of triples $(I,\{V_i\}_{i\in I}, \{X_i\}_{i\in I})$ where $I$ is a finite indexing set, the $X_i$ are objects of $\cA$,  and the $V_i$ are graded $\FF$ vector spaces, 
 We  indicate the gradings by a superscript.  As usual, we use the more convenient notation
$$\bigoplus_{i\in I} V_i\otimes X_i$$
for such a triple.  
 The morphism spaces are given by 
$$\Hom_{\Sigma\cA}\big(\bigoplus_{i} V_i\otimes X_i,\bigoplus_{j} V_j\otimes X_j\big)=\bigoplus_{i,j}\Hom_{\FF}(V_i,V_i)\otimes \Hom_\cA(X_i,X_j).
$$
The vector spaces $\Hom_{\FF}(V_i,V_j)\otimes \Hom_\cA(X_i,X_j)$ are given    the tensor product grading. The structure maps $\mu^n_{\Sigma \cA}$  are  given by composing the linear homomorphism  factors and applying $\mu^n_\cA$ to the $\Hom_\cA(X_i,X_j)$ factors.   

If $\mu^n_\cA=0$ for $n>k\ge 2$, then $\mu^n_{\Sigma\cA}=0$ for $n>k$.
The $A_\infty$-category $\Sigma\cA$ admits a {\em shift} operation which replaces $V\otimes X$ by $(V\otimes X) \{\sigma\}:=V\{\sigma\}\otimes X$, where $V\{\sigma\}$ is obtained from the graded vector space $V$ by shifting down by $\sigma$, $$\left(V\{\sigma\}\right) ^{d}=V^{ d+\sigma }.$$  If $\alpha:W\to V$ is a homomorphism of degree $k$, then $\alpha$ induces a homomorphism $\alpha:W\{i\} \to V\{j\}$ of degree $k-j+i$ (which we continue to denote by $\alpha$).  Note that $\cA$ includes into $\Sigma\cA$ via the $A_\infty$-functor $X\mapsto \FF\otimes X$.

\medskip

The $A_\infty$-category $\Tw \cA $ of {\em twisted complexes over $\cA$} has objects $(X,\delta)$ where $X$ is an object of $\Sigma \cA$ and  $\delta\in \Hom_{\Sigma\cA}(X,X)$ has grading 1 and satisfies 
$$\sum_n\mu^n_{\Sigma\cA}(\delta,\dots,\delta)=0.$$
In addition, $X$ is required to admit a filtration with respect to which  $\delta$ is strictly lower triangular.
The morphism spaces are the same as in $\Sigma\cA$:
$$\Hom_{\Tw \cA}((X, \delta_X), (Y,\delta_Y))=
\Hom_{\Sigma\cA}( X,Y).$$
 The structure maps $\mu^n_{\Tw \cA}$ incorporate $\delta$ as well as the morphisms that are the inputs.  We refer to Equation (3.20) of \cite{Seidel} for the general definition, but list here the first few terms in the formulas, from which the reader can easily obtain the general formula. \begin{align*}\mu^1_{\Tw \cA}(a)&=\mu^1_{\Sigma \cA}(a) +
\mu^2_{\Sigma \cA}(\delta_Y,a)+\mu^2_{\Sigma \cA}(a,\delta_X) \\ &\qquad +
\mu^3_{\Sigma \cA}(\delta_Y,\delta_Y,a)+\mu^3_{\Sigma \cA}(\delta_Y, a,\delta_X)+\mu^3_{\Sigma \cA}(a,\delta_X,\delta_X)+\cdots,
\end{align*}
$$\mu^2_{\Tw \cA}(a_2,a_1) =\mu^2_{\Sigma \cA}(a_2,a_1) +
\mu^3_{\Sigma \cA}(\delta_Z,a_2,a_1)+\mu^3_{\Sigma \cA}(a_2,\delta_Y,a_1)+\mu^3_{\Sigma \cA}(a_2,a_1,\delta_X)+\cdots,  $$
and
$$\mu^3_{\Tw \cA }(a_3,a_2,a_1) =\mu^3_{\Sigma \cA}(a_3,a_2, a_1)+\cdots.
$$

The filtration requirement ensures that these are finite sums. Our examples  have $\mu^n_{\Sigma\cA}=0$ for $n\ge 4$, so that these three formulas will suffice for our needs. 

The $A_\infty$-category $\cA$, as well as its additive enlargement $\Sigma\cA$,  fully faithfully embeds in $\Tw \cA$ via $X\mapsto (X,0)$. 
If $\cA$ is strictly unital, then so is $\Tw \cA$. A (strictly unital) $A_\infty$-functor $\cF:\cA\to\cB$ extends to a (strictly unital)
$A_\infty$-functor $\Tw \cF :\Tw \cA\to\Tw\cB.$

\begin{example} \label{exofCh2}  Continuing Example \ref{exofCh}, it is easy to see that the additive enlargement of the dg-category $Ch$ of chain complexes is again $Ch$.  

If  $((A,d_A),\delta)\in\text{Obj}(\Tw(Ch))$, then
the   endomorphism $\delta:(A, d_A)\to (A,d_A) $ satisfies   
$$0=\sum_n \mu_{Ch} ^n (\delta, \dots, \delta )   =\mu_{Ch} ^1 (\delta) + \mu_{Ch} ^2 (\delta, \delta) = d_A \circ \delta + \delta \circ d_A + \delta \circ \delta.$$ Since $d_A\circ d_A=0$, this is equivalent to the condition that $d_A + \delta \in \Hom_{Ch} (A,A)$ has square zero, i.e.,  that $(A, d_A+\delta)$ is again a chain complex. 
 The assignment 
 \begin{equation*} 
((A,d_A),\delta) \mapsto (A, d_A+\delta) 
 \end{equation*}
extends to a dg-functor $\Tw(Ch)\to Ch$.   To avoid burdensome notation, we will not give this functor a formal name. \end{example}

 \section{Figure 8 Lagrangians in the pillowcase}\label{fig8lag}

The {\em pillowcase} $P$ is the quotient of $\RR^2$ by the group generated by the three homeomorphisms $(\gamma,\theta)\mapsto(\gamma+2\pi, \theta)$, $(\gamma,\theta)\mapsto (\gamma,\theta+2\pi)$ and $(\gamma,\theta)\mapsto(-\gamma,-\theta)$. Equivalently, it is $\TT^2/(\ZZ/2)$, the quotient of the torus by the elliptic involution $(e^{\bbi\gamma}, e^{\bbi\theta})\mapsto (e^{-\bbi\gamma}, e^{-\bbi\theta})$. The lattice $(\pi\ZZ)^2$ maps to four points in $P$ which we call the {\em singular points}, and we denote the complement of these four singular points by $P^*$.    Note that the symplectic form $dx\wedge dy$ on $\RR^2\setminus (\pi\ZZ)^2$ descends to $P^*$.

\medskip

We define six  immersed curves, $L_0,L_1, L_\times, W_0, W_1,$ and $W_\times$,  in $P$. Fix a small $\delta>0$.
 Denote by $[\gamma,\theta]\in P$ the image of $(\gamma,\theta)\in \RR^2$ under the branched cover $\RR^2\to P$.

\begin{equation}\label{theW}
W_\ell :t\in [0,\pi]\mapsto
\begin{cases}   [t,t]  & \text{ if } \ell=0,\\
[\pi, -t]    & \text{ if } \ell=1,\\
[-t+\pi,0]   & \text{ if } \ell=\times.
\end{cases}
\end{equation}

 \begin{equation}\label{theL}
L_\ell,~t\in \RR/(2\pi)\mapsto
\begin{cases} 
[\tfrac\pi 2 + t +\delta\sin t,\tfrac\pi 2 + t -\delta\sin t]& \text{ if } \ell=0,\\
 [\pi-2\delta\sin t ,\tfrac\pi 2 -t -\delta\sin t] & \text{ if } \ell=1,\\
[\tfrac\pi 2 -t +\delta\sin t,  2\delta\sin t]& \text{ if } \ell=\times.\\
\end{cases}
\end{equation}
   These curves are illustrated in Figure \ref{exacttriplesfig}. The $W_\ell$ are embedded intervals, and the $L_\ell$ are immersed circles with precisely one transverse double point.

\begin{figure} 
\begin{center}
\def\svgwidth{4.4in}
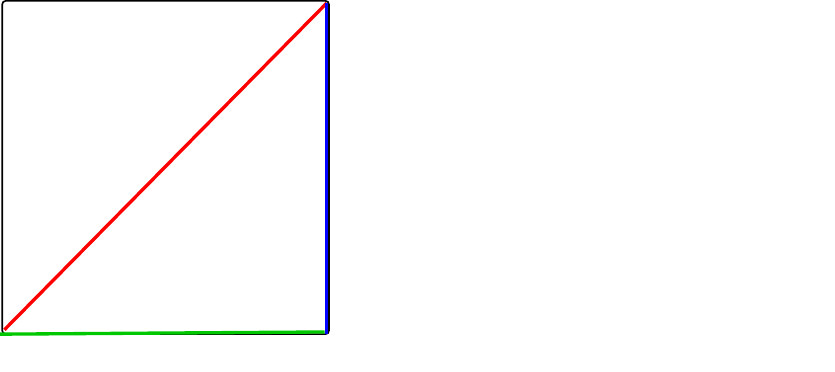
 \caption{\label{exacttriplesfig} }
\end{center}
\end{figure}

The function $f(\gamma,\theta)=-\cos (\gamma)\cos(\theta-\gamma)$ on $\RR^2$ is invariant under the action giving rise to $P$, and hence its Hamiltonian vector field 
 \begin{equation}
\label{hamiltonvf}
X_f= \cos(\gamma)\sin(\theta-\gamma)\tfrac {\partial}{\partial \gamma} +(-\sin(\gamma)\cos(\theta-\gamma)+\cos(\gamma)\sin(\theta-\gamma))\tfrac {\partial}{\partial \theta}
\end{equation}
is as well.  Thus both descend to  the pillowcase.  The vector field $X_f$ is illustrated in Figure \ref{vffig}.     Let $H:=H(\gamma,\theta,s)$ denote  the corresponding Hamiltonian flow on the pillowcase. 

\begin{figure} 
\begin{center}
\includegraphics[width=.6\textwidth]{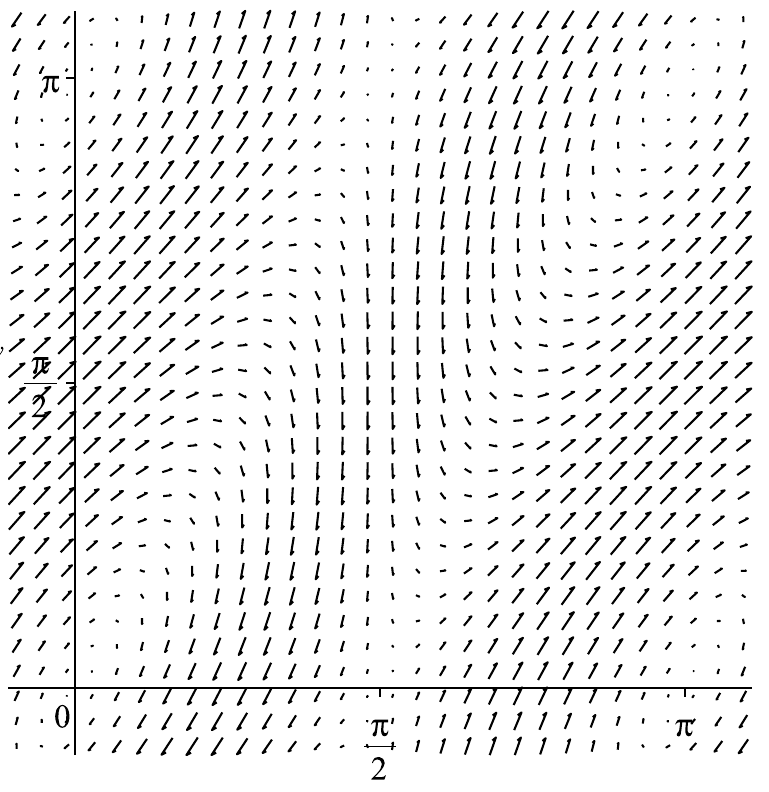}
 \caption{\label{vffig} }
\end{center}
\end{figure}

With $H$  
 defined (and recalling that $L_0$ and $L_1$ depend on a choice of small $\delta>0$),  the following elementary, if somewhat technical, result is readily proved.

\begin{theorem}\label{lookslike} Given a collection $C_i, ~ i=1,\dots ,n$ of curves with each $C_i\in   \{L_0,L_1\}$, there exists  an $\epsilon_0>0$, so that for all increasing sequences $0=s_1<s_2<\dots< s_{n+1}\leq \epsilon_0$, 
the sequence of curves:
\begin{multline*}\hskip 1in C_1=H(C_1,s_1), ~C_2'=H(C_2,s_2), ~C_3''=H(C_3,s_3), \dots, \\ C_n^{(n)}=H(C_n,s_n),  ~W_0^{(n+1)}=H(W_0,s_{n+1})\hskip1in\end{multline*} are pairwise transverse, their union has no triple points,   and up to ambient isotopy are configured in the pattern illustrated in Figure \ref{Heraldlemmafig}. \end{theorem}
  \begin{figure} 
\begin{center}
\def\svgwidth{2.6in}
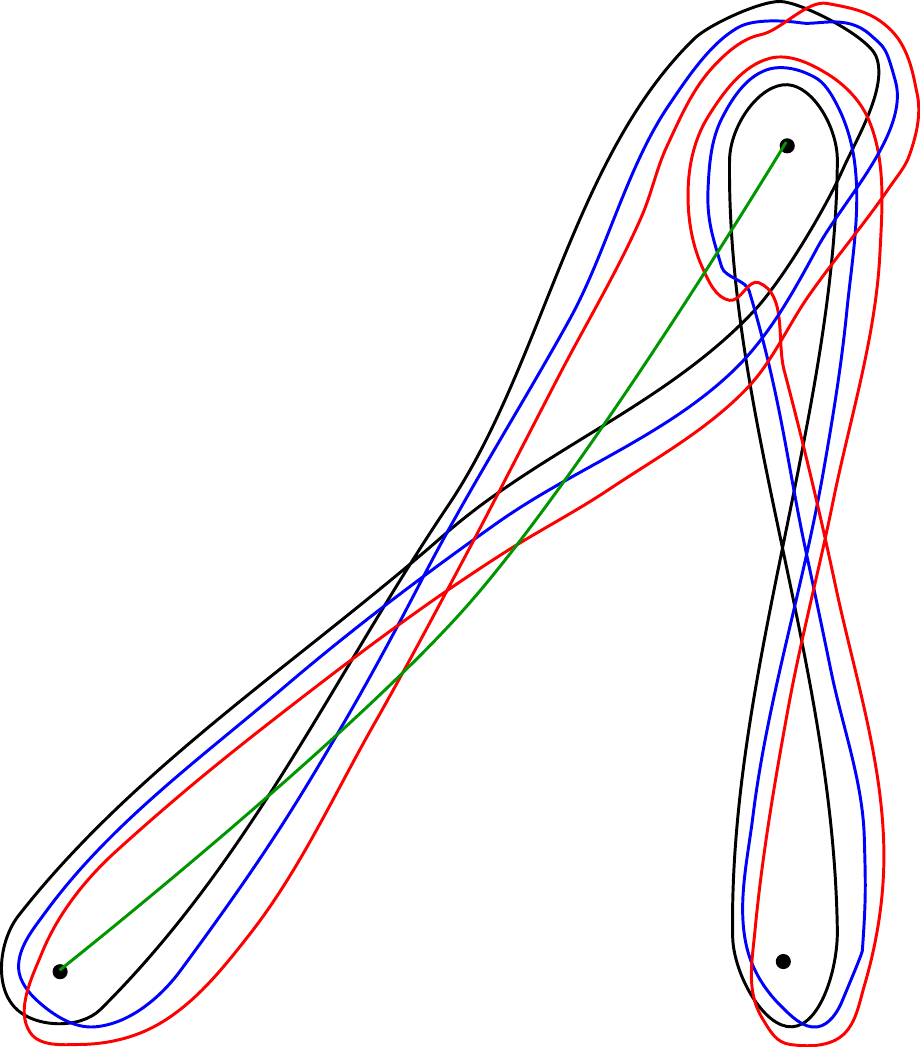
 \caption{\label{Heraldlemmafig}}
\end{center}
\end{figure} 
  Figure \ref{Heraldlemmafig} 
 illustrates the successive Hamiltonian pushoffs $L_0,L_0',L_0''$ and $L_1,L_1'',L_1'''$, as well as $W_0'''$ and $W_1'''$, as in Theorem \ref{hamiltonvf}.   The top stratum $P^*$ of the pillowcase  has been identified with the thrice-punctured plane, such that the singular point with coordinates $(0,\pi)$ is  mapped to $\infty$.  The important features of this figure for the calculation of the $\mu^k$ are the relative locations of the intersection points of $L_i^{(n)}$ with $L_j^{(m)}$ and the  relative locations of the intersection points of $L_i^{(n)}$ with $W_j'''$

\section{Some calculations in the Fukaya category of the pillowcase}\label{fukayacomp}

We will invoke some aspects of the {\em  Fukaya $A_\infty$-category} of $P^*$.  Precise definitions can be found in \cite{abou},  \cite{Seidel}, and \cite{auroux}.  We next review the definitions of the $A_\infty$-category associated to the pillowcase, to set  notation.
Define the objects of an $A_\infty$-category $\cL$ to be set of curves $\{L_0,L_1\}$ in the pillowcase, for a fixed choice of small $\delta>0$ in Equation (\ref{theL}).

Given a curve $C\in  \{L_0,L_1\}$, denote by 
$C', C'', \dots,C^{(n)}$ its successive Hamiltonian pushoffs as in Theorem \ref{lookslike}, for some small $\epsilon_0>0$.

Given $C,D\in  \{L_0,L_1\}$, the curves $C'$ and $D$ intersect transversally. The morphism space $\Hom_\cL(C,D)$, which we frequently abbreviate to $(C,D)$, is defined  as the $\FF$ vector space spanned by the transverse intersection points of $C'$ with $D$
 
  $$(C,D):=\Hom_{\cL} (C,D)=\FF\langle C'\cap  D \rangle.$$

To define the structure maps 
$$\mu^1:(C,D)\to (C,D),~\mu^2:(C,D)\times (B,C)\to (B,D),$$
etc.,
 first observe that given a positive integer $n$, if the successive Hamiltonian pushoffs are chosen sufficiently closely,  there are {\em  canonical closest point} identifications
 $$ (C,D)=\FF\langle C'\cap  D \rangle \cong \FF\langle C^{(n)}\cap  D^{(n-1)} \rangle \cong \FF\langle C^{(n)}\cap  D \rangle.$$

Given an ordered sequence of $n+1$ objects $C_0,C_1,\dots ,C_n$ in $\Obj(\cL)$,  define the multilinear map 
$$\mu^n_\cL:(C_{{n-1}},C_{{n}})\otimes(C_{{n-2}},C_{{n-1}})\otimes\dots\otimes(C_{0},C_{1})\to (C_{0},C_{n})
$$
 by counting equivalence classes of immersed oriented $(n+1)$-gons   with convex corners in $P^*$ whose edges, ordered counterclockwise,   lie on \\
$C_0^{(n)}, C_1^{(n-1)}, \dots, C_{n-1}', C_n$ as explained next.

For $n\ge 1$, let $G_{n+1}$ be a   $(n+1)$-gon in the plane with convex corners, and counterclockwise labelled   vertices $(v_0,\cdots,v_n)$ and edges $E_0,E_1,\dots, E_{n}$.  
More precisely, take $G_2$ to be a bigon with convex corners at $v_0=1$ and $v_1=-1$,
and let $G_{n+1}\subset \RR^2=\CC$ be the convex $(n+1)$-gon with vertices  the $(n+1)$st roots of unity: $v_0=1, v_1=e^{2\pi \bbi/(n+1)},\dots, v_n=e^{2\pi k \bbi/(n+1)}$ .  
Label  the edges of $G_n$,  counterclockwise starting at $1$, by $E_0,E_1,\dots, E_{n}$.

 Given   generators
 $$x_\ell\in C_{\ell-1}'\cap C_{\ell}=C_{\ell-1}^{(n-\ell+1)}\cap C_\ell^{(n-\ell)},$$ for $(C_{\ell-1},C_\ell,), ~ \ell=1,\dots, n $, the morphism
  $\mu^n_\cL(x_n,x_{n-1},\dots, x_1)\in (C_0,C_n)
 $ 
 counts equivalence classes of orientation preserving immersions of $G_{n+1}$ into $P^*$
 $$\iota:G_{n+1}\to P^*, $$
 satisfying 
  $$ \iota(v_\ell)=x_\ell \text{ and } \iota(E_\ell)\subset C_\ell^{(n-\ell)}  \text{ for  } \ell=1,\dots n.$$

     Two such   immersions, $ \iota_0,\iota_1:G_{n+1}\to P^*$  are called equivalent if $\iota_1=\iota_0\circ h$ for some diffeomorphism $h:G_{n+1}\to G_{n+1}$ which fixes the vertices.  Write $[\iota]$ for the equivalence class of $\iota$.

Then, 
\begin{equation}\label{stmps} \mu^n_\cL(x_n,x_{n-1},\dots, x_1)=\sum_{[\iota]} \iota(v_0)\in \FF\langle C_0^{(n)}\cap C_n\rangle=(C_0, C_n).\end{equation}

The resulting $A_\infty $ category is denoted $\cL$.  It has two objects, $L_0$ and $L_1$, and is equipped with the structure maps $\mu^n_\cL$ defined by Equation (\ref{stmps}). In this article we refer to $\cL$ as {\em the Fukaya category of the pillowcase,} in spite of the fact that we mean the  (full) Fukaya subcategory with objects $L_0$ and $L_1$.  We   endow $\cL$ with a $\ZZ$-grading, in the sense of \cite{Seidel2}, in    Section \ref{sec:gradings}.

The curve  $W_0$ (and similarly $W_1,W_\times$, or, more generally, any unobstructed immersed curve $W$ whose boundary points map to the corners of $P$)      gives rise to a  left    $A_\infty$-module $\cM_{W_0}$ over $\cL$.  For $C\in \{L_0,L_1\}$, one sets $\cM_{W_0}(C)$ equal to  the  $\FF$ vector space  spanned by the intersection points of   $W_0'$    with $C$, and structure maps obtained by counting immersed polygons,  whose first edge  lies on $W_0$. 
To streamline notation, when no chance of confusion is possible we will write
 $(W_0,C )$ rather than $\cM_{W_0}(C)$.  Hence for $C\in \{L_0,L_1\}$, taking closest points gives identifications
 $$\cM_{W_0}(C)=(W_0,C)=   \FF\langle W_0'\cap  C\rangle\cong \FF\langle W_0^{(n)}\cap  C^{(n-1)} \rangle\cong \FF\langle W_0^{(n)}\cap  C\rangle.$$

\noindent{\bf Remark.} A more unifying approach might be to work in the larger {\em wrapped} or {\em partially wrapped}  Fukaya category with objects $L_0,L_1, W_0$ and $W_1$ (see \cite{AAEKO}).  But for this article it suffices to consider the technically more elementary perspective of viewing $W_0$ and $W_1$ as giving modules over the Fukaya category with objects $L_0,L_1$. This avoids having to to wrap near the corners.

\medskip

   The generators of $(L_i,L_j)$ are illustrated and labelled in  Figure \ref{Generatorsfig}. In particular (writing $\langle x_i\rangle$ for the $\FF$ vector space spanned by $x_i$),
\begin{equation}\label{fukgens}\begin{split} (L_0,L_0)&=\langle a_0,b_0,c_0,d_0\rangle,\\ (L_1,L_1)&=\langle a_1,b_1,c_1,d_1\rangle,\\ (L_1,L_0)&=\langle p_{01},q_{01}\rangle,\text{ and }\\ 
(L_0,L_1)&=\langle p_{10},q_{10}\rangle, 
\end{split}\end{equation}
 
   \begin{figure} 
\begin{center}
\def\svgwidth{2.2in}
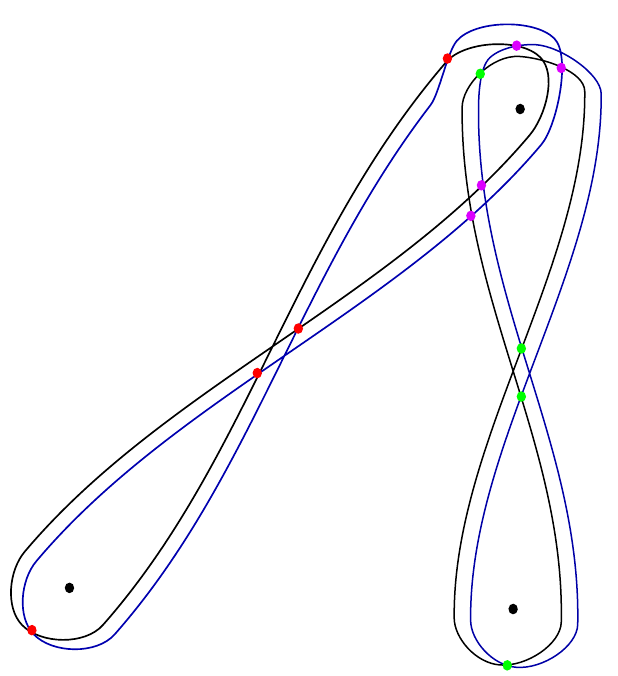
 \caption{\label{Generatorsfig} }
\end{center}
\end{figure} 
 \medskip

  The generators of $(W_i,L_j)$, $i,j\in \{0,1\}$,  are illustrated and labelled in  Figure \ref{rectanglestestfig}. In particular,  \begin{equation}\label{fukgens2}(W_0,L_0)=\langle \alpha, \beta\rangle, ~ (W_0,L_1)=\langle \gamma\rangle, ~(W_1,L_0)=\langle \tau \rangle, ~(W_1, L_1)=\langle \rho, \sigma \rangle.\end{equation}
   \begin{figure} 
\begin{center}
\def\svgwidth{2in}
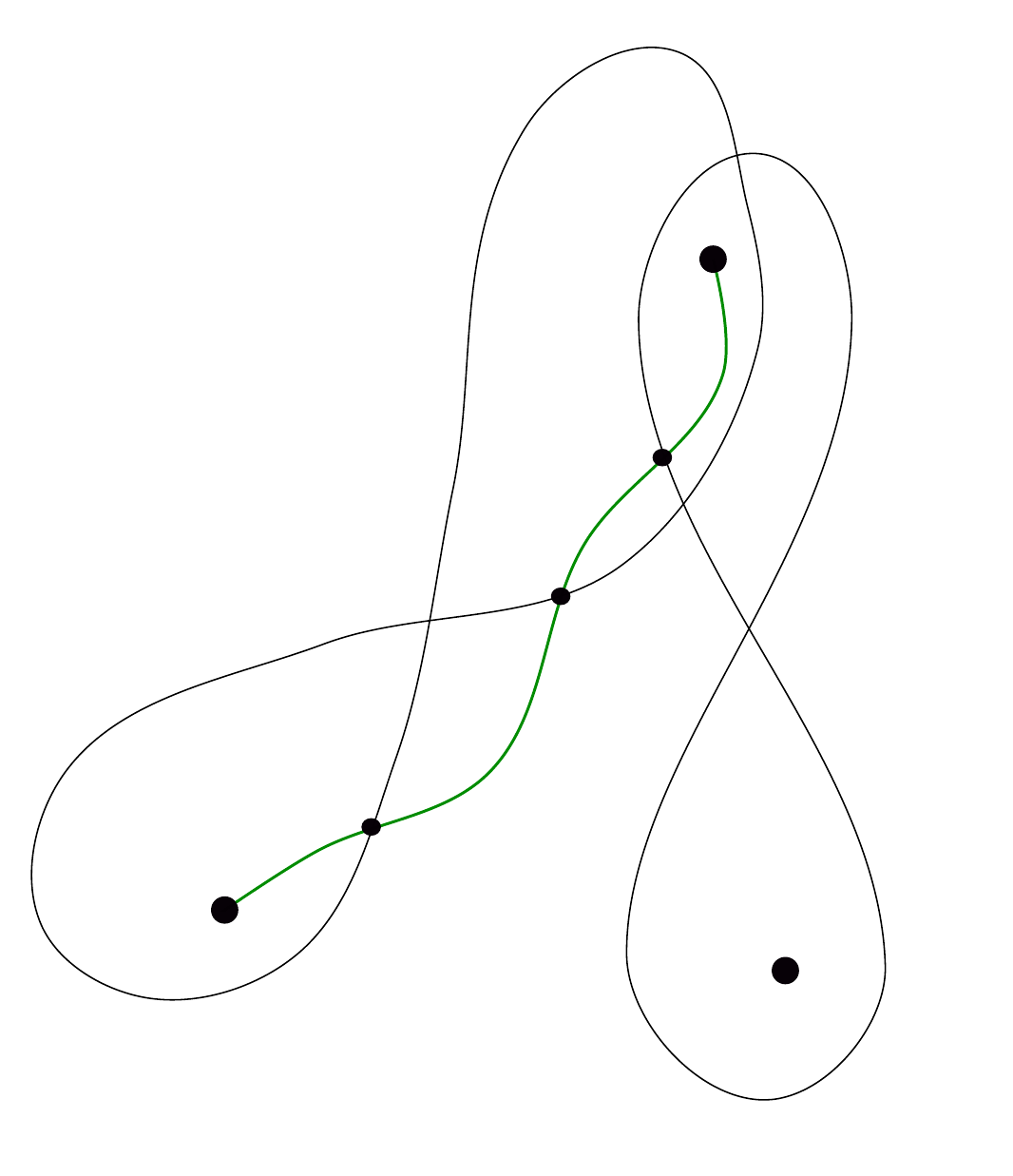
 \caption{\label{rectanglestestfig} }
\end{center}
\end{figure}

   We  turn now to the calculation of the structure maps.  
Our first calculation shows that the structure of the $A_\infty$-category    $\cL$ is relatively simple.

\begin{theorem}\label{maincalc}
The  $A_\infty$-category $\cL$ is associative and strictly unital, with identity elements $a_0\in (L_0,L_0)$ and $a_1\in (L_1,L_1)$, i.e.,  $\mu^2(a_i, x)=x$ and $\mu^2(y,a_i)=y$ whenever these products are defined.  The only other nonzero values of $\mu^2$ are: 
\begin{align*}
\mu^2(b_0,c_0)=\mu^2(c_0,b_0)=d_0,~\mu^2(b_1,c_1)=\mu^2(c_1,b_1)=d_1,\\
\mu^2(b_0,p_{01})=\mu^2(p_{01}, b_1 )=q_{01}, ~ ~\mu^2(q_{10}, b_0 )=\mu^2(b_1,q_{10})=p_{10},\\
\mu^2(p_{01},p_{10})=\mu^2(q_{01},q_{10})=d_0,~\mu^2(q_{10},q_{01})=\mu^2(p_{10},p_{01})=d_1,\\
\mu^2(p_{01},q_{10})=c_0,~\mu^2(q_{10},p_{01})=c_1.
\end{align*}

For $k=3$ we have twenty   four  non-zero products
\begin{align*}\mu^3(q_{10},b_0,p_{01})=a_1,~\mu^3(p_{01}, q_{10}, b_0)=a_0, \\
  \mu^3 (q_{01}, q_{10}, b_0)=\mu^3(p_{01},p_{10},b_0)=b_0 \\
~\mu^3(p_{10},p_{01}, b_1)= \mu^3(b_1,q_{10},q_{01})=b_1, \\
\mu^3(c_0,b_0,c_0)= \mu^3(c_0,q_{01},q_{10})= \mu^3(c_0,p_{01},p_{10})=c_0,\\
\mu^3(c_1,b_1,c_1)= \mu^3(c_1,p_{10},p_{01})=\mu^3(q_{10},q_{01}, c_1)=c_1,\\
\mu^3(d_0,c_0,b_0)=\mu^3(b_0,c_0,d_0)=\mu^3(q_{01},q_{10},d_0)= \mu^3(p_{01},p_{10},d_0)=d_0,\\
\mu^3(b_1,c_1,d_1)=\mu^3(d_1,c_1,b_1)=\mu^3(d_1,q_{10},q_{01})=\mu^3(p_{10}, p_{01}, d_1)=d_1,\\
~ \mu^3(p_{01},p_{10},q_{01})=q_{01},~ \mu^3(p_{10},p_{01},p_{10})=p_{10},\\
\mu^3(q_{10},q_{01},q_{10})=\mu^3(q_{10},p_{01},p_{10})=q_{10}.
\end{align*}
On all other  ordered triples of generators, $\mu^3$ vanishes.  Furthermore, $\mu^k$ is zero if $k\ne 2,3$.
\end{theorem}
 
 \medskip

We now shift our attention to  the $A_\infty$-modules $\cM_{W_0}$ and $\cM_{W_1}$.  Recall that $\cM_{W_i}(L_j)=(W_i,L_j)$,  and these vector spaces have bases identified in Figure \ref{rectanglestestfig}.   The structure maps for these modules,
\begin{equation}\label{test}
\mu^n:
(L_{\ell_{n-1}}, L_{\ell_{n}})\otimes (L_{\ell_{n-1}}, L_{\ell_{n-2}})
 \otimes\dots\otimes(L_{\ell_1}, L_{\ell_2})\otimes \cM_{W_i}(L_{\ell_1})\to \cM_{W_i}(L_{\ell_n}).
\end{equation}
are described in Theorem \ref{testthm}.

\begin{theorem} \label{testthm}  The only non-zero $\mu^n$  in Equation (\ref{test}) are given by
\begin{align*}
 \mu^2(a_0,\alpha)=\alpha,  ~\mu^2(a_0,\beta)=\beta,~\mu^2(c_0,\alpha)=\beta,~\mu^2(a_1,\gamma)=&\gamma,\\
  \mu^2(q_{10},\alpha)=\gamma, \mu^2(p_{01},\gamma)=&\beta  \\   
 \mu^2(a_1,\rho)=\rho,  ~\mu^2(a_1,\sigma)=\sigma,~\mu^2(c_1,\rho)=\sigma,~\mu^2(a_0,\tau)=&\tau,\\
   \mu^2(q_{10},\tau)=\sigma, \mu^2(p_{01},\rho) =&\tau, \\  
\mu^3(b_0,p_{01},\gamma)=\alpha, ~~\mu^3(q_{01},q_{10},\alpha)=\alpha, \mu^3(c_0,b_0,\beta)=\beta , ~ \mu^3(p_{01},p_{10},\alpha)=&\alpha,\\
 \mu^3(b_1,q_{10},\tau)=\rho, 
  ~\mu^3(q_{10},q_{01},\sigma)=\sigma,  
 \mu^3(c_1,b_1,\sigma)=\sigma , ~ 
   \mu^3(p_{10},p_{01},\rho)=&\rho, \\     \mu^3(p_{01},p_{10},\tau)=&\tau. \\
\end{align*}

\end{theorem}
The process of verifying Theorems \ref{maincalc} and \ref{testthm} is described in Appendix \ref{APa}.

\section{Gradings}\label{sec:gradings}

We outline, in a simple form sufficient for our purposes, the notion of gradings (\cite{Seidel2, kontsevich}) 
 appropriate for the  smooth locus $P^*$ of the  pillowcase.  
 \medskip
 
 Let $\PP(TP^*)\to P^*$ denote the bundle whose fiber over $x\in P^*$ is the space of lines in $T_xP^*$. The tangent line to an immersed curve $\iota:C\to P^*$ defines a  lift 
$d\iota:C\to \PP(TP^*)$.
 
 The tangent bundle of $P^*$ is trivial, as $P^*$ is a 4-punctured 2-sphere.   Thus $\PP(TP^*)\to P^*$ admits  sections, called {\em line fields on $P^*$}.    A line field $\lambda$ on $P^*$ defines a trivialization 
 $$\PP(TP^*)\cong P^*\times\RR P^1$$
which takes a line in $T_pP^*$ to its angle with $\lambda_p$,  where the angle is taken in $\RR P^1=[0,\pi]/_{0\sim \pi}$.   
Denote the composite
 $$\tau_\lambda: \PP(TP^*)\cong P^*\times\RR P^1\xrightarrow{\pi_2}\RR P^1.$$
 
 Given a non-negative integer $N$, 
Let 
$$\rho: {(\RR P ^1) _N}\to \RR P^1$$
denote the connected  $ \ZZ/N $ cover.

  An immersed curve $\iota:C\to P^*$ is called {\em $(\ZZ/N)$-gradable with respect to $\lambda$} if the composite of its tangent map  $d\iota :C\to \PP(TP^*)$ and $\tau_\lambda$ lifts to the $\ZZ/N$  cover:
\[\begin{diagram}
\node[2]{(\RR P ^1)_N}\arrow{s,r}{\rho}\\
\node{C} \arrow{e,b}{\tau_\lambda\circ d\iota}\arrow{ne,l,..}{\tilde \iota}\node{\RR P^1}
\end{diagram}\]  
 A {\em $(\ZZ/N)$-grading of 
$\iota:C\to P^*$} is a choice $\tilde \iota:C\to (\RR P ^1) _N$ of such a lift.  
 
If   $C$ is an interval, then $ \iota:C\to P^* $ is obviously gradable. If $C$ is a circle, then $ \iota:C\to P^* $ is gradable if and only if the degree of the composite $\tau_\lambda\circ d\iota$ is zero mod $N$. (This degree is  known as the {\em Maslov index of the curve $\iota$ with respect to $\lambda$}.)

\medskip

 If $\iota_j:C_j\to P^*,~j=1,2$   are transverse immersed curves equipped with gradings $\tilde \iota_j:C_j\to (\RR P ^1)_N$, then to every generator of $(C_1,C_2)$, that is,  to an intersection point $x$ of $\iota_1$ and $\iota_2$,  one can assign the 
grading in $\ZZ/N$, denoted by $p_\cL(x)$, as follows.

Suppose that $\iota_1(a)=x=\iota_2(b)$.  Choose a path $\alpha:I\to (\RR P ^1)_N$ from $\tilde\iota_2 (b)$ to $ \tilde\iota_1 (a)$.  We define the grading of $x$ to be the mod N oriented intersection number 
 of the path  $\rho\circ \alpha$ from $\rho(\tilde\iota_2(b))$ to $\rho(\tilde\iota_1(a))$ in $ \PP(TP^*)_x$ with a small positive rotation of  $\rho(\tilde\iota_2(b))$, 
 $$p_\cL(x) = (\rho\circ \alpha) \cdot e^{\epsilon \bbi} \rho(\tilde\iota_2(b))  $$  for any sufficiently small $\epsilon>0$.  If we wish to emphasize the dependence on the line field $\lambda$, we will write $p_\cL^\lambda$ instead of $p_\cL$.

\medskip

  Let $\lambda_0$ be the line field on $P^*$ whose preimage under the branched cover $\RR^2\to P$ is the horizontal line field.  If $\iota:S^1\to P^*$ is a small embedded loop encircling one of the four corners counterclockwise, the Maslov index of  $\iota$ with respect to $\lambda_0$ equals 1.

 If $\lambda$ is any other line field, the difference in Maslov indices well-defines a cohomology class  $\delta_{(\lambda,\lambda_0)}\in H^1(P^*; \ZZ)$.
Denote the homology class of the counterclockwise loop around the point $[0,0]$ in $P^*$ by $\xi_1$, and similarly $\xi_2$ for the point $[\pi, 0]$, $\xi_3$ for $[\pi,\pi],$ and $\xi_4$ for $[0,\pi].$
 
Since $\delta_{(\lambda,\lambda_0)}(\xi_1+\xi_2+\xi_3+\xi_4)=0$,  the sum of the Maslov indices of counterclockwise loops around all four corners with respect to $\lambda$ is again equal to 4.
 
 In particular,  the immersed circles $L_0$ and $L_1$ are $(\ZZ/N)$-gradable with respect to  $\lambda$ if and only if  $\delta_{(\lambda,\lambda_0)}(\xi_1-\xi_3)=0$ and $\delta_{(\lambda,\lambda_0)}(\xi_2-\xi_3)=0$, so that
 \begin{equation}
\label{gradability}
\delta_{(\lambda,\lambda_0)}(\xi_1)=\delta_{(\lambda,\lambda_0)}(\xi_2)=\delta_{(\lambda,\lambda_0)}(\xi_3)\mod N.\end{equation}
This implies that $\delta_{(\lambda,\lambda_0)}(\xi_4)=-3\delta_{(\lambda,\lambda_0)}(\xi_1) \mod{N}.$

 \medskip

If $x\in C_0\cap C_1$ for some curves $C_0,C_1$, then $x$ can be viewed as a generator of both $(C_0,C_1)$ and $(C_1,C_0)$, and $x$ may have different gradings in these two choices. 
In the following Lemmas, $a_i, b_i, c_i, d_i, p_{ij}, q_{ij},\alpha, \beta,\gamma,\tau,\rho,\sigma$  refer to the generators of the different morphism spaces as in Equations (\ref{fukgens}) and (\ref{fukgens2}).  For example,    we consider $\alpha$ as a generator of $ (W_0,L_0),$   {\em not} $(L_0,W_0)$.

 \begin{lemma}\label{pinning_delta} Fix any non-negative integer $N$.  Fix a line field $\lambda$ on $P^*$ such that Equation {\rm (\ref{gradability})} holds.   Define $\delta=\delta_{(\lambda,\lambda_0)}(\xi_3)$.
 \begin{enumerate} 
 \item The Lagrangians $L_0, L_1, W_0,$ and $W_1$ are gradable with respect to $\lambda$.
   \item  Consider any gradings of $L_0, L_1, W_0$ and $W_1$.  Then  $$p^\lambda_\cL (a_i)=0\text{ and }   p^\lambda_\cL 
   (d_i) = 1.$$  In addition,  the gradings satisfy
$$p_\cL^\lambda(\alpha)-p_\cL^\lambda(\beta)=1+\delta=p_\cL^\lambda(\rho)-p_\cL^\lambda(\sigma).
$$
\end{enumerate} \end{lemma}

\begin{proof} The arcs $W_0$ and $ W_1$ are gradable with respect to any line field. The sentence preceding Equation (\ref{gradability})  shows that $L_0,L_1$ are gradable.

 The calculation $\mu^2(a_i,a_i)=a_i$, together with the fact that $\mu^k$ has degree $2-k$, implies that $p_\cL^\lambda(a_i)=0$. The existence of a bigon from $a_i$ to $d_i$ implies that 
$p_\cL^\lambda(d_0)=1=p_\cL^\lambda(d_1)$.

The following relative gradings    can be calculated directly: \begin{equation*}
p_\cL^\lambda(\beta)=p_\cL^\lambda(\alpha) -1-\delta\text{ and } 
p_\cL^\lambda(\sigma)=p_\cL^\lambda(\rho) -1-\delta.
\end{equation*} 
\end{proof}

 Rather than to postpone the assignment of a specific choice  of line field $\lambda$ and gradings of $L_0, L_1, W_0$ and $W_1$ until required in the proof of our main result, we  choose $\lambda$ and $N$ presently, in order to prevent notation from becoming unwieldy. 

Specifically, {\em we henceforth take $N=0$, that is, we work  with $\ZZ$-gradings,  and, we choose $\delta=1$.} This determines a unique up to isotopy line field $\lambda$ satisfying\footnote{The line field $\lambda$ coincides with the line field called $\lambda_{inst}$ in \cite{HHK2}, which arises as a consequence of the  $\ZZ/4$-grading on Kronheimer-Mrowka's singular instanton homology \cite{KM1} and a splitting theorem for spectral flow.}   \begin{equation}
\label{defoflam}
\delta_{(\lambda,\lambda_0)}(\xi_i)=1\text{ for }i=1,2,3\text{ and }\delta_{(\lambda,\lambda_0)}(\xi_4)=-3.
\end{equation}
These choices are motivated as follows.  In Section \ref{TwistedSister} we construct a twisted complex over $\cL$ from any tangle diagram which, when paired with  $W_0$ or $W_1$, gives a chain complex, which Theorem \ref{invariance2} identifies with the reduced Khovanov complex of the link obtained by closing the tangle with one of the two trivial tangles.  In the special case of the 2-tangle diagram $T_0$, the twisted complex constructed is simply $L_0$, and the Khovanov complex  corresponding to the planar closure with two components is $(W_0,L_0)=\langle \alpha, \beta\rangle$, with zero differential.
Because the reduced Khovanov cohomology for the trivial 2-component link has two generators, one in quantum grading 0 and one in grading $-2$, 
we require $p_\cL^\lambda(\alpha)-p_\cL^\lambda(\beta)=\pm2$, and  therefore $\delta$ must equal $1$ or $-3$ by Lemma \ref{pinning_delta}.  These two choices give isomorphic graded $A_\infty$-categories (the isomorphism corresponding to moving the earring). Hence we will simply choose $\delta =1$.  

Similar considerations, corresponding to   unknot   closures of $T_ 0$ and $T_1$, and the fact that the reduced Khovanov cohomology of the unknot has rank one, concentrated  in grading $-1$,    requires   $p_\cL ^\lambda (\gamma) =p_\cL ^\lambda (\tau)=-1$. Such considerations lead us to take specific choices of gradings on $L_1,W_0,$ and $W_1$, and are made precise in 
the following lemma.

  \begin{proposition}\label{prop:gradings} Let $\lambda$ be the line field on $P^*$ satisfying Equation {\rm (\ref{defoflam})}.
  For any $\lambda $ grading of $L_0$, there are uniquely determined $\lambda$ gradings of $ L_1, W_0$ and $W_1$ such that  
   $$p_\cL ^\lambda (\gamma) =p_\cL ^\lambda (\tau)=-1,\text{ and }p_\cL ^\lambda (\alpha) =  0   .$$  
 
 For these gradings of $ L_1, W_0$, and $W_1$, the gradings of the other generators are given by 
\begin{align*} 
p_\cL^\lambda(a_i)&=0,  ~~p^\lambda_\cL (b_i) =3, ~~p^\lambda_\cL (c_i) = -2,~~ p_\cL^\lambda(d_i)=1,\\
~~p^\lambda_\cL(p_{01}) &=  - 1, ~~
p^\lambda_\cL (q_{01}) = 2, ~~
p^\lambda_\cL (p_{10} )= 2,  ~~p^\lambda_\cL (q_{10} ) = -1,\\
p_\cL ^\lambda (\beta) &= -2, ~~p_\cL^\lambda(\sigma) = -2,~~
p_\cL^\lambda(\rho)=  0.
\end{align*} 
\end{proposition}

\begin{proof} Lemma \ref{pinning_delta} with $\delta=1$ and $N=0$  shows that $L_0,L_1, W_0, W_1$ are gradable with respect to $\lambda$, and that $p_\cL(a_i)=0$ and $p_\cL(d_i)=1$.  Given a grading of $L_0$,  $W_1$ can be graded so that $\tau\in (W_1,L_0)$   has any desired grading. Choose a grading of $W_1$ so that $p_\cL^\lambda(\tau)=-1$.   Similarly, choose a grading of $W_0$ so that $\alpha\in (W_0,L_0)$ satisfies $p_\cL^\lambda(\alpha)=0$.  Lemma 
\ref{pinning_delta} then implies that $p_\cL^\lambda(\beta)=-2$.  Finally, choose a grading of $L_1$ so that $\gamma\in (W_0, L_1)$ satisfies $p_\cL^\lambda(\gamma)=-1$.

The other calculations of Theorems \ref{maincalc} and \ref{testthm} give further linear 
relations on the gradings, since  $\mu^k$ has degree $2-k$.
For example, $\mu^3(q_{10}, b_0, p_{01})=a_1$ 
implies that $p_\cL^\lambda(q_{10})+p_\cL^\lambda(b_0)+p_\cL^\lambda(p_{01})= p_\cL^\lambda(a_1)+1=1.$
The remaining formulae are the result of solving similar equations from Theorems \ref{maincalc} and \ref{testthm}.
\end{proof}

 \section{Tangle categories, following Bar-Natan}\label{BNsect}
 In this section, we describe a category called the {\em dotted cobordism category}, a minor variant of a construction of Bar-Natan.   
\subsection{$2$-stranded tangles and planar tangles}

Fix $2m$ points on the boundary of the unit disk $D^2$, say $B_{2m}=\{e^{\pi \bbi \ell/m} \mid\ell=0, \dots,2m-1\}. $   A {\em planar $m$-tangle} is a properly embedded compact 1-dimensional submanifold of $D^2$ with boundary equal to $B_{2m}$. Thus a planar $m$-tangle has $m$ arc components and some number of circle components.  

A (generic) {\em $m$-tangle} is a properly embedded compact 1-dimensional submanifold of $D^2\times I$ with boundary equal to $B_{2m}\times\{\tfrac 1 2 \}$ such that the projection to $D^2$ has  finitely many transverse double points. In the usual way, the projection of an $m$-tangle $T$ to $D^2$ which has $d$ double points determines $2^d$ planar $m$-tangles by taking complete resolutions of the double points. If the $d$ double points are ordered, each complete resolution can be labelled with a sequence of $0$s and $1$s of length $d$, with $0$ corresponding to the ``overcrossing turns left'' resolution. 

In this article, we focus on  2-tangles.   Figure \ref{tripleballfig} illustrates three 2-tangles $T_\times, T_0,$ and $ T_1$.  The tangles $T_0$ and $T_1$ are planar, and are the two resolutions of the non-planar tangle $T_\times$. 
 Every other planar tangle is isotopic to  the planar tangle obtained from either $T_0$ or $T_1$ by adding circle components in the complementary regions.

    \begin{figure} 
\begin{center}
\def\svgwidth{3.5in}
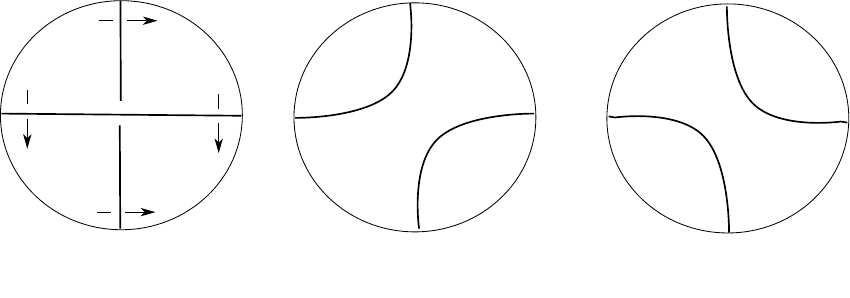
 \caption{\label{tripleballfig} }
\end{center}
\end{figure}  

\begin{definition}
For $\ell=0,1$ and $n=0,1,2,\dots$, let $T_\ell(n),$ denote the  set of all  planar 2-tangles with $n$ circle components and whose two arc components are isotopic to $T_\ell$.  There are several such planar isotopy classes, corresponding to where the circles are placed in complementary regions, but they are all isotopic in $D^2\times I$.   We will say that a planar 2-tangle $P\in T_\ell(n)$ {\em has type $T_\ell(n)$}. When no chance of confusion is possible we will 
abuse notation by letting $T_\ell(n)$ denote any planar 2-tangle in the set $T_\ell(n)$. \end{definition}

Our emphasis in this article is on {\em reduced} Khovanov homology, which is defined by making a choice of base point on a link or tangle.   We place this base point at $1\in B_{2n}\subset \partial D^2$ and call it the {\em earring}.  The terminology comes from \cite{HHK1, HHK2} (see also Appendix \ref{tracelesssect}) and corresponds to the gauge-theoretic construction introduced in \cite{KM1}.  It also serves to distinguish it from the {\em dots} introduced below in Bar-Natan's dotted cobordism category.   Given an $n$-tangle (planar or not) $T$, we call the arc component of $T$ with one end point $1$ the {\em earringed component}, and the the other arcs the {\em unearringed components} of $T$.  For example, the earringed component of $T_0$ is the arc with endpoints $1$ and $-\bbi$.

\subsection{The dotted cobordism category}

We construct, following \cite{BN} very closely, a category $\cD$ of planar 2-tangles.  Bar-Natan defines much more general $n$-tangle categories,    which carry an additional planar algebra structure.    We  only  consider  $2$-tangles in this article.

Construct two  categories, $\widetilde \cD$ and  $\cD$ of planar 2-tangles and dotted cobordisms in the unit disk $D^2$ as follows. 
The objects of  $\widetilde \cD$ and $\cD$ are the   planar 2-tangles  with boundary $B_4=\{1, \bbi, -1, -\bbi\}$.

For any pair $P_0,P_1$ of planar 2-tangles, define  $\widehat{\Hom}(P_0,P_1)$ to be the $\FF$ vector space with basis the set of isotopy classes of embedded oriented cobordisms rel $B_4$ with dots (a choice of finitely many distinguished points on the cobordism)
in $D^2\times[0,1]$ which restrict to $P_0$ and $P_1$ at the ends of the interval. 

For the category $\widetilde \cD$, define  $\Hom_{\widetilde \cD}(P_0,P_1)$ as the quotient of $\widehat{\Hom}(P_0,P_1)$ by the  additional   local relations  in Figure \ref{dottedfig} (see Bar-Natan \cite[Section 11.2]{BN}).    \begin{figure} 
\begin{center}
\def\svgwidth{3.8in}
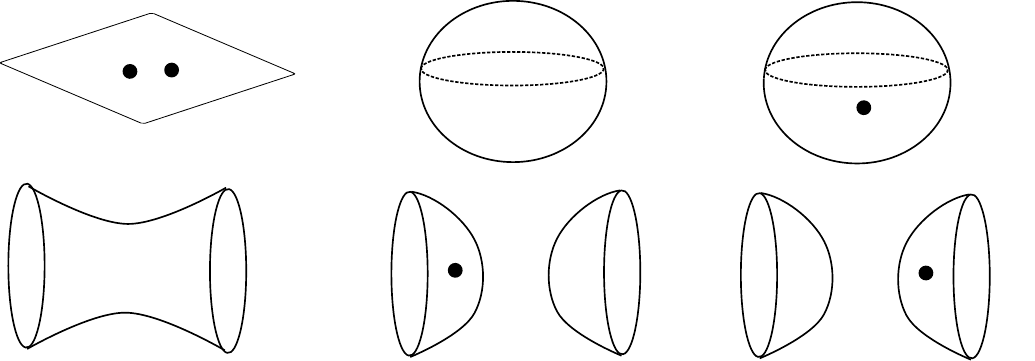
 \caption{\label{dottedfig} }
\end{center}
\end{figure}
The second and third relations in this list mean that, if $C$ is a cobordism and we change it by adding a disjoint 2-sphere with or without a dot, then 
$\mbox{(undotted sphere)} \sqcup C =0$ and $\mbox{(dotted sphere)} \sqcup C = C$.

  Since our focus is strictly on reduced Khovanov homology, we impose one additional relation on the morphisms, which will shorten some calculations and formulas. Thus, for  
 the category $\cD$, define the $\Hom_{\cD}(P_0,P_1)$ as the quotient of $\widehat{\Hom}(P_0,P_1)$ by the  subspace generated by Bar-Natan's  local relations and one additional relation, namely that   
   \begin{align}\label{kill}
  \text{ \it Any dotted cobordism carrying a dot  on the path component }\hskip.3in \\
  \text{\it  containing the earring (ie. passing through } 
 1\in B \text{\it  ) is equivalent to zero.} \notag
  \end{align}

Placing a dot on a component of a cobordism corresponds to ``multiplying by $x$'' on that component.  The reduced Khovanov complex is defined as the kernel of the  ``multiply by $x$ on the earringed component'' morphism (see Section \ref{circlecomp}), justifying (\ref{kill}).

Composition is given by stacking cobordisms, and we extend this linearly. The product cobordism  $P_0\times I$ with no dots represents the identity morphism.

Following Bar Natan, endow  $\Hom_\cD(P_0,P_1)$ with a $\ZZ$-grading by defining the grading of a dotted cobordism $C$ to be 
\begin{equation}
\label{BNgrading}
p_\cD(C)=-2+\chi(C)-2\cdot\text{(number of dots)}
\end{equation}
where $\chi$ denotes the Euler characteristic.
 Since each arc has $\chi =1$, and closed components have $\chi =0$, this is additive with respect to composition of cobordisms. 
 The relations that declare different dotted cobordisms to be equivalent involve cobordisms with the same grading, and hence $\cD$ is a $\ZZ$-graded abelian category.  It will be convenient for us to view $\cD$  as an $A_\infty$-category with $\mu^2$ defined by composition and all other $\mu^n$ zero.

  \medskip

We turn now to a calculation of $\Hom_\cD(P_0,P_1)$ for any two planar tangles $P_0, P_1$. We begin with the simple cases when the $P_\ell$ have no circle components. For $\ell\ne {\ell'}\in \{0,1\}$, let $S_{{\ell'}\ell}\in  \Hom_{\cD}(T_\ell,T_{\ell'})$ be the obvious saddle cobordism from $T_\ell$ to $T_{\ell'}$. 
 \begin{proposition}\label{easy} Let $\ell, \ell'\in \{0,1\}$.
 The vector space $\Hom_{\cD}(T_\ell,T_{\ell'})$ is
 \begin{enumerate}
\item   1-dimensional with basis the 1-saddle cobordism $S_{{\ell'}\ell}$ if $\ell\ne \ell'$.
\item 2-dimensional with basis $\{A_\ell, C_\ell\}$, where $A_\ell$ is the product cobordism $T_\ell\times I$, and    $C_\ell$ is the product cobordism $T_\ell\times I$ with a dot on the unearringed component if $\ell=\ell'$.
\end{enumerate}  
The gradings are given by $$p_\cD(A_\ell)=0, ~p_\cD(C_\ell)=-2, ~p_\cD(S_{\ell'\ell})=-1.$$
The composition operation $\mu_\cD^2$, written multiplicatively, is given  by 
\begin{align*}A_\ell^2=A_\ell,~A_\ell C_\ell=C_\ell=C_\ell A_\ell,~ C_\ell^2=0, \hskip.5in& \\ ~  A_{\ell'}S_{\ell'\ell}=S_{\ell'\ell}=S_{\ell'\ell}A_\ell, ~C_{\ell'}S_{\ell'\ell}=0=S_{\ell'\ell}C_\ell,
~
S_{\ell \ell'}S_{\ell'\ell}=C_{\ell'}. &
\end{align*}
\end{proposition}
\noindent{\em Sketch of proof.} 
 Bar-Natan's relations, illustrated in Figure \ref{dottedfig}, show that $\Hom_{\cD}(T_\ell,T_{\ell'})$ is spanned by a collection of cobordisms, each of which is homeomorphic to a union of disjoint disks.
 
  If $\ell\ne {\ell'}$, then the cobordism is a disk in $D^2\times I$ with connected boundary $T_0\cup (B\times I)\cup T_1$, and by the Sch\"onflies theorem, isotopic to $S_{\ell{\ell'}}$.   Since this is connected, the additional relation  (\ref{kill}) shows that to be non-zero, this cobordism cannot contain any dots.  Thus $S_{\ell{\ell'}}$ generates $\Hom_{\cD}(T_\ell,T_{\ell'})$. 
  
  If $\ell=\ell'$, then a similar argument shows that any cobordism consisting only of disk components is a product.  There are two possibilities, according to whether or not the unearringed component has a dot.
  
Thus the stated cobordisms span the space of morphisms.  We can see  that the equivalence classes of $\{ A_\ell, C_{\ell'}\}$ and $\{S_{\ell{\ell'}}\}$ are linearly independent sets as follows.  This is easy to verify if we consider the two trivial planar tangles outside the disk that close off $T_0$ and $T_1$ to make a two-component planar curve $C$.  Each of these represents a functor from $\cD$ to the analogous dotted cobordism category with one object $C$.   Composing each of these with the reduced Khovanov homology functor to the category of vector spaces, it is straightforward to check that these are linearly independent.     

It is clear that pre- or post-composing with $A_\ell$ does not change the equivalence class of any dotted cobordism.  
The relations  $C_{\ell'}S_{\ell{\ell'}}=0=S_{\ell{\ell'}}C_\ell$ follow from the additional relation that any cobordism with a dot on the earringed component is equivalent to zero. Finally, $S_{\ell{\ell'}}S_{\ell{\ell'}}=C_{\ell'}$ follows from the neck cutting relation (the last relation in Figure \ref{dottedfig}) combined with the additional relation  (\ref{kill}).
\qed 
\medskip

\subsection{Algebras} \label{algebras}
Although not needed for  our results, in this section we  make some observations about the ($\ZZ$-graded) categories $\cD,~ \widetilde \cD$ and $\cL$.  These motivate our use of the category $\cD$ rather than the larger $\widetilde \cD$.

Form a  12 dimensional $\ZZ$-graded $\FF$-algebra $A_\cL$ by taking
\begin{equation}\label{Lalgebra}\begin{split}
A_\cL&=\Hom_\cL(L_0,L_0)\oplus\Hom_\cL(L_1,L_1)\oplus\Hom_\cL(L_0,L_1)\oplus\Hom_\cL(L_1,L_0)\\
&=\langle a_0,b_0,c_0,d_0\rangle\oplus\langle a_1,b_1,c_1,d_1\rangle\oplus\langle p_{10},q_{10}\rangle\oplus\langle p_{01},q_{01}\rangle
\end{split}\end{equation}
with multiplications using the  $\mu^2_\cL$ of Theorem \ref{maincalc} when defined, and set    equal zero   when a pair of morphisms are not composable, and extended linearly. Gradings are given  using Proposition \ref{prop:gradings}.

Similarly, form  the 6-dimensional $\ZZ$-graded   algebra:
\begin{equation*}\label{Dalgebra}\begin{split}
A_\cD&=\Hom_\cD(T_0,T_0)\oplus\Hom_\cD(T_1,T_1)\oplus\Hom_\cD(T_0,T_1)\oplus\Hom_\cD(T_1,T_0)\\
&=\langle A_0,C_0\rangle\oplus\langle A_1,C_1\rangle\oplus\langle S_{10}\rangle\oplus\langle S_{01}\rangle
\end{split}\end{equation*}
with products and $\ZZ$-gradings   given  in  Proposition \ref{easy}.

Thirdly, using Bar Natan's  category $\widetilde \cD$, in which one does not impose the relation (\ref{kill}), form  the 12-dimensional $\ZZ$-graded algebra 
\begin{equation*}\label{Dtilalgebra}\begin{split}
A_{\widetilde \cD}&=\Hom_{\widetilde \cD}(T_0,T_0)\oplus\Hom_{\widetilde \cD}(T_1,T_1)\oplus\Hom_{\widetilde \cD}(T_0,T_1)\oplus\Hom_{\widetilde \cD}(T_1,T_0)\\
&=\langle A_0,B_0,C_0,D_0\rangle\oplus\langle A_1,B_1,C_1,D_1 \rangle\oplus\langle S_{10},S_{10}'  \rangle\oplus\langle S_{01}, S_{01}'\rangle
\end{split}\end{equation*}
where $B_i$ is the product cobordism from $T_i$ to itself which has a dot on the earringed component, $D_i$ is the product cobordism from $T_i$ to itself which has a dot on both components, and $S_{ij}'$ are the saddle cobordisms which contain a dot.  (A similar argument to that of Proposition \ref{easy} shows that these cobordisms form a basis of  $\cA_{\widetilde \cD}$.) Then   $A_i$ has grading $0$, $B_i,~C_i$ have grading $-2$, $D_i$ has grading $-4$, $S_{ij}$ has grading $-1$, and $S_{ij}'$ has grading $-3$.

\medskip

Theorem \ref{maincalc} and Proposition \ref{easy}, and parity considerations can be used to verify the following.  

\begin{proposition}\label{Feasy} \hfill
\begin{enumerate}
\item The 
assignment
 $$A_i\mapsto a_i, ~C_i\mapsto c_i, ~S_{10}\mapsto q_{10}, ~S_{01}\mapsto p_{01}$$
 defines  an injective, $\ZZ$-grading preserving, algebra homomorphism
$$A_\cD\to A_\cL.$$
\item  The assignment 
\begin{align*}A_i\mapsto a_i, ~B_i\mapsto b_i, ~C_i\mapsto b_i+c_i, ~D_i\mapsto d_i, \\
 ~S_{10}\mapsto q_{10}, ~S_{01}\mapsto p_{01},  ~S_{10}'\mapsto p_{10}, ~S_{01}'\mapsto q_{01},\end{align*}
 defines an algebra isomorphism 
 $$A_{\widetilde \cD}\to A_\cL.$$
 This isomorphism is {\em not}  grading preserving.  In fact  $A_{\widetilde \cD}$ and $A_\cL $ are not isomorphic as graded algebras.  Moreover,
 there does not exist {\em any}  line field  on $P^*$ in which $L_0, L_1$ are $\ZZ/2$-graded,  such that   $A_\cL$ and $A_{\widetilde \cD}$  are $\ZZ/2$-graded isomorphic.  
 \end{enumerate}

\qed 
\end{proposition}

\subsection{Tangles with circle components}\label{circlecomp}
We extend the construction to include planar tangles which contain closed circle components, as these arise in complete resolutions of a projection of a 2-tangle.

\medskip

 Let  $A=\{0,1,x,x+1\}$, which we think of  as a two dimensional graded vector space with basis $\{1,x\}$, graded  by $p_A(1)=1$, $p_A(x)=-1$.   It is usual to write $A=\FF[x]/(x^2=0)$, so that 
multiplication by $x$ defines a linear map $x:A\to A$ with kernel the 1-dimensional graded vector subspace
  \begin{equation}
\label{defnofR}
R:=\langle x\rangle\subset A=\langle 1, x\rangle.
\end{equation}

  Define four linear maps
 \begin{equation}\label{epeta}
 \begin{split}
 &\epsilon:A\rightarrow \FF, \ \epsilon(x)=1,\ \epsilon(1)=0,\\
 & \dot{\epsilon}:A\rightarrow \FF,\ \dot{\epsilon}(x)=0,\ \dot{\epsilon}(1)=1,\\
&\eta:\FF\rightarrow A,\ \eta(1)=1,\\
&\dot{\eta}:\FF\to A,\ \dot{\eta}(1)=x.
 \end{split}
 \end{equation}
 
   Viewing  $\FF$ as a graded vector space lying entirely in grading 0, we compute
 \begin{equation}\label{mapgr}
 p_A(\ep)=1,\ p_A(\dot{\epsilon})=-1, \ p_A(\eta)=1, \ p_A(\dot{\eta})=-1\end{equation}

 \begin{definition}\label{defpfB2} 
 Given $n,m\geq 0$, define a basis $\cB_{n,m}$ for 
 $$(A^*)^{\otimes n}\otimes A^{\otimes m}= \Hom_\FF(A^{\otimes n},A^{\otimes m}),$$ where $V^{\otimes 0}$ denotes $\FF$, as follows.  Fix bases  for $\{\ep,\dot\ep\}$ for $A^*=\Hom_\FF(A,\FF)$ and $\{\eta, \dot\eta\}$ for $A$, and then take the basis $\cB_{n,m}$ to be appropriate monomials in   the higher tensor products, e.g.  $(\ep\otimes\dot\ep)\otimes (\eta\otimes\dot\eta\otimes\eta)\in\cB_{2,3}$.  The basis  $\cB_{n,m}$ contains $2^{n+m}$ elements.  Equation (\ref {mapgr}) extends to tensor products to define $p_A({\bf b})\in \ZZ$ for every ${\bf b}\in \cB_{n,m}$ in the usual way.
\end{definition}

 Recall that $T_\ell(n), \ell=0,1$ denotes any planar tangle whose two arc components are isotopic to $T_\ell$, as in Figure \ref{tripleballfig}, and which has $n$ closed (i.e., circle) components.    
 Any cobordism $C$ from $T_\ell$ to $T_{\ell'}$ can be extended to a cobordism from (any) $T_\ell(n)$ to $T_{\ell'}(m)$,   by capping off  the $n$ circles in $T_\ell(n)$  with disks, called {\em caps}, eliminating innermost circles first  in any nested families, and then  introducing $m$ disks, called {\em cups}, bounded by the $m$ circles in $T_{\ell'}(m)$. The resulting cobordism, $C'$, is unique up to isotopy. There are $2^{m+n}$ corresponding {\em dotted} cobordisms, obtained by placing zero or one dot on each of these $m+n$ disk components of $C'$.   Ordering the $n$ circles in $T_\ell(n)$  and the $m$ circles in $T_{\ell'}(m)$ sets up an identification between the dotted cobordisms   obtained from $C'$ by placing dots on the cups and caps, and the basis $\cB_{n,m}$ of $(A^*)^{\otimes n}\otimes A^{\otimes m}$.   We denote the dotted cobordism obtained from $C'$ by placing dots on the cups and caps according to this scheme by $C_{\bf b}$.

 The following extension of Proposition \ref{easy} is   immediate.
 
  \begin{proposition}\label{easy2} A basis of  $\Hom_{\cD}(T_\ell(n), T_{\ell'}(m))$ is given by  $\{S_{\ell'\ell,{\bf b} }\}_{{\bf b}\in \cB_{n,m}}$ when ${\ell}\ne {\ell'}$, and by 
  $\{A_{{\ell},\bf b},C_{{\ell},\bf b}\}_{{\bf b}\in \cB_{n,m}}$ when ${\ell}={\ell'}$.
  This sets up an identification:   
 $$ \Hom_{\cD}(T_{\ell}(n), T_{\ell'}(m))= (A^*)^{\otimes n}\otimes A^{\otimes m}\otimes \Hom_\cD(T_{\ell},T_{\ell'}).$$
 \end{proposition}

\section{A pair of $A_\infty$-functors $\cF:\cD\to \Sigma\cL$ and $\cG:\cL\to Ch$}\label{functors}
The traceless character variety $R^\nat_\pi(D^2\times I, T_\ell(m))$  is identified in Theorem \ref{fukayadudes2} with $\{\pm 1\}^m \times L_\ell$. Each path component is a circle which immerses in the pillowcase with one double point, illustrated in Figures \ref{exacttriplesfig}, and  illustrated in the flattened punctured pillowcase in Figure \ref{Generatorsfig}. We think of this collection of $2^m$ objects in $\cL$ as the  object 
$$A^{\otimes m}\otimes L_\ell\in {\rm Obj}(\Sigma\cL),$$ 
with $A=\FF[x]/\langle x^2\rangle$, as before. This assignment  re-expresses  Equation (\ref{tracelessfunctor}) as: 
 \begin{equation} 
\label{cF}\cF:\Obj(\cD)\to \Obj(\Sigma \cL), ~  
\cF(T_\ell(m))= A^{\otimes m} \otimes L_\ell.
\end{equation}

We extend this to a strictly unital $A_\infty$-functor  as follows. 
To express 
\begin{align*}\cF^1:\Hom_{\cD}(T_\ell(m),T_{\ell'}(n))\to&\Hom_{\Sigma\cL}(\cF(T_\ell(n)), \cF(T_{\ell'}(n)))\\
&= \Hom_{\Sigma\cL}(A^{\otimes m}\otimes L_\ell,A^{\otimes n}\otimes L_{\ell'})\\
&=\Hom_{\FF}(A^{\otimes m},A^{\otimes n})\otimes  \Hom_{\cL}(  L_\ell,  L_{\ell'})\\
&=(A^*)^{\otimes m}\otimes A^{\otimes n}\otimes   \Hom_{\cL}(L_\ell,L_{\ell'}),
\end{align*}
define
\begin{multline}
\label{F1}
\hskip.4in \cF^1:\Hom_{\cD}(T_\ell(n),T_{\ell'}(m))\to (A^*)^{\otimes m}\otimes A^{\otimes n}\otimes   \Hom_{\cL}(L_\ell,L_{\ell'}),\\
\cF^1(A_{i,\bf b})={\bf b}\otimes a_i,~ \cF^1(C_{i,\bf b})={\bf b}\otimes c_i, ~\cF^1(S_{10,\bf b})= {\bf b}\otimes q_{10}, ~\cF^1(S_{01,\bf b})= {\bf b}\otimes p_{01} \hskip0in
\end{multline}
 and then extend linearly.  Finally, define 
 \begin{equation}
\label{Fn}
\cF^n=0 \text{ for }n>1\end{equation}

 \begin{theorem} The assignments $\cF, \cF^n$ of Equations (\ref{cF}),  (\ref{F1}), and (\ref{Fn})  define a strictly unital $A_\infty$-functor
 $$\cF:\cD\to \Sigma\cL$$   for which $\cF^1$ is injective on morphism spaces.  \end{theorem}

\begin{proof} Referring to Proposition \ref{easy}  and Equation (\ref{fukgens}),   one sees that when $i=j$, $\Hom_\cD(T_i,T_i)$ is 2-dimensional, spanned by 
$\{ A_i, C_i\}$, and $\Hom_{\cL}(L_i,L_i)$ is 4-dimensional, spanned by $\{a_i,b_i,c_i,d_i\}$. When $i\ne j$, $\Hom_\cD(T_i,T_j)$ is 1-dimensional, spanned by 
$S_{ji}$, and $\Hom_{\cL}(L_i,L_j)$ is 2-dimensional, spanned by $\{p_{ji}, q_{ji}\}$.   It follows that $\cF^1$ is injective.  The map $\cF^1$ preserves the gradings   listed in Proposition \ref{easy},  Definition \ref{defpfB2},   and Proposition \ref{prop:gradings}.

Given that $\cF^n=0$ for $n\geq 2$, $\mu^r _{\cD} = 0$ for $r\neq 2$,  and $\mu^r _{\cL } = 0 $ except when $r=2,3$, the polynomial equations condition to verify that  $\cF$ is an $A_\infty$-functor, 
\begin{eqnarray*} 
\lefteqn{\sum_r \sum_{ s_1+\dots +s_r=d
} \mu^r _{\cL } \left( \cF^r(x_d, \dots, x_{d-s_r +1}), \dots, \cF^1 (x_{s_1}, \dots, x_1 ) \right) }\\
&=& \sum_{0\leq n <d, ~~1\leq m <d} \cF^{d-m+1}(x_d, \dots, x_{n+m+1}, \mu^m _{\cD} (x_{n+m}, \dots, x_{n+1}), x_n , \dots, x_1 ), \end{eqnarray*} 
 are trivial except for the $d=2,3$ cases, which reduce to 
  \begin{equation}\label{eq7.1.1}
 \mu^2 _{\cL} (\mathcal F^1 (x) , \mathcal F^1(y)) =  \mathcal F^1 (\mu^2 _{\cD} (x,y)),
\end{equation}
 and 
 \begin{equation}\label{eq7.1.2}\mu^3_{\cL}(\mathcal F^1 (x) , \mathcal F^1(y), \cF^1(z))=0
\end{equation}

 Equation (\ref{eq7.1.1}) merely says that $\cF, \cF^1$ is an ordinary functor on the underlying categories, which is easy to check  using Proposition \ref{easy} and Theorem \ref{maincalc}.
 Equation (\ref{eq7.1.2}) is also true, because Theorem \ref{maincalc}  shows that the 3-fold products $\mu^3_{\cL } (x,y,z)$ 
  vanish  whenever 
  $$x,y,z\in {\rm Image }~\cF^1={\rm Span}\{\cF^1(A_i),\cF^1(C_i),\cF^1(S_{ij}) \}={\rm Span}\{a_0,a_1,c_0,c_1,q_{10},p_{01} \}.$$  
 Finally, since $\cF^1(A_i)=a_i$ and $\cF^n=0$ for $n>1$, this $A_\infty$-functor is strictly unital.
\end{proof}

 Recall   the definitions of the {\em Khovanov merge and split maps}:
 \begin{equation}\label{mergesplit}
 \begin{split}
 &M:A\otimes A\to A, \ M(1\otimes 1)=1,\ M(1\otimes x)=x=M(x\otimes 1), \ M(x\otimes x)=0. \\
& S:A\to A\otimes A, \ S(1)=1\otimes x + x\otimes 1, \ s(x)=x\otimes x.
 \end{split}
 \end{equation}
 These have gradings $p_A(M)=-1 $ and $ p_A(S)= -1.$

   From the definitions (\ref{epeta}),  it is easy to check the following formulas:
 \begin{align}
\label{compsofeta}
\hskip1in \ep\eta=0=\dot\ep\dot\eta,~\ep\dot\eta={\rm Id}_\FF=\dot\ep\eta, ~\eta\dot\ep+\dot\eta\ep={\rm Id}_A \hskip.5in  \\
M=\dot\ep\otimes\dot\ep\otimes \eta+\ep\otimes\dot\ep\otimes\dot \eta
+\dot\ep\otimes\ep\otimes \dot\eta, ~~
S=
\dot\ep\otimes\eta\otimes \dot\eta+\ep\otimes\dot\eta\otimes \eta
+\ep\otimes\dot\eta\otimes \dot\eta\nonumber 
\end{align}

 \medskip

  It is convenient for calculations  to have explicit formulas for $\cF^1(C)$ for a set of generators for
  the set of morphisms 
 between any two 2-tangles.  Any dotted cobordism 
may be expressed, up to isotopy,  as a composition of 
 {\em elementary dotted cobordisms}, defined as either  cobordisms with no dots in $D^2 \times I$ whose projection to $I$ has one Morse critical point, or a product cobordism with  one dot on one component, or the trace of an isotopy. 
 Up to reordering the circle components and applying the same re-orderings to the factors in $A^{\otimes m}$ and $(A^*)^{\otimes n}$, the values of $\cF^1$ on the elementary dotted cobordisms are given in the following proposition.

 \begin{proposition}\label{elementary} The $A_\infty$-functor $\cF$ takes the following values on elementary cobordisms.  
 \begin{enumerate}
\item If $C$ denotes the  cobordism  with no dots,   obtained as the trace of an isotopy of a planar tangle,    $$\cF^1(C)=id_{A^{\otimes m}}\otimes a_\ell\in 
\Hom_\FF(A^{\otimes m},A^{\otimes m})\otimes (L_\ell,L_\ell).$$
\item If $C$ denotes the product cobordism   from $T_\ell(n)$ to itself with a dot on the earringed arc component,
 $$\cF^1(C)=0 \in \Hom_\FF(A^{\otimes m},A^{\otimes m})\otimes (L_\ell,L_\ell).$$
\item If $C$ denotes the product cobordism   from $T_\ell(n)$, with a dot on the un-earringed arc component,
 $$\cF^1(C)=
id_{A^{\otimes m}}\otimes  c_\ell\in \Hom_\FF(A^{\otimes m},A^{\otimes m})\otimes (L_\ell,L_\ell).$$
\item If $C$ denotes the product cobordism   from $T_\ell(n)$, with  a dot on the last circle component,
 $$\cF^1(C)=
 id_{A^{\otimes m-1}}\otimes  (\dot\ep\otimes\dot\eta)\otimes a_\ell\in \Hom_\FF(A^{\otimes m-1},A^{\otimes m-1})\otimes A^*\otimes A\otimes (L_\ell,L_\ell).$$
\item If $C$ denotes the cobordism  from $T_\ell(m)$ to  $T_\ell(m+1)$ with a single critical point of index 0, and the new circle is last, $$\cF^1(C)=
 id_{A^{\otimes m}}\otimes  \eta\otimes a_\ell\in \Hom_\FF(A^{\otimes m},A^{\otimes m})\otimes A \otimes (L_\ell,L_\ell).$$
 \item If $C$ denotes the cobordism  from $T_\ell(m+1)$ to  $T_\ell(m)$ with a single critical point of index 2, and the circle capped off is last, $$\cF^1(C)=
 id_{A^{\otimes m}}\otimes  \ep\otimes a_\ell\in \Hom_\FF(A^{\otimes m},A^{\otimes m})\otimes A^* \otimes (L_\ell,L_\ell).$$
 \item If $C$ denotes the cobordism  from $T_\ell(m)$ to  $T_\ell(m+1)$   with a single critical point of index 1, obtained by performing a band move with  both ends of the band attached to the earringed component,  with the new circle listed last,   
 $$\cF^1(C)=
 id_{A^{\otimes m}}\otimes  \dot\eta\otimes a_\ell\in \Hom_\FF(A^{\otimes m},A^{\otimes m})\otimes A \otimes (L_\ell,L_\ell).$$ 
  \item If $C$ denotes the cobordism  from $T_\ell(m+1)$ to  $T_\ell(m)$ with  a single critical point of index 1, obtained by performing a band move with one end  of the band attached to the earringed component, and the other on  the last   circle, $$\cF^1(C)=
 id_{A^{\otimes m}}\otimes  \dot\ep\otimes a_\ell\in \Hom_\FF(A^{\otimes m},A^{\otimes m})\otimes A^* \otimes (L_\ell,L_\ell).$$
\item If $C$ denotes the cobordism  from $T_\ell(m)$ to  $T_\ell(m+1)$   with a single critical point of index 1, obtained by performing a band move with both ends of the band attached to the unearringed component,   and the new circle listed last,   $$\cF^1(C)=
 id_{A^{\otimes m}}\otimes  (\dot\eta\otimes a_\ell+ \eta\otimes c_\ell)\in \Hom_\FF(A^{\otimes m},A^{\otimes m})\otimes A \otimes (L_\ell,L_\ell).$$
  
  \item If $C$ denotes the cobordism  from $T_\ell(m+1)$ to  $T_\ell(m)$ with a single critical point of index 1, obtained by performing a band move with one end  of the band attached to the unearringed component, and the other on  the last  circle,    $$\cF^1(C)=
 id_{A^{\otimes m}}\otimes  (\dot\ep\otimes a_\ell+ \ep \otimes c_\ell)\in \Hom_\FF(A^{\otimes m},A^{\otimes m})\otimes A^* \otimes (L_\ell,L_\ell).$$
  \item If $C$ denotes the cobordism  from $T_\ell(m+1)$ to  $T_\ell(m)$ with a single critical point of index 1, obtained by performing a band move with the ends  of the band attached to the last two circles, $$\cF^1(C)=
 id_{A^{\otimes m-1}}\otimes  M\otimes a_\ell\in \Hom_\FF(A^{\otimes m-1},A^{\otimes m-1})\otimes \Hom_\FF(A^{\otimes 2}, A) \otimes (L_\ell,L_\ell).$$

  \item  If $C$ denotes the cobordism  from $T_\ell(m)$ to  $T_\ell(m+1)$   with a single critical point of index 1, obtained by performing a band move with both ends  of the band attached to the last   circle, $$\cF^1(C)=
 id_{A^{\otimes m-1}}\otimes  S\otimes a_\ell\in \Hom_\FF(A^{\otimes m-1},A^{\otimes m-1})\otimes \Hom_\FF(A, A^{\otimes 2}) \otimes (L_\ell,L_\ell).$$

  \item  If $C$ denotes the cobordism  from $T_\ell(m)$ to  $T_{1-\ell}(m)$   with a single critical point of index 1, obtained by performing a band move with one end  of the band attached to the earringed component, and the other on the unearringed component, $$\cF^1(C)=
 id_{A^{\otimes m}}\otimes s\in \Hom_\FF(A^{\otimes m},A^{\otimes m})\otimes   (L_\ell,L_{1-\ell})$$
where $s=q_{10}$ or $p_{01}$, according to whether $T_\ell=T_0$ or $T_\ell=T_1$.
\end{enumerate}
\end{proposition}

\begin{proof}
 These are proved using the relations of Figure \ref{dottedfig} and the additional relation (\ref{kill}) that any cobordism with a dot on the earringed component is equivalent to the zero morphism. Figure \ref{move12fig} gives the argument for (12).   The top line is the cobordism $C$ (with $m=1$ and $\ell=0$).  The bottom line equals $$\dot\ep\otimes\dot\eta\otimes\eta\otimes a_0+\dot\ep\otimes\eta\otimes\dot\eta\otimes a_0+\dot\ep\otimes\dot\eta\otimes\dot\eta\otimes a_0+0=S\otimes a_0$$
 where we refer to Equation (\ref{compsofeta}) for the last equality.
    \begin{figure} 
\begin{center}
\def\svgwidth{5.7in}
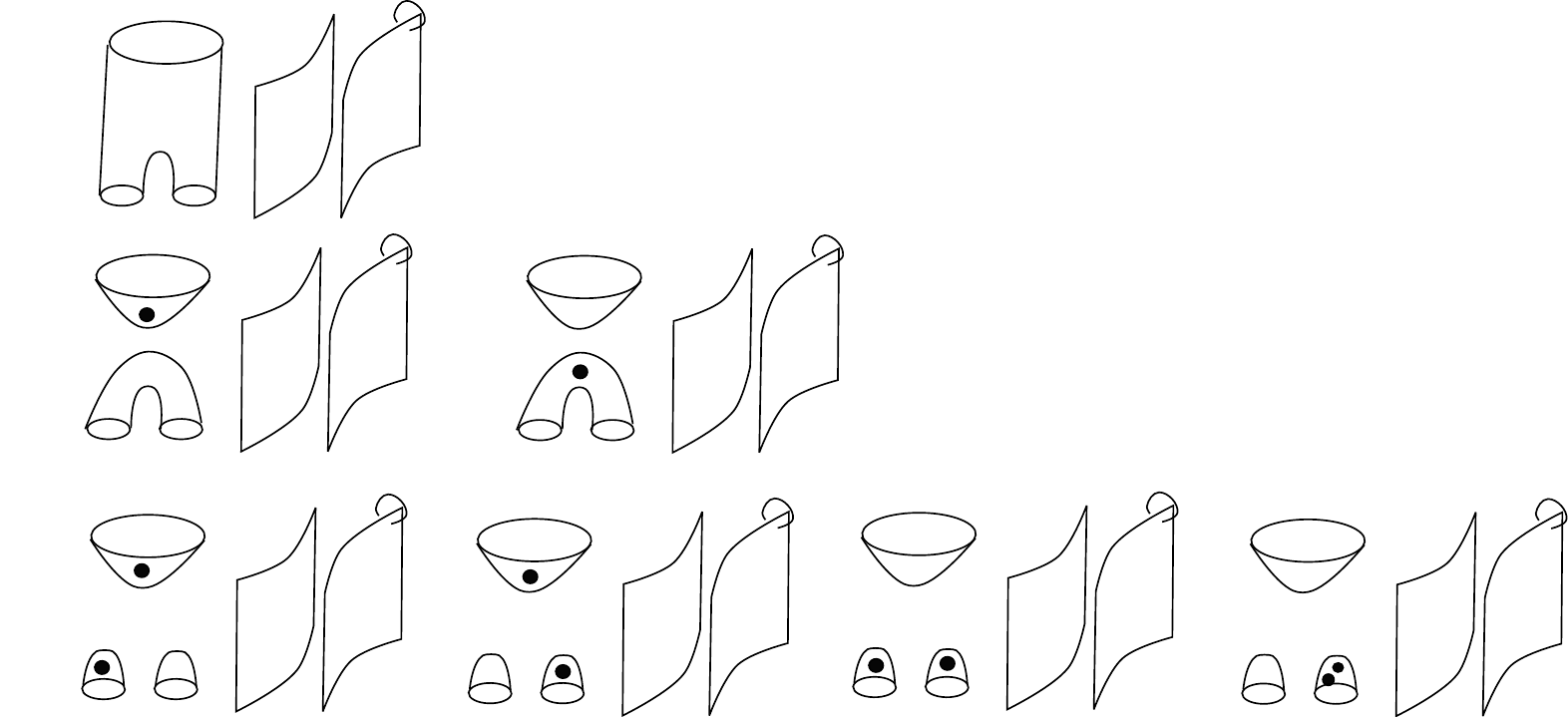
 \caption{\label{move12fig} }
\end{center}
\end{figure} 
Proofs for the  others cases are similar and are  left to the reader. 
 \end{proof}

\subsection{The $A_\infty$-functors $\cG_k$}\label{cGk}  We finish this section by constructing   strictly unital  $A_\infty$-functors
$\cG_k:\cL\to Ch$, $k=0,1$ using the curves $W_0,W_1$. 
Recall that $Ch$ is the dg-category of graded $\FF$  chain complexes.

Suppose that $W$ is an immersed curve in the pillowcase, transverse to $L_0$ and $L_1$, such that all interval components have endpoints mapping to the corners of $P$.

Then $W$ determines the $A_\infty$-module $\cM_W$ over $\cL$, with $$\cM_W(L_i)=(W,L_i), ~ i=0,1,$$   and structure maps
$$
\mu^n_\cL:
(L_{\ell_{n-1}}, L_{\ell_{n}})
 \otimes\dots\otimes(L_{\ell_1}, L_{\ell_2})\otimes \cM_{W}(L_{\ell_1})\to \cM_{W}(L_{\ell_n}).
$$
In particular  $\cM_W(L_i)=(W,L_i), ~ i=0,1$ is equipped with  the differential
$\mu^1_\cL:(W,L_i)\to (W,L_i)$.
Hence $$\big((W, L_i),\mu_\cL^1\big)\in \text{Obj}(Ch).$$

As outlined in Example \ref{exmorita}, $W$ determines the $A_\infty$-functor $$\cG_W:\cL\to Ch$$  by
 \begin{align} 
\label{Louis}
&\cG_W(L_\ell) = \big( (W,L_\ell),\mu^1_\cL\big), ~\ell=0,1 \\
&\cG_W^n : (L_{\ell_{n-1}}, L_{\ell_n})\otimes\dots\otimes  (L_{\ell_0},L_{\ell_1}) \to \Hom_{Ch}( (W,L_{\ell_0}), (W,L_{\ell_1})) \nonumber.  \\
&\cG_W^n(x_n,\dots, x_1)(y)=\mu^{n+1}_\cL(x_n,\dots, x_1,y) \nonumber
\end{align}

We mostly restrict our focus to the two cases $W=W_0$ or $W_1$, and will write $\cG_k,$ $k=0,1,$ rather than $\cG_{W_k}$ to streamline notation.    In these cases, the differentials $\mu^1_\cL:(W_k,L_\ell)\to (W_k,L_\ell)$ are zero, and so the  $\cG_k(L_\ell),$ $k=0,1,$ are chain complexes with trivial differential.

 Moreover, Theorem \ref{testthm} shows that $\cG_{k}^n=0$ for $n>2$, $k=0,1$. But $\cG_{k}^1$ and $\cG_{k}^2$ are non-zero. For example, the statement  $\mu^3(b_0,p_{01},\gamma)=\alpha$ of Theorem \ref{testthm} is equivalent to $\cG^2_{0}(b_0,p_{01}) (\gamma)=\alpha$. Similarly,  the formula $\mu^2(c_1,\rho)=\sigma$ shows that $\cG^1_1(c_1)$ is non-zero.  That $\cG_k$ is strictly unital also follows from the calculations of Theorem \ref{testthm}.

 \section{Three twisted complexes  associated to a tangle}\label{TwistedSister}
 
 We have introduced the category $\cD$, the $A_\infty$-category $\cL $ and its additive enlargement $\Sigma\cL$, two $A_\infty$-modules $\cM_{W_0}, \cM_{W_1}$ over $\cL$, and  the dg-categories $Ch$ and $\Tw Ch$.  We have also introduced  strictly unital  $A_\infty $-functors
 $$\cF:\cD\to \Sigma\cL ,~ \cG_k:\cL\to Ch,~k=0,1.
 $$
 
  As explained in Section \ref{category}, one can form the $A_\infty$-categories 
  $$\Tw \cD ,~\Tw \cL  ,~\Tw Ch.$$

     The $A_\infty$-functors $\cF$ and $\cG_k$ extend to strictly unital $A_\infty$-functors  \cite[Section I.3m]{Seidel}
 $$ \Tw \cF :\Tw\cD\to \Tw\cL ,~  \Tw\cG_k:\Tw\cL\to  \Tw Ch.
 $$

  \medskip

\subsection{Bigradings}\label{bigradingsect}
We pause momentarily to clarify notation and terminology we'll use for   bigrading  on endomorphisms  in $\Sigma \cA$ for those objects whose index set has $\ZZ$-grading.

First, if $V,W$ are $\ZZ$-graded vector spaces,   the usual  $\ZZ$-grading on $\Hom_\FF(V,W)$ is defined by
$$\Hom_\FF^r(V,W) =\bigoplus_{k\in \ZZ}\Hom_\FF(V^k,W^{k+r}).$$

Next, given an  $A_\infty$-category $\cA$ with grading $p_\cA$ on morphisms,   an object $A$ in $\Sigma\cA$ has   the form
 \begin{equation}
\label{objinS}
A=\bigoplus_{i\in I} V_i\otimes A_i
\end{equation}
where $A_i$ are objects in $\cA$,   the $V_i=\{V_i^k\}_{k\in \ZZ}$ are graded vector spaces, and $I$ is some finite indexing set. 
Then\begin{align*}\Hom_{\Sigma\cA}(A,A)&=
\bigoplus_{i\in I}\bigoplus_{j\in I}\Hom_\FF(V_i,V_j)\otimes \Hom_{\cA}(A_i,A_j),\\
&=
\bigoplus_{i\in I}\bigoplus_{j\in I}\left(\bigoplus_{r\in \ZZ}\Hom_\FF^r(V_i,V_j )\right)\otimes \Hom_{\cA}(A_i,A_j).
\end{align*}

Suppose now  that the indexing set $I$ is equipped with a function $$h:I\to \ZZ.$$  Then $\Hom_{\Sigma\cA}(A,A)$ can be bigraded by declaring the bigrading of  $\phi\otimes \psi\in \Hom_\FF^r(V_i,V_j)\otimes \Hom_{\cA}(A_i,A_j)$   to be
 \begin{equation}
\label{bigrading}
\operatorname{bg}(\phi\otimes\psi)=(r+p_\cA(\psi), h(j)-h(i)).
\end{equation}

By definition, if $(A,\delta)$ is an object in $\Tw\cA$, then $\delta\in \Hom_{\Sigma\cA}(A,A)$ has degree 1 with respect to the first component of this bigrading, but {\em a priori} $\delta$ need not be homogenous with respect to the second component of this bidegree.

In the special case when $\cA=Ch$, we can also define a  bigrading on those {\em objects} of $\Sigma(Ch)$ (and hence also the objects of $\Tw(Ch)$) whose indexing set $I$ is equipped with an $h:I\to \ZZ$.  Namely, with $A$ as in Equation (\ref{objinS}), 
any $a\in A$ can be expressed as $\bigoplus_{i\in I} \bigoplus_{r\in \ZZ}v_i^r\otimes a_i\in A$.
Define the bidegree of $v_i^r\otimes a_i\in V_i\otimes A_i$ to be
\begin{equation}
\label{bigradeobjects}
\operatorname{bd}(v_i^r\otimes a_i)=(r+ p_A(a_i), h(i)).
\end{equation}
In this special case the bigrading (\ref{bigradeobjects}) is compatible with  the bigrading (\ref{bigrading}) on morphisms.

\subsection{A twisted complex over $\cD$} Following  Bar-Natan \cite{BN}, we next associate a twisted complex $(D_{(T,o)},d_{(T,o)})\in \Tw\cD$ to an    non-planar 2-tangle $T$  with a fixed (generic) projection to the disk together  with an orientation $o$ on all tangle components. Let $n^-_{(T,o)}$ denote the number of negative crossings, and $n^+_{(T,o)}$ the number of positive crossings. We use $n^\pm$ as shorthand for $n^\pm_{(T,o)}$ when $(T,o)$ is understood. 
 
Suppose the projection of $T$ has $n=n^++n^-$ crossings, and order the crossings of $T$ arbitrarily.  Let $I$ be the set of $2^n$ sequences of zeros and ones, and for each $i\in I$, let $P_i$ be the {\em planar} 2-tangle  obtained by smoothing the crossings of $T$ as determined by $i$.

Define {\em the homological grading function} $$h:I\to \ZZ$$ by
\begin{equation}\label{defh}
h(i)=-n^-+ (\text{the number of ones in ~} i).
\end{equation}

Define a twisted complex over $\cD$ as follows.  First, define $D_{(T,o)} \in \Sigma \cD$ to be 
 \begin{equation}\label{twistcom}
D_{(T,o)}    =  \bigoplus_{i\in I} \FF\{n^- - n^+ -2 h(i)\} \otimes P_i.
\end{equation}
Equip $\Hom_{\Sigma \cD} (D_{(T,o)}, D_{(T,o)})$ with the bigrading described above.  

Given $i,j\in I$, say $j$ {\em covers } $i$ if $j$ is obtained from $i$ by changing a single zero to a one. 
Whenever $j$ covers $i$, note that  $h(i)=h(j)-1$. Assume  $j$ covers $i$ and define 
$R_{ji}$ to be the (obvious) saddle cobordism from $P_i$ to $P_j$. The Euler characteristic of $R_{ji}$ is $1$, and thus $p_\cD(R_{ji})=-2+1=-1$.

 We define
\begin{align*}d_{ji} :=id_\FF\otimes R_{ji}&\in \Hom_\FF(\FF  \{ n^- - n^+ -2 h(i)\} ,\FF \{ n^- - n^+ -2 h(j)\} )\otimes \Hom_\cD (P_i,P_j).\end{align*}
These grading shifts make an adjustment  so that the bigrading of this $d_{ji}$ is $$\operatorname{bg}(d_{ji}) = (1,1).$$ 
Whenever $j$ does not cover $i$, define $d_{ji}=0$.
Define \begin{equation}
\label{BNdiff}d_T:=  \bigoplus_{i,j\in I}d_{ji}\in \Hom_{\Sigma\cD} (D_{(T,o)},D_{(T,o)}).\end{equation}

 Filter $(D_{(T,o)},d_T)$ by taking the $k$th level $F_k(D_{(T,o)})$ to be 
 $$F_k(D_{(T,o)})=   \bigoplus_{i\in I,~ h(i)\geq k} \FF  \{  n^- - n^+-2h(i) \}  \otimes P_i.$$
Notice that $d_{ji}$ is zero when  $h(i)\ge h(j)$ so that $d_T$ is strictly lower triangular with respect to this filtration.

Our   bigrading convention is slightly different from  the usual $(q,h)$  bigrading on the Khovanov complex, which gives the differential   bidegree $(0,1)$.  Our choice is made because Seidel's definition of a twisted complex in \cite{Seidel} requires the differential in a twisted complex to have degree $1$, not $0$.  This explains why the shift in Equation (\ref{twistcom})  is $-2h(i)$, rather than $-h(i)$.  In brief, our bigrading corresponds to $(p+2h,h)=(q+h,h)$. See Theorem \ref{invariance2}.

\begin{lemma} $(D_{(T,o)},d_T)$ is a twisted complex over $\cD$.
 \end{lemma}

\begin{proof}
We noted above that
 $d_T$ is strictly lower triangular with respect to the filtration $F_k$,     as required in the definition of a twisted complex, and that $d_T$ has degree 1 with respect to the first component of the bigrading.

 To see that $(D_{(T,o)},d_T)$ is a twisted complex, it remains only to verify that $\mu^2_{\Sigma\cD}(d,d)=0$, since $\mu^n_\cD=0, n>2$.  The $(i,j)$ component of $\mu^2_{\Sigma\cD}(d,d)$ is the sum of $\mu^2_{\Sigma\cD}(d_{jk},d_{ki})$ over all those $k$ so that $j$ covers $k$ and $k$ covers $i$. Either this sum is empty, so that 
 $\mu^2_{\Sigma\cD}(d,d)_{ij}=0$, or else there are precisely two choices of $k$, say $k'$ and $k''$.  In this latter case  
\[\begin{split}\mu^2_{\Sigma\cD}(d,d)_{ij}&= (id_\FF\circ id_\FF)\otimes\mu^2_\cD(R_{ik'},R_{k'j})+(id_\FF\circ id_\FF)\otimes\mu^2_\cD(R_{ik''},R_{k''j})\\
 &=id_\FF\otimes(R_{ik'}\circ R_{k'j}+R_{ik''}\circ R_{k''j} ).\end{split}\]
 But the cobordisms $R_{ik'}\circ R_{k'j}$ and $R_{ik''}\circ R_{k''j}$ are isotopic, so  this equals zero as well. 
\end{proof} 

The complex $(D_{(T,o)},d_T)$ depends on the orientation only up to an overall shift by the negative writhe $n^--n^+$. We will sometimes streamline notation and write $(D_T,d_T)$ rather  than $(D_{(T,o)},d_T)$.

 \subsection{A twisted complex over $\cL$}\label{L-delta}  Define the twisted complex $(L_{(T,o)},\delta_T)$ to be the image of 
 $(D_{(T,o)}, d_T)$ via the functor $\Tw\cF:\Tw\cD\to \Tw\cL$
 \begin{equation}
(L_{(T,o)},\delta_T)=\Tw\cF(D_{(T,o)},d_T).
\end{equation}
We pause to unwind this definition, to set some notation, and  for the benefit of the reader.

For each $j\in I$, the complete resolution $P_j$ is a   planar  2-tangle, and hence $P_j$ is of type $T_\ell(m)$ for some $\ell\in \{0,1\}$ and non-negative integer $m$.  We use this to define two functions, $$\ell:I\to \{0,1\}, m:I\to \{0,1,2,\dots\}$$ by 
$$P_j=T_{\ell_j}(m_j).$$
  Thus $\ell(i)$ determines how the four boundary points are connected in the complete resolution $P_i$ and $m(i)$ counts the number of circle components of $P_i$.

Then  $L_{(T,o)} \in \Sigma\cL$  equals 
\begin{align*} L_{(T,o)} &=  \bigoplus_{i\in I} \FF \{n^- - n^+-2h(i)\} \otimes \cF(P_i)\\ & = \bigoplus_{i\in I} A^{\otimes m_i}\{n^- - n^+-2h(i)\}\otimes L_{\ell_i}.\end{align*}
Since $\cF^n=0$ for $n\geq  2$, $\delta_T=\cF^1(d_T)$ (\cite[Section I.3m]{Seidel}).    Hence, writing $\delta_{ji}=\cF^1(d_{ji})$,
$$\delta_T=\bigoplus_{j,i\in I}\delta_{ji}.$$

Since each $d_{ji}$ is a band move, that is, an elementary cobordism corresponding to an index 1 critical point, the formulas for $\delta_{ji}=\cF^1(d_{ji})$
are given in Cases  (8), (9), (10), (11), (12), and (13) of Proposition \ref{elementary}.  

   In summary, we associate to  a projection of a 2-tangle $T$, the following object in $\Tw \cL$.   
\begin{equation}
 \label{Lagcom}
 (L_{(T,o)}, \delta_T)=\left( \bigoplus_{i\in I} A^{\otimes m_i}\{ n^- - n^+-2h(i)\} \otimes L_{\ell_i}, \bigoplus_{i,j\in I} 
 \delta_{ji} \right).\end{equation}

This complex depends on the orientation of $T$ only up to an overall shift in the $\mathbb{Z}$-grading by $n^- - n^+$, the negative writhe of the tangle.   If we wished to neglect orientations, therefore, we could associate to an unoriented tangle a well-defined {\em relatively} $\mathbb{Z}$-graded twisted complex of $\mathbb{Z}$-graded  Lagrangians by removing the $n^- - n^+$ in the formula above.  

 We will sometimes streamline notation and  write $(L_T,\delta_T)$ rather than $ (L_{(T,o)}, \delta_T)$.

\subsection{A twisted complex over $Ch$}
 
Finally, for $k=0,1$  we apply the functor $\Tw\cG_k$ to $ (L_{[T,0]}, \delta_T)$ to obtain a twisted complex over the   dg-category $Ch$ of chain complexes:
 \begin{equation}
\label{Khovanovcube}
( K_{[(T,o),k]} , \partial_T)=\Tw\cG_k(L_{(T,o)} , \delta_T)=\Tw\cG_k\circ \Tw\cF(D_{(T,o)},d_T).
\end{equation}
 
  In the remainder of this section, we unravel this definition.
  Since $\cG_k(L_i)=(W_k,L_i)$, from Equations (\ref{Lagcom}) and  (\ref{Khovanovcube}) one sees that 
\begin{equation}
\label{redkho1}   K_{[(T,o),k]}= \bigoplus_{i\in I} A^{\otimes m_i }\{n^- -n^+ -2h(i)\} \otimes (W_k,L_{\ell_i})\in  \Obj(Ch)
\end{equation}
The vector space  $K_{[(T,o),k]}$ is bigraded by Equation (\ref{bigradeobjects}), using the grading function $h$ of Equation (\ref{defh}).

   {\em A priori}, the fact that $\cG_k^2$ is non-zero  (see the paragraph following Equation (\ref{Louis}))  allows for the possibility that $\partial_T$ is not simply  equal to  $\cG_k^1(\delta_T)$.  However,   the (challenging) calculations  of Theorem \ref{testthm} permit us to conclude that $\partial_T$ is in fact equal to $\cG_k^1(\delta_T)$.

\begin{theorem}\label{nohigherdif} For $k=0,1$, the differential $\partial_T$ is $$\partial_T=\cG_k^1(\delta_T)=\bigoplus_{i,j\in I} 
\cG^1_k (\delta_{ji}).$$
In particular, $\partial_T$ has well-defined bidegree $\operatorname{bg}(\partial_T)=(1,1)$, and hence the cohomology $H^*( K_{[(T,o),k]} , \partial_T)$ inherits a bigrading.
\end{theorem}
\begin{proof}
 The formula for $\partial_T=\Tw(\cG_k)(\delta_T)$ can be found in \cite[Section I.3m]{Seidel}. 	Since $\cG_k^n=0$ for $n>2$ (see Section \ref{cGk}), the result is 
$$\partial_T=\cG_k^1(\delta_T)+ \cG_k^2(\delta_T,\delta_T).$$

We claim that  $\cG_k^2(\delta_T,\delta_T)=0$.  To begin with, $$\delta_T=\bigoplus_{i,j\in I}\delta_{ji}=\bigoplus_{i,j\in I}\cF^1(d_{ji})
$$
where $d_{ij}=id_\FF\otimes R_{ij}$, with   $R_{ij}$  a band move, if $j$ covers $i$, and $d_{ij}$ is zero otherwise. There are seven cases to consider, corresponding to the possibilities for index 1 critical points,  which depend on which components in the tangle the band meets.   The formulas for $\delta_{ji}=\cF^1(d_{ji})$
are given in Cases  (7), (8), (9), (10), (11), (12), and (13) of Proposition \ref{elementary}.  Applying $\cG_k ^2$ to any pair of these involves evaluating $\mu^3(x,y,z)$ where $x,y\in \{a_0,a_1, c_0,c_1, p_{01}, q_{10}\}$ and $ z\in \{\alpha,\beta,\gamma,\tau, \rho,\sigma\},$ and a glance at the eight   non-zero $\mu^3$ in Theorem \ref{testthm} shows that such a $\mu^3(x,y,z)$ is always zero. 

Hence $\partial_T=\cG_k^1(\delta_T)$.  The claims about the bidegree and the bigrading follow immediately. 
\end{proof}

\section{Main results}\label{mainresults}
 
\subsection{Homotopy equivalence}

In this section we apply Bar-Natan's homotopy invariance theorem \cite{BN} to show that the homotopy types of  the twisted complexes $(D_{(T,o)},d_T), (L_{(T,o)},\delta_T)$, $( K_{[(T,o),0]},\partial_T)$, and $( K_{[(T,o),1]},\partial_T)$  depend only on the isotopy class of the oriented 2-tangle $(T,o)$.

\medskip

Let $\cA$ be a strictly unital $A_\infty$-category, so that $\Tw\cA$ is also a strictly unital $A_\infty$-category.  Suppose that $(X,\delta )$, $(X',\delta')$ are two  twisted complexes over $\cA$.
Call a morphism $f\in \Hom_{\Tw \cA}((X,\delta),(X',\delta'))$ a {\em   homotopy equivalence} provided there exists a morphism $g\in\Hom_{\Tw \cA}((X',\delta'),(X,\delta))$, as well as morphisms $H\in  \Hom_{\Tw \cA}((X,\delta),(X,\delta))$ and $K\in  \Hom_{\Tw \cA}((X',\delta'),(X',\delta'))$
satisfying\begin{enumerate}
\item $\mu^1_{\Tw\cA}(f)=0, \mu^1_{\Tw\cA}(g)=0$.
\item $\mu^2_{\Tw\cA}(g, f)+e_{(X,\delta)}=\mu^1_{\Tw\cA}(H)$, $\mu^2_{\Tw\cA}(f,g)+e_{(X',\delta')}=\mu^1_{\Tw\cA}(K)$.
 
\end{enumerate}

We omit the routine verification of the following lemma.
\begin{lemma}\label{check}  Homotopy equivalence is an equivalence relation.   Moreover, if $\cF:\cA\to \cB$ is a strictly unital $A_\infty$-functor, the induced strictly unital $A_\infty$-functor $\Tw\cF:\Tw \cA\to \Tw\cB$ takes  homotopy equivalences to  homotopy equivalences.\qed
 \end{lemma}

 Bar-Natan's invariance theorem
\cite[Theorem 1]{BN} says  that the   homotopy class of $(D_{(T,o)},d_T)$ depends only on the isotopy class of the tangle $(T,o)$.

\begin{corollary}\label{invariance} The  homotopy classes of the twisted complexes $(L_{(T,o)},\delta_T)$, $( K_{[(T,o),0]},\partial_T)$,   and   $(K_{[(T,o),1]},\partial_T)$ are invariants of the isotopy class of the oriented 2-tangle $(T,o)$. \qed
\end{corollary}
 
\subsection{Closures of a tangle}
A  2-tangle $T$  can be {\em closed} to produce a knot or link by adding two arcs outside $D^2\times I$ which connect the four boundary points in pairs.  There are two planar choices up to isotopy. Explicitly:

\begin{enumerate}
\item Let $ [T,0]$ denote the {\em 0-closure of  $T$}, that is, the link projection obtained by joining 
the points $-\bbi$ and $1$ by the arc $e^{\bbi\theta},~ \theta\in [-\tfrac \pi 2, 0]$  and the points $\bbi$ and $-1$ by  the arc $e^{\bbi\theta},~ \theta\in [\tfrac \pi 2 ,\pi ]$.
\item Let $ [T,1]$ denote the {\em 1-closure of $T$}, that is, the link projection obtained by joining 
the points $1$ and $\bbi$ by the arc $e^{\bbi\theta},~ \theta\in [0, \tfrac \pi 2]$  and the points $-1$ and $-\bbi$  by  the arc $e^{\bbi\theta},~ \theta\in [\pi, \tfrac {3\pi} 2]$.
\end{enumerate}

For example, referring to Figure \ref{tripleballfig}, $[T_0, 1]$ and $[T_1, 0]$ are unknots, and $[ T_0, 0]$ and $[T_1,1]$ are two component unlinks.

Given an oriented 2-tangle $(T,o)$, the orientation is always compatible with at least one of these closures, but not necessarily both.  When an orientation $o$ on $T$ extends to an orientation on the closure $[ T, k]$, we denote the latter orientation by $\hat o$.  

\medskip

\subsection{Identification of $K_{[(T,o),k]}$ with the reduced Khovanov complex}  
We next establish that, if the orientation $o$ on $T$ extends to  an orientation $\hat o$ on  $[T,k]$,  the twisted complex $(K_{[(T,o),k]},\partial_T)$ is isomorphic to the reduced Khovanov complex of the projection of $([T,k], \hat o)$, which we denote by $C\Kh^{\operatorname{red}}(([T,k], \hat o))$.  Recall (\cite{Kh}) that $C\Kh^{\operatorname{red}}(([T,k], \hat o))$ is equipped with a $\ZZ\oplus \ZZ$ bigrading.

\begin{theorem}\label{invariance2} Assume $(T,o)$ is a tangle with orientation $o$ which extends to  an orientation $\hat o$ on the closure $[T,k]$.  Assume that $\lambda$ is a line field in the isotopy class specified by Equation (\ref{defoflam}) and  $L_0,L_1,W_0,$ and $W_1$ have been $\ZZ$-graded so that the conclusion of Proposition \ref{prop:gradings} holds.

 Then there is a bigraded chain isomorphism between the bigraded chain complex  $( K_{[(T,o),k]},\partial_T)$ and the  reduced Khovanov complex $ C\Kh ^{\operatorname{red}} ([ T,k], \hat o)$ of the closure link projection $[ T,k]$, with the  bigrading $(q+h,h)$.   Hence  
$$H^{r,s}(K_{[(T,o),k]},\partial_T)\cong {\rm Kh}^{\operatorname{red}} ([T,k], \hat o)^{r-s,s}$$
where ${\rm Kh}^{\operatorname{red}} ([T,k], \hat o)$ denotes the reduced Khovanov cohomology of the closure $[T,k]$ of the tangle $T$.
 \end{theorem}

\begin{proof} For simplicity, we will assume $k=0$; the proof for $k=1$ is similar.
Throughout this proof, $(T,o)$ is fixed, so we use the streamlined notation
$$C\Kh=C\Kh([T,0],\hat o),~ K=K_{[(T,o),0]}.$$

The diagram for the closure of $[T,o]$ has the same crossings as $T$, in particular $n^\pm([T,o])=n^\pm(T)$.  Recall that for $i\in I$, $m_i$ denotes the number of circles in the complete resolution $T_{\ell_i}(m_i)$ of $T$ corresponding to $i\in I$. Let $m_i'$ denote the number of circles in the complete resolution of the closure $[T,o]$, and, for convenience, define $\ep_i=m_i'-m_i-1$. Then 
\begin{equation}\label{emi} m_i' =m_i+1+\ep_i=\begin{cases} m_i+2&\text{ if }  T_{\ell_i}=T_0\\  m_i+1& \text{ if } T_{\ell_i}=T_1.
\end{cases}\end{equation}

With $A=\FF[x]/(x^2=0), p_A:A\to\ZZ$, and $R=\langle x\rangle\subset A$ as in Section \ref{circlecomp}, the Khovanov complex has chain groups
$$C\Kh=\bigoplus_{i\in I} A^{\otimes m_i'}\{ n^--n^+-h(i)\}$$
with the  $(q,h)$ bigrading given as
$$C\Kh^{q,h}=\bigoplus_{i\in I, h(i)=h} \{z\in A^{\otimes m_i'}~|~ p_A(z)+h(i) + n^+-n^-=q\}.$$

The reduced Khovanov chain groups are defined as the kernel of multiplication by $x$ in the tensor factor of $A^{\otimes  m_i'}$ corresponding to the earringed component, which, for convenience we assume is placed last in the tensor product. Thus 
\begin{equation}\label{kho1} C\Kh^{\operatorname{red}}=\bigoplus_{i\in I} (A^{\otimes (m_i'-1)}\otimes R)\{ n^--n^+-h(i)\}\end{equation}
with the induced bigrading.

Using (\ref{emi}) and the fact that $(V\otimes W)\{s\}= V\{s\} \otimes W$ for graded vector spaces, we can rewrite (\ref{kho1}) as
\begin{equation}\label{kho2}
C\Kh^{\operatorname{red}}=\bigoplus_{i\in I} A^{\otimes m_i}\{ n^--n^+-h(i)\}\otimes (A^{\otimes\ep_i}\otimes R)
\end{equation}
(where $A^{\otimes 1}=A$ and $A^{\otimes 0}=\FF$).

We compare this to $K$. Repeating Equation (\ref{redkho1})
\begin{equation}
\label{redkho2}   K= \bigoplus_{i\in I} A^{\otimes m_i }\{n^- -n^+ -2h(i)\} \otimes (W_0,L_{\ell_i}).\end{equation}
The grading shifts in (\ref{kho2}) and (\ref{redkho2}) differ by $-h(i)$. 

Now 
\begin{equation}
\label{sm1}
A^{\ep_i}\otimes R=\begin{cases}\langle x\rangle&\text{ if } \ep_i=0\\
\langle 1\otimes x, x\otimes x\rangle&\text{ if } \ep_i=1
\end{cases}
\end{equation}
and 
\begin{equation}
\label{sm2}
(W_0,L_{\ell_i})=\begin{cases}\langle \gamma\rangle&\text{ if } \ep_i=0\\
\langle \alpha,\beta \rangle&\text{ if } \ep_i=1
\end{cases}
\end{equation}
Proposition \ref{prop:gradings} asserts that with the chosen line field $\lambda$ and gradings of $L_i, W_k$, 
$$p_\cL(\gamma)=-1=p_A(x), ~p_\cL(\alpha)=0=p_A(1\otimes x),\text{ and } 
p_\cL(\beta)=-2=p_A(x\otimes x).$$
Thus there are graded isomorphisms 
$$\Phi_i:(W_0,L_{\ell_i})\to A^{\ep_i}\otimes R, ~ \Phi_0(\gamma)=x,~ \Phi_1(\alpha)=1\otimes x, ~ \Phi_1(\beta)=x\otimes x$$
inducing an isomorphism of the chain groups 
\begin{equation}
\Phi:K\to C\Kh^{\operatorname{red}} 
\end{equation}
with $\Phi(K^{r,s})=(C\Kh^{\operatorname{red}})^{r-s,s}$.
\medskip

Proposition \ref{elementary}   and Theorem \ref{nohigherdif} imply that the differentials match. 
Indeed, for merges and splits to circles internal to the tangle $T$, this follows from parts (11) and (12) of Proposition \ref{elementary}.  More interesting are merges or splits corresponding to saddle cobordisms which involve one of the arcs of the tangle, which  correspond to parts (7), (8), (9), (10), and (13).  
  
  We describe in detail the case corresponding to (10); the others are similar (and easier).    This case corresponds to the cobordism $C: T_{\ell_i}(m_i) \to 
  T_{\ell_j}(m_j)$ obtained by splitting off a circle from the unearringed tangle component.  Then, reindexing if needed,  $j$ is obtained from $i$ by replacing the penultimate entry in $i$, which is $0$, by $1$. In particular, $\ell_j=\ell_i\in \{0,1\}$, and $m_j=m_i+1$.
 to streamline notation write $\ell=\ell_j=\ell_i$ and $m=m_i$ so that $m_j=m+1$.
  
 The morphism  
\begin{align*}\delta_{ji}\in \Hom_{\Sigma\cL}(\cF(T_\ell(m)),\cF(T_\ell(m+1)))&=\Hom_{\Sigma\cL}(A^{\otimes m}\otimes L_\ell, A^{\otimes (m+1)}\otimes L_\ell)\\
&=\Hom(A^{\otimes m},A^{\otimes m} \otimes A ) \otimes \Hom_\cL(L_\ell,L_\ell) \end{align*}   
 is given by 
 $$\delta_{ji}=id_{A^{\otimes m}}\otimes (\dot\eta\otimes a_\ell+\eta\otimes c_\ell)= (id_{A^{\otimes m}}\otimes \dot\eta) \otimes a_\ell+(id_{A^{\otimes m}}\otimes\eta)\otimes c_\ell.$$
 Applying $\cG_0$ (and suppressing the $A^{\otimes m}$ factors corresponding to the  circles in $\hat T$ for clarity) one gets, if
 $\ell=0$, the diagram
  \[
\begin{diagram}
\node{(W_0,L_0)}\arrow{e,t}{\Phi_0}\arrow{s,l}{\cG_0  ^1(\delta_{ji})=\mu^2(\dot\eta\otimes a_0+ \eta\otimes c_0,-)}\node{ A\otimes R}\arrow{s,r}{S\otimes id_R}\\ 
\node{A\otimes (W_0,L_0)} \arrow{e,b}{id_{A} \otimes \Phi_0}\node{A\otimes(A\otimes R)}
\end{diagram}\]
which is readily seen to commute, using Theorem \ref{testthm} and Equation (\ref{epeta}).  
Similarly, if $\ell=1$ one obtains
   \[
\begin{diagram}
\node{(W_0,L_1)}\arrow{e,t}{\Phi_1}\arrow{s,l}{\cG_0  ^1(\delta_{ji})=\mu^2(\dot\eta\otimes a_1+ \eta\otimes c_1,-)}\node{ R}\arrow{s,r}{S }\\ 
\node{A\otimes (W_0,L_1)} \arrow{e,b}{id_A \otimes \Phi_1}\node{ A\otimes R}
\end{diagram}\]
 The right vertical maps are those that Khovanov assigns to the corresponding split  in the reduced Khovanov complex.
Cases (7), (8), (9), and (13) are checked similarly.

  Since the components $\cG^k_0$ are zero for $k>1$ by  Theorem \ref{nohigherdif}, 
there are no longer differentials in $(K,\partial_T)$, (nor in the reduced Khovanov complex), and so it follows that $\Phi$ is a chain isomorphism.  This completes the proof for the $k=0$ case;  a similar argument handles the $k=1$ case to finish the proof of Theorem \ref{invariance2}.  \end{proof}

\medskip

 Theorem \ref{invariance2} and the basic fact \cite{Kh} that the Jones polynomial is the bigraded Euler characteristic of the Khovanov complex have the consequence that one can recover the Jones polynomial of the closures $[(T,o),k]$ from the twisted complex $(L_{(T,o)},\delta_T)\in \Tw\cL$.

\section{Calculation of an example}\label{example}

Figure \ref{trefoilexamplefig} illustrates an oriented 2-tangle $T$, with an earring is placed near $1$.
   \begin{figure} [h]
\begin{center}
\includegraphics[width=.3\textwidth]{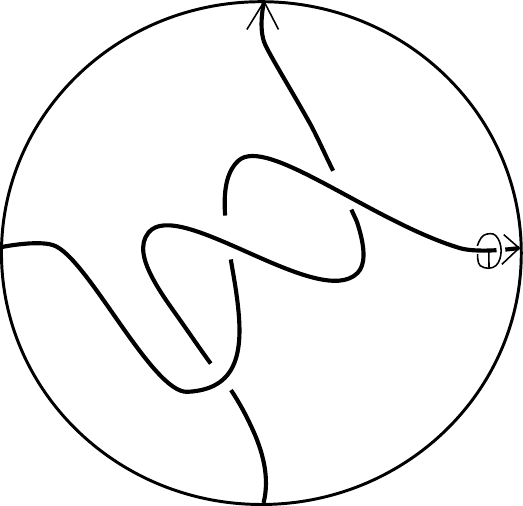}
 \caption{\label{trefoilexamplefig} The tangle $T$ }
\end{center}
\end{figure}

Its Bar-Natan cube of resolutions in $\Tw \cD$  is indicated in Figure     \ref{cubetreffig}. The labellings of arrows correspond to  Cases (7), (8), (9), (10), (11), (12), or (13) of Proposition \ref{elementary} (restated for convenience below), the seven possible band moves.

   \begin{figure} [h]
\begin{center}
\def\svgwidth{6in}
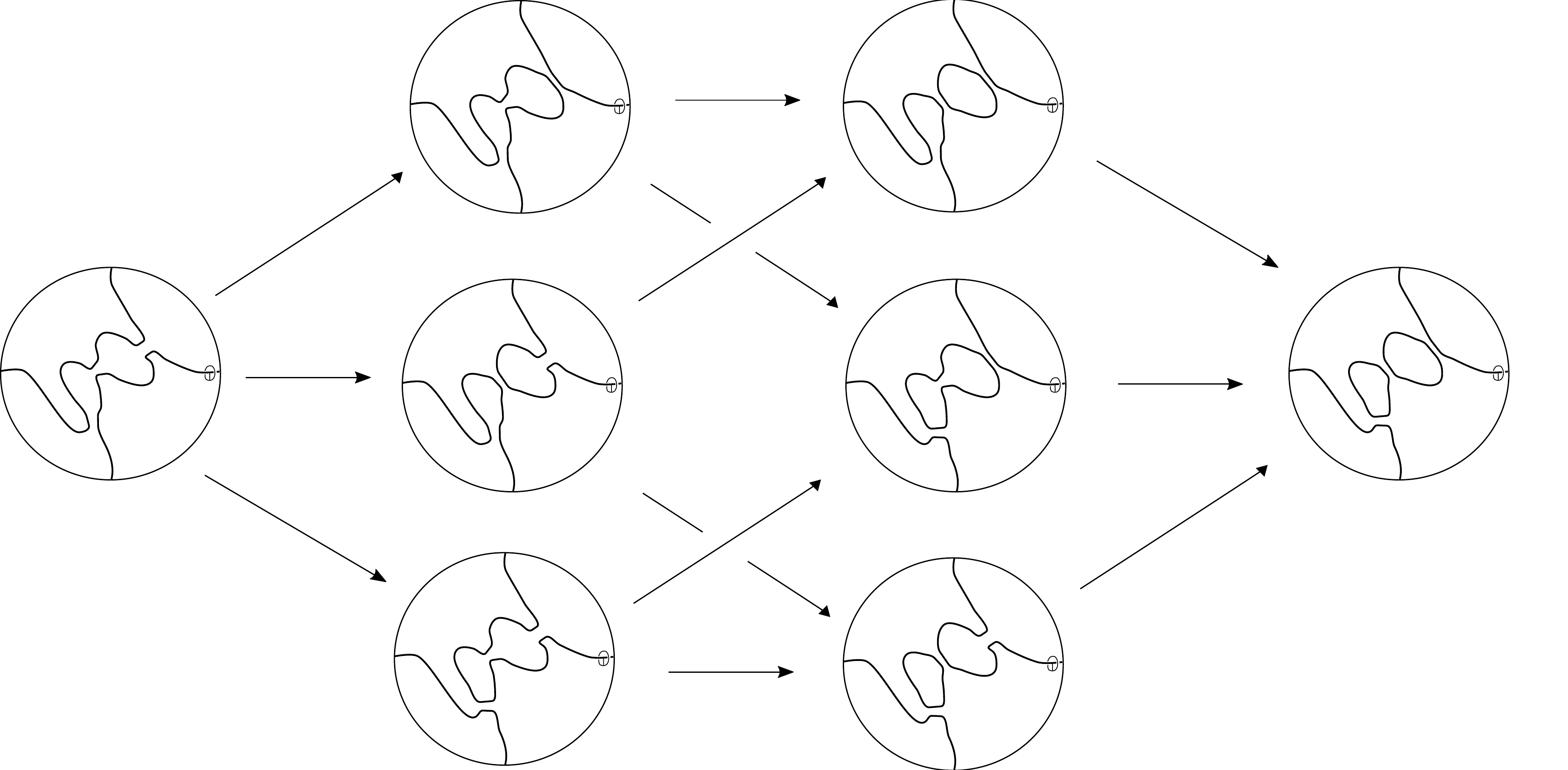
 \caption{\label{cubetreffig} The Bar-Natan twisted complex $D_{T,o}\in \Tw\cD$}
\end{center}
\end{figure}

The $A_\infty$-functor $\cF$ sends the tangle $T$ to the twisted complex $(L_{T}, \delta_T)\in \Tw \cL$ illustrated in Figure \ref{trefdiagfig}.

\begin{figure}
\begin{tikzcd}[column sep=1.5cm, row sep=huge]
&  \FF\{ -5\} \otimes L_1 \arrow[r, "\dot \eta \otimes a_1 + \eta \otimes c_1" ]   \arrow[dr, "\dot \eta \otimes a_1 + \eta \otimes c_1" , pos=0.8]
 &A\{ -7\} \otimes L_1 \arrow[dr, "id_A\otimes (\dot \eta \otimes a_1 + \eta \otimes c_1)" ] \\
\FF\{-3\} \otimes L_0 \arrow[ur, " q_{10}" ] \arrow[r, "q_{10}" , swap]  \arrow[dr, "q_{10}", swap] &  \FF\{ -5\} \otimes L_1\arrow[ur, crossing over
, "\dot \eta \otimes a_1" , pos=0.2
]  &  A\{ -7\} \otimes L_1  \arrow[r, "S\otimes a_1"]  & A^{\otimes 2} \{-9\}\otimes L_1\\
 &  \FF\{ -5\} \otimes L_1 \arrow[ur, , "\dot \eta \otimes a_1" , pos=0.2] \arrow[r, "\dot \eta \otimes a_1", swap] & A\{ -7\} \otimes L_1 \arrow[ur, swap,  "id_A \otimes (\dot \eta \otimes a_1)" ] \arrow[from=ul, crossing over, "\dot \eta \otimes a_1 + \eta \otimes c_1", pos=0.8]
\end{tikzcd}
\caption{The twisted complex $L_{(T,o)}$ in $  \Tw\cL$.\label{trefdiagfig} }
\end{figure}
To write down this complex, one first computes $n^-=0, n^+=3$. This leads to the resulting grading shifts at each vertex.   

 The  complex of Figure \ref{cubetreffig} has its morphisms labeled by elementary cobordisms of types (7), (9), (12), and (13). Their image under $\cF^1$ are given in Proposition \ref{elementary}, and restated below  in slightly abbreviated form. (For general tangles the morphisms might also include (8), (10), and (12).)
 \begin{enumerate}

 \item[(7)] If $C$ denotes the band move  from $T_\ell(m)$ to  $T_\ell(m+1)$ with  both ends of the band attached to the earringed component, $$\cF^1(C)=
 id_{A^{\otimes m}}\otimes  \dot\eta\otimes a_\ell.$$ 
\item[(9)] If $C$ denotes the band move  from $T_\ell(m)$ to  $T_\ell(m+1)$  with both ends of the band attached to the unearringed component,   and the new circle listed last,   $$\cF^1(C)=
 id_{A^{\otimes m}}\otimes  (\dot\eta\otimes a_\ell+ \eta\otimes c_\ell) .$$
  

\item[(12)] If $C$ denotes the band move  from $T_\ell(m)$ to  $T_\ell(m+1)$   obtained by performing a band move with both ends  of the band attached to the last   circle, $$\cF^1(C)=
 id_{A^{\otimes m-1}}\otimes  S\otimes a_\ell .$$

\item[(13)]  If $C$ denotes the band move  from $T_\ell(m)$ to  $T_{1-\ell}(m)$   with one end  of the band attached to the earringed component, and the other on the unearringed component, $$\cF^1(C)=
 id_{A^{\otimes m}}\otimes s $$
where $s=q_{10}$ or $p_{01}$, according to whether $T_\ell=T_0$ or $T_\ell=T_1$.
\end{enumerate}

\medskip

   The 0-closure $[T,0]$ of $T$ is the trefoil knot.  The 
$A_\infty$-functor $\Tw \cG_0$ takes $(L_T,\delta_T)$ to the chain complex $(K_{[(T,o),0]},\partial_T)$  obtained by pairing with $W_0$
$$(K_{[(T,o),0]},\partial_T)=\Tw \cG_0(L_T,\delta_T)= \big(W_0,  (L_T,\delta_T)  \big).$$
The vector spaces $(W_i, L_j)$ are given  in Equation (\ref{fukgens2}).

\hspace{-0.5in}

 \begin{figure}[h]
 \begin{center}
 
\begin{tikzcd}[column sep=1.1cm, row sep=huge]
&  (W_0, L_1)\{ -5\} \arrow[r, "\phi" ]   \arrow[dr, "\phi" , pos=0.8]
 &A \otimes (W_0, L_1) \{ -7\}\arrow[dr, "id_A\otimes \phi" ] \\
  (W_0, L_0)\{-3\} \arrow[ur, " \mbox{$\mu^2(q_{10}, -)$}" ] \arrow[r, " \mbox{$\mu^2(q_{10}, -)$} " , swap]  \arrow[dr, " \mbox{$\mu^2(q_{10}, -)$}", swap] &  (W_0,  L_1) \{ -5\} \arrow[ur, crossing over, "\mbox{$\dot \eta \otimes \mu^2(a_1,-)$}" , pos=0.2
]  &  A \otimes (W_0, L_1)\{ -7\}  \arrow[r, " \mbox{$S\otimes \mu^2(a_1,-)$}"]  & A^{\otimes 2}\otimes (W_0, L_1) \{-9\}\\
 &  (W_0, L_1)\{ -5\}  \arrow[ur, , "\mbox{$\dot \eta \otimes \mu^2(a_1,-)$}" , pos=0.2] \arrow[r, "\mbox{$\dot \eta \otimes \mu^2(a_1,-)$}", swap] & A\otimes (W_0, L_1)\{ -7\} \arrow[ur, swap,  "\mbox{$id_A \otimes (\dot \eta \otimes \mu^2(a_1,-))$}" ] \arrow[from=ul, crossing over, "\phi", pos=0.8]
\end{tikzcd}
\end{center}
\caption{\label{diagram2} The twisted complex $K_{[(T,o),0]}$ in $\Tw\cL$.}
\end{figure}
In Figure \ref{diagram2},  $\phi=\dot \eta \otimes \mu^2(a_1, -) + \eta \otimes \mu^2(c_1, -)$.
The calculations of $\mu^2$ in Theorem \ref{testthm} identify the complex of Figure \ref{diagram2} with the reduced Khovanov complex of the trefoil knot, illustrating the statement of Theorem \ref{invariance2}.

\medskip

In this manner, a 2-tangle $T$ determines the object $L_T\in \Tw \cL$,  where $\cL$ denotes the  Fukaya category with objects  the two figure eights $L_0, L_1$ in $P^*$, the pillowcase with its corners removed.

The twisted complex $L_T$  can be paired with any 
unobstructed, graded immersed curve $W$ in $P^*$. This includes both immersed unions of circles as well as immersed unobstructed arcs  with each endpoint mapping a corner,  such as $W_0$.  
More formally, from a tangle $T$ one obtains, via pairing,  an
$A_\infty$-functor from the (triangulated envelope of the) wrapped Fukaya category of (immersed, unobstructed, graded curves in) the pillowcase to chain complexes: $$\cG_T:\Tw {\rm WrFuk}(P^*)\to Ch.$$
In the above example, applying $\cG_T$ to $W_0\in {\rm WrFuk}(P^*)$ yields the reduced Khovanov complex of the trefoil. 

\medskip

In the articles \cite{HHK1, HHK2} a different construction of a compact Lagrangian in the pillowcase associated to a tangle is proposed, namely,  the perturbed traceless perturbed-flat moduli space of the tangle equipped with an earring, denoted $R^\nat_\pi(T)$, and its restriction to the boundary (see Appendix \ref{tracelesssect} for some discussion).  

\bigskip

Returning to the specific tangle $T$ of Figure  \ref{trefoilexamplefig},  Theorem 7.1 of \cite{HHK1} identifies $R^\nat_\pi(T)$  and its image in the  pillowcase   as the immersed circle illustrated in Figure \ref{pillowtanglefig}. In this case, no perturbation $\pi$ is needed, so $R^\nat_\pi(T)=R^\nat(T)$.
  \begin{figure} [h]
\begin{center}
\def\svgwidth{2.2in}
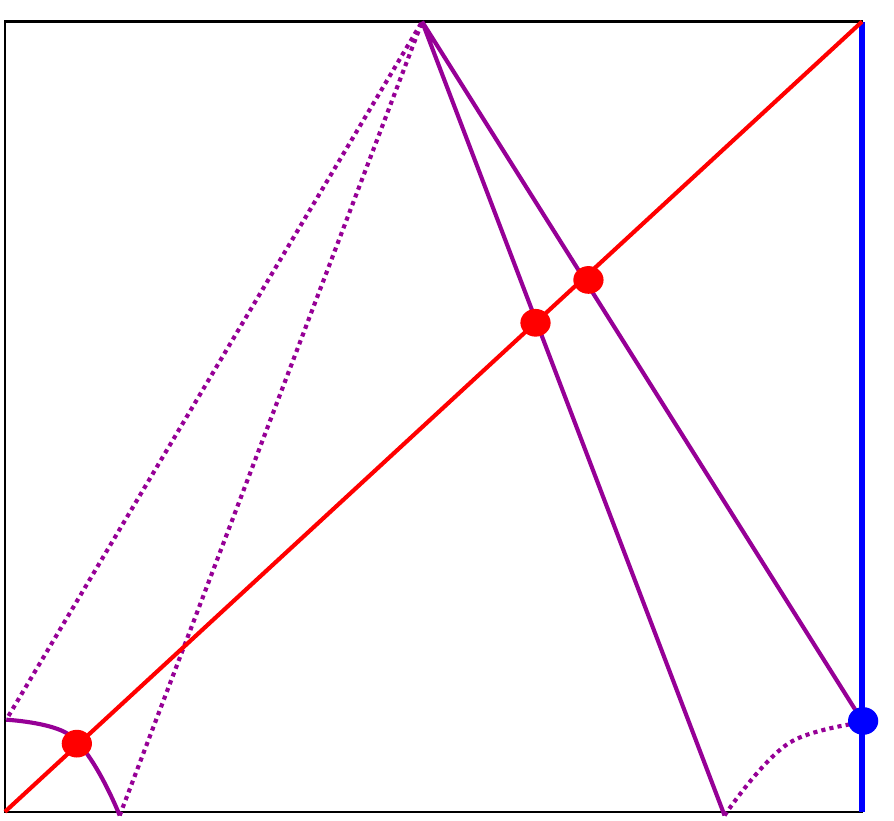
 \caption{\label{pillowtanglefig} The traceless character variety $R^\nat(T)$ of the tangle $T$.}
\end{center}
\end{figure}

Note that the Lagrangian Floer complex $CF(R^\nat(T), W_0)$ is generated by the three intersection points indicated in Figure \ref{pillowtanglefig}. All differentials are zero, as there are no bigons joining these points.  Its homology has rank 3, equal to the rank of the reduced Khovanov  homology (and the singular instanton homology) of the trefoil knot.   Similarly, the Lagrangian Floer complex $CF(R^\nat(T), W_1)$ is generated by the one indicated intersection point and so  its homology has rank 1, equal to the rank of the reduced Khovanov  homology (and the singular instanton homology) of the unknot, which is the 1-closure of $T$. 

In fact, recent work of Kotelskiy, Watson, and Zibrowius \cite{KWZ} shows that $L_T$ and $R^\nat(T)$ for the tangle of Figure  \ref{trefoilexamplefig}   (and many other tangles)   represent quasi-isomorphic objects in $\Tw {\rm WrFuk}(P^*)$, and hence pairing them with {\em any} unobstructed curve $W$ yields isomorphic Lagrangian Floer homologies.

 \color{black}
 
\appendix

\section{Theorems \ref{maincalc} and \ref{testthm}}\label{APa}

In this appendix we indicate how to prove Theorems \ref{maincalc} and \ref{testthm}.  

\medskip
First,   the $\mu^k$  for $k=1,2,3$ can be calculated by  straightforward pencil-and-paper work. For example, the following figures identify certain   $\mu^k$. The remaining ones are found by patient examination. In these pictures the pillowcase is stereographically projected to the plane, with the corner point $(0,\pi)$ sent to $\infty$.

\begin{figure} [H]
\begin{center}
\def\svgwidth{1.7in}
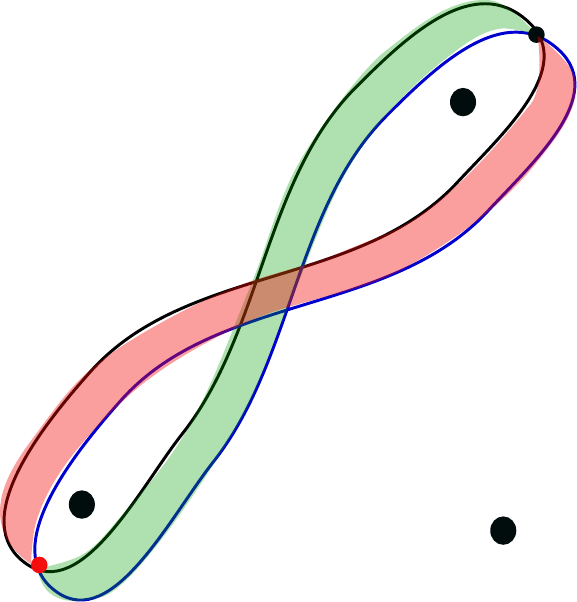
 \caption{ $\mu^1(a_0)=2d_0=0$ }
\end{center}
\end{figure}

\begin{figure}[H]
\begin{center}
\def\svgwidth{2.5in}
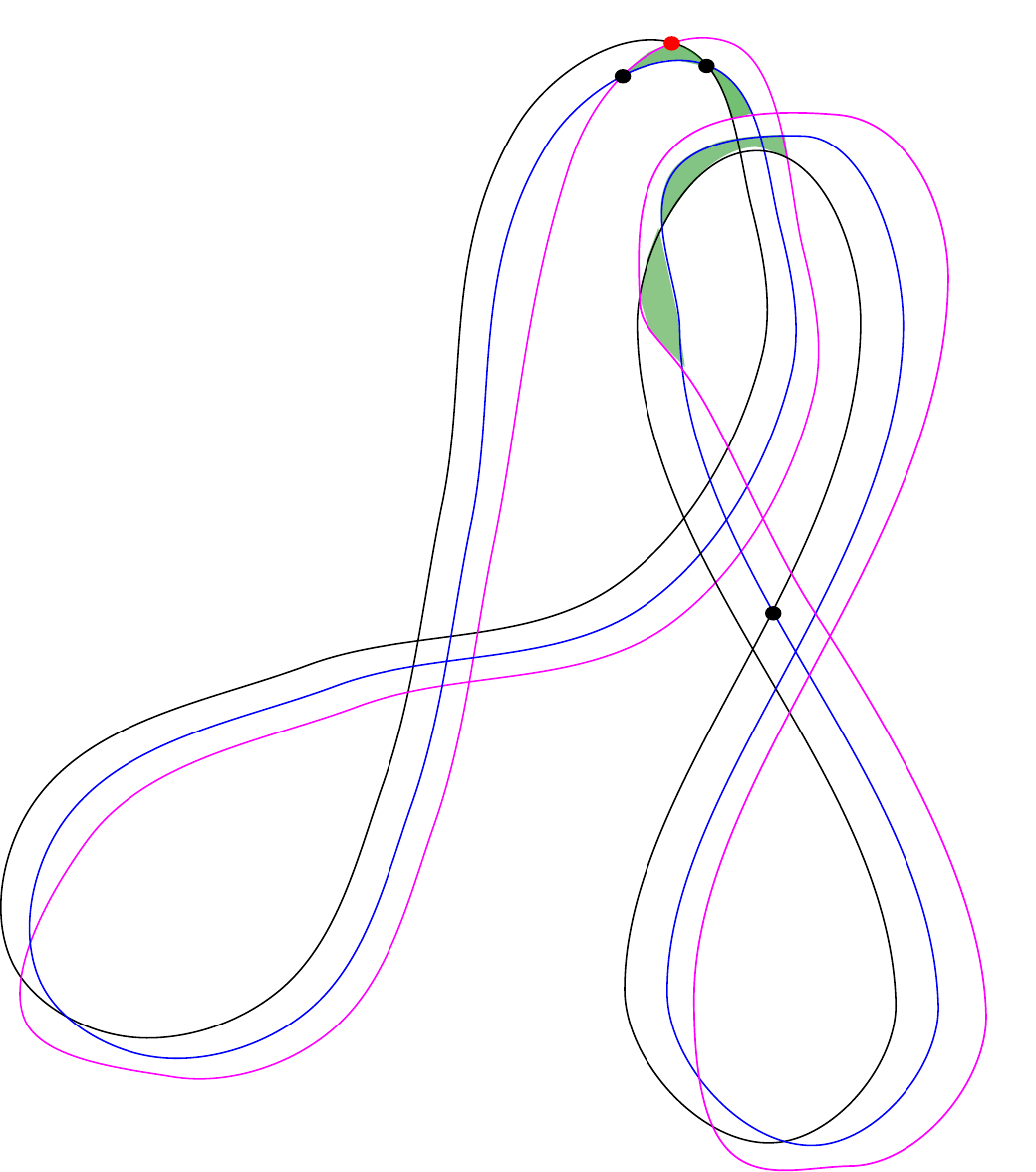
 \caption{  }
\end{center}
\end{figure}

\begin{figure}[h]
\begin{center}
\def\svgwidth{2.4in}
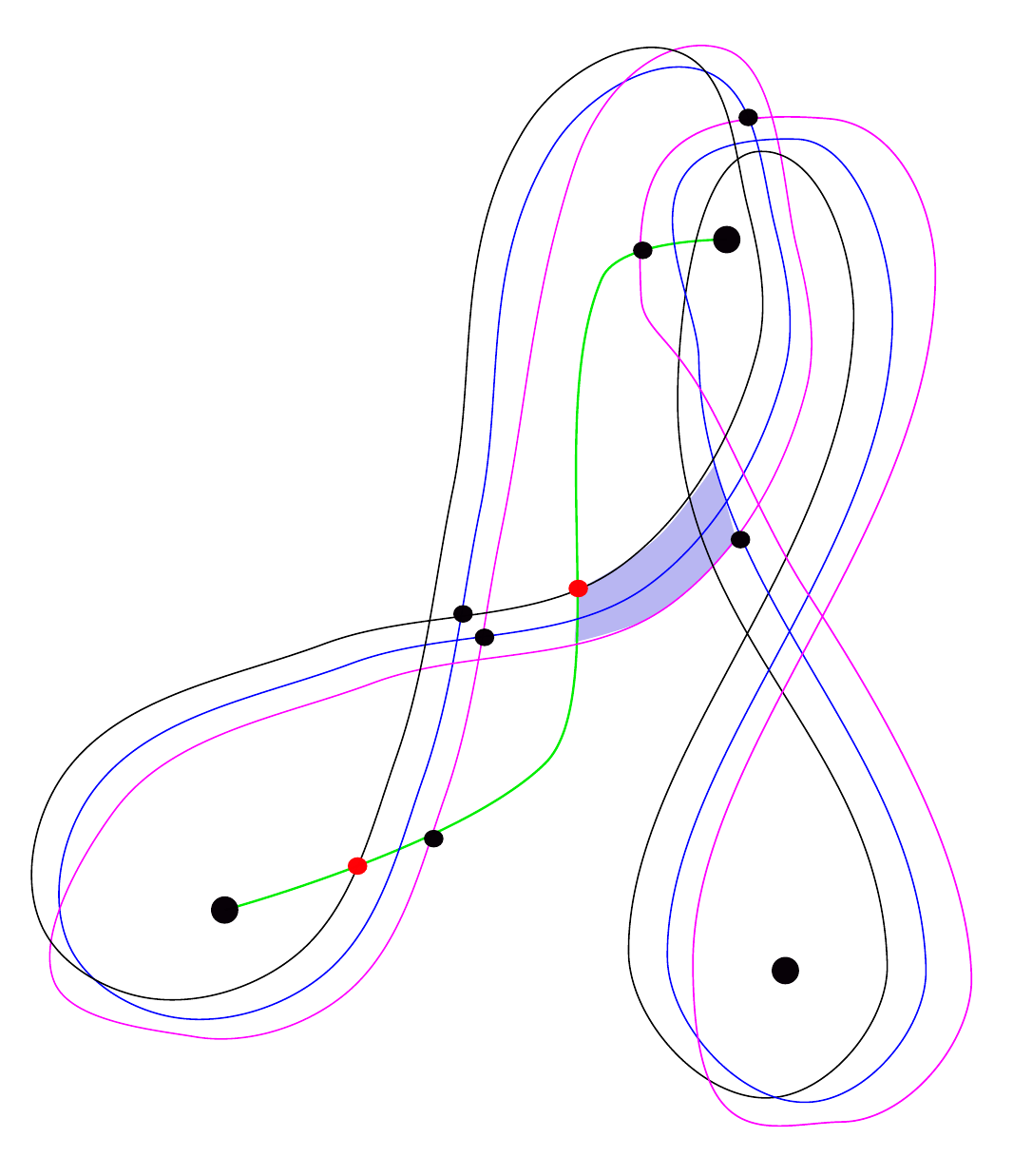
 \caption{  }
\end{center}
\end{figure}

We next outline the argument  that the only immersed convex polygons in the pillowcase giving rise to non-trivial $\mu$ maps are those listed in Theorem \ref{maincalc}.   

If $\mu^n(x_1,\dots, x_n)$ is non-zero, then there exists an orientation preserving immersed convex $(n+1)$-gon in the pillowcase whose edges, cyclically ordered,  lie on successive  pushoffs of the $L_i$. In particular restricting the immersion to a  subdisk with smooth boundary of the polygon yields a smoothly immersed disk.  The problem of determining whether a smoothly immersed curve in $\RR^2$  is the boundary of a smooth orientation preserving immersion of the disk is an unwieldy problem in general. One constraint is that the Gauss map on the boundary circle have degree 1.    There are algorithms \cite{blank} to determine whether a given curve is the boundary of an immersed disk, but such are not suited to our problem, since the problem of listing of all  the  curves that {\em might be} the boundary of an immersed polygon is itself complicated. 
 
A useful technique in determining whether a curve can be the boundary of an immersed disk is to focus on monogons, as we explain in the following definition and proposition.
 
\begin{definition}Let $f:D^2\to S^2$ be an orientation preserving immersion such that the restriction of $f$ to the boundary circle is in general position. A {\it monogon} is an closed arc $A\subset S^1=\partial D^2$ so that the restriction of $f$ to $A$ embeds the interior of $A$ and maps the endpoints of $A$ to the same (transverse double) point in $S^2$.

The double point of a monogon has a well-defined self-intersection number in $\{\pm 1\}$ (obtained by comparing the  two derivatives of $f|_{S^1}$ at the transverse double point  to the orientation of $S^2$). Call the monogon $A\subset D^2\xrightarrow{f} S^2$ {\em positive} or {\em negative} according to whether this self-intersection number is positive and negative. 
The monogon separates $S^2$ into two disks, one which has a convex (i.e., $\frac\pi 2$) corner and one which has a concave (i.e., $\frac{3\pi} 2$) corner.
\end{definition}
  
  Figure \ref{fishyfig} illustrates an immersion of a disk in $\RR^2=S^2\setminus\{\infty\}$ with two monogons, a  negative monogon $f(A)$ and a positive monogon $f(B)$. The disk bounded by $f(A)$ (and $f(B)$) with convex corners in this example is the disk which does not contain $\infty$.

  Note that $f(B)$ lies in the interior of   the  unique disk bounded by $f(A)$.
 The following proposition shows this always happens.  
  
   \begin{figure} 
\begin{center}
\def\svgwidth{2.3in}
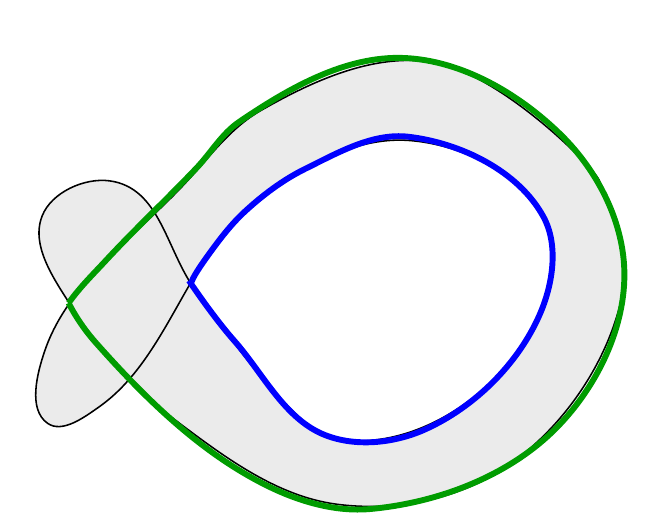
 \caption{\label{fishyfig} }
\end{center}
\end{figure}

\begin{proposition} \label{monocle} Suppose that $f:D^2\to S^2$ is an orientation-preserving  immersion whose boundary is in general position. Suppose that $A\subset \partial D^2\xrightarrow{f} S^2$ is a negative monogon.     Let $U\subset S^2$ denote the disk bounded by $f(A)$ with convex corner. 
 Then there exists another monogon $B\subset \partial D^2\setminus A\xrightarrow{f} S^2$ with $f(B)\subset \text{\rm Int}(U)$. 
\end{proposition}
 
\begin{proof}[Sketch of proof] Suppose that no such $B$ exists. Then the preimage of $U$  by $f|_{\partial D^2}$ is a union of arcs in $\partial D^2$, each of which is embedded in $U$.  By a regular homotopy of $f:D^2\to S^2$ one can slide all embedded arcs  off $U$.  Thus we may arrange, after a regular homotopy, that   $f(\partial D^2)\cap \text{Int}(U)$ is empty. 

Let $C\subset S^2$ be a smooth circle which forms the boundary of an $\epsilon$ neighborhood of (the closed disk) $U$, with $\epsilon$ small and $C$ transverse to $f(\partial D^2)$.  Since $f$ is an immersion, $f$ is transverse to $C$. Since $f(\partial D^2)\cap \text{Int}(U)$ is empty, $f^{-1}(C)$ consists of precisely one arc (and perhaps several circles) with boundary points $x,y\in \partial D^2$, such that the clockwise interval in $\partial D^2$ from  $x$ to $y$ contains $A$ in its exterior.  The situation is illustrated in Figure \ref{fishy2fig}.

   \begin{figure} 
\begin{center}
\def\svgwidth{4.5in}
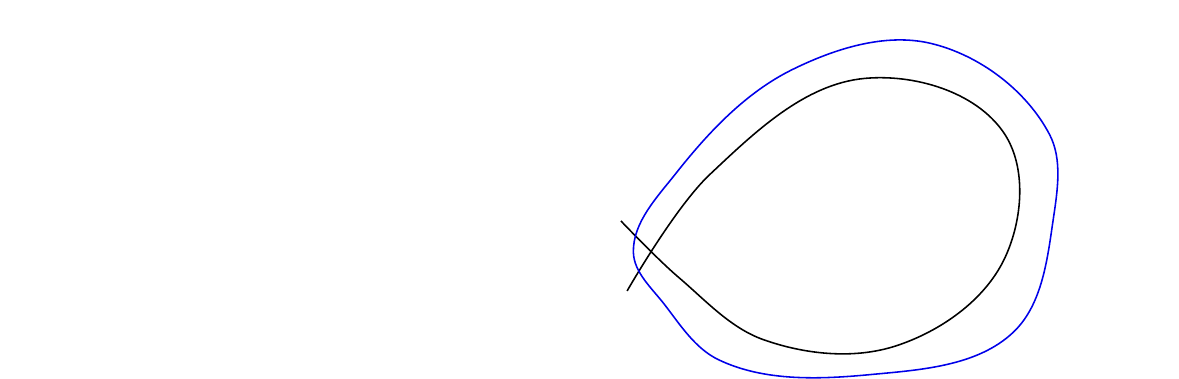
 \caption{\label{fishy2fig} }
\end{center}
\end{figure}

 The arc component of $f^{-1}(C)$ cuts  $ D^2$ into two subdisks.   Let $D'\subset D^2$ denote the subdisk containing $A$. The restriction of $f$ to $D'$ is (after smoothing the corners near $x$ and $y$) an immersion of a disk whose Gauss map along the boundary has degree greater than or equal to 2, since the immersion $f$ must wrap the arc component of $f^{-1}(C)$  monotonically and clockwise around $C$.  But this is impossible, since the Gauss map of the boundary of any immersed disk  must have degree $\pm1$.
\end{proof}

 Proposition \ref{monocle} is used to  greatly reduce the number of cases to be considered in the hunt for immersed polygons in the pillowcase. We illustrate one example, which highlights why $\mu^n=0$ for $n>3$.  
  
 Suppose that  $C\to S$ is a general position immersion of a 1-manifold $C$ in a surface $S$. Suppose an isotopy $H:S\times I\to S$ is given, whose flow lines are transverse to  both branches of the immersion near a transverse double point.   Then near this double point, the first four push-offs of the curve, $C'=H(C,\epsilon), C''=H(C,2\epsilon), C''=H(C,3\epsilon), C'''=H(C,4\epsilon)$,  form a grid, as illustrated in  Figure \ref{griddyfig}.
 
    \begin{figure} 
\begin{center}
\def\svgwidth{2.7in}
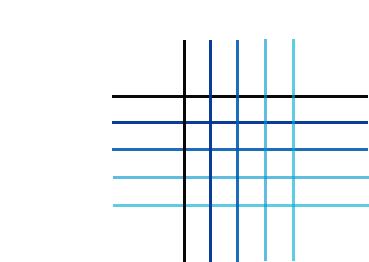
 \caption{\label{griddyfig} }
\end{center}
\end{figure}

  The ten points in this grid indicate potential corners for any immersed convex $5$-gon contributing to  $$\mu^n:(C',C)\times (C'',C')\times (C''',C'')\times (C'''',C''')\to (C'''',C).$$ Following the image of the counterclockwise boundary of an immersed 5-gon whose boundary enters a neighborhood of this grid, one sees that 
 \begin{enumerate}
\item  Corner points can map to one of the indicated 10 points,
\item If three or more consecutive corner points and the arcs between them are mapped into this neighborhood, then the conclusion of Proposition \ref{monocle} is violated. 
\end{enumerate}
 Hence at most 2 corners and the arcs between them are mapped into this grid.

 More generally, similar arguments show the  following result. 
 \begin{lemma} In a neighborhood of a transverse intersection of $L_i$ with $L_j$, where $i,j$ may or may not be equal, the successive push-offs form a grid similar to Figure \ref{griddyfig}.  If the boundary of an immersed convex $n$-gon passes through such a grid region, it  passes through without turning, or else it makes either a left turn, or a left U-turn (i.e., two consecutive left turns, after which it   leaves the grid region on a parallel copy of the curve it entered on.)  \qed \end{lemma}  
 
 \color{black}

 Another useful observation in organizing the enumeration of polygons is to focus on a neighborhood of the the point $d_0$ of Figure \ref{Generatorsfig}.  A illustration of a neighborhood of this point is given in Figure \ref{twistyfig}.    It is straightforward to check that any immersed convex polygon representing some component $\mu^n$ which intersects this neighborhood must intersect it in one of exactly four ways, two of which are illustrated in Figure \ref{twistyfig}.    This is because the complementary regions on both sides of $L_0$ near $d_0$ contain corners of the pillowcase.  
 
 In particular, the polygon boundary must make a left U-turn in this region, involving exactly one or two vertices.   This argument also  shows that any polygon representing a non-zero $\mu^n$ which enters this neighborhood must have output point $d_0$, that is, must  contribute to a calculation of the type $\mu^n(x_1,\dots, x_n)= d_0+ \dots.$  The same argument applies equally well to the portion of $L_1$ around $d_1$.\color{black} 
 
     \begin{figure} 
\begin{center}
\def\svgwidth{2.0in}
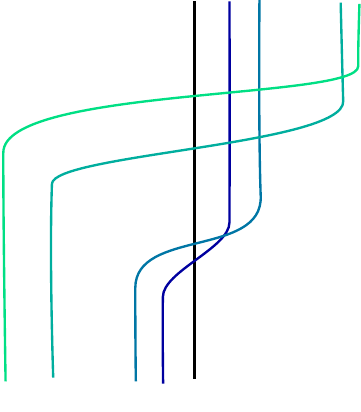
 \caption{\label{twistyfig} }
\end{center}
\end{figure}

 In a similar way, the behavior of a polygon boundary passing through the area around $a_0$ or $a_1$ can be analyzed.  In this case, only one of the two complementary regions next to $L_i$ contains a pillowcase corner, which allows  a polygon boundary to pass through the $a_i$ region without the polygon lying completely in a neighborhood of  $L_i$, but the additional options  to complete this portion with only left turns and Gauss map degree 1 are easily enumerated.  \color{black}  

 A patient reader, armed with these observations, can complete the proof of  Theorems \ref{maincalc} and \ref{testthm}.

\section{Traceless character varieties of planar tangles}\label{tracelesssect}

   In this appendix, we explain how traceless character varieties provide a relationship between  2-tangles and the Fukaya category of the pillowcase. 

In \cite{HHK1}, the  traceless $SU(2)$ character variety $R(Y,L)$ is defined for a codimension 2 pair, and a variant $R^\nat(Y,T)$ is defined when $Y$ is 3-dimensional and an earring (determining a non-trivial $SO(3)$ bundle) has been placed on a component of $T$. Holonomy perturbations $\pi=\{ (C_i, f_i(x)\}_{i=1}^k$  are introduced to make these varieties regular. Here $C_i$ are disjoint solid tori embedded in $Y\setminus T$ and $f_i(x)$ are odd $2\pi$-periodic functions. The holonomy perturbed versions of these spaces (we refer to \cite[Section 7]{HHK1} for their definition) are denoted by $R_\pi(Y,T)$ and $R_\pi^\nat(Y,T)$ respectively.

  When  $L$ is a collection of $k$ points in the 2-sphere, we abbreviate $R(S^2,L)$ as $R(S^2, k)$.
In \cite{HK,K} it is established that $R(S^2, 2n)$ is a compact symplectic variety of dimension $4n-6$ with $2^{2n-2}$ singular points, each of which which admits a  cone neighborhood.  Let $R(S^2, 2n)^*\subset R(S^2, 2n)$  be the complement of the (finite) singular set, a smooth symplectic (non-compact) manifold with $2^{2n-2} $ cylindrical ends.

The pillowcase $P$ can be identified with the traceless $SU(2)$  character variety of the 4-punctured 2-sphere $R(S^2,4)$ (see Equation (\ref{pillowcord}) and \cite[Proposition 3.1]{HHK1}). It can be shown that the symplectic form on the top stratum $P^*$, descended from the Atiyah-Bott-Goldman form on the flat moduli space of a genus 2-surface via symplectic reduction  as in  \cite{HK}, equals $dx\wedge dy$ up to a multiplicative constant. 

The results of \cite{HK} show that if $Y$ is a 3-manifold satisfying $\partial Y=S^2$ and $T$ has $2n$ boundary components,
there exist  small perturbations $\pi$ so that $R^\nat_\pi(Y,T)$ is a smooth compact manifold of dimension $2n-3$, with the restriction map $R^\nat_\pi(Y,T)\to R(S^2,2n)$ a Lagrangian immersion with image in $R(S^2, 2n)^*$.  Also, there exist small perturbations $\pi$ so that $R_\pi(Y,T)$
is compact, has smooth stratum of dimension $2n-3$, and has  finitely many singular points whose neighborhoods are cones on $\CC P^{n-2}$.   The restriction map $R_\pi(Y,T)\to R(S^2,2n)$ is a Lagrangian immersion which  takes singular points to singular points and preserves the cone structure, so that removing the singular points yields a proper Lagrangian immersion $R_\pi(Y,T)^*\to R(S^2, 2n)^*$ with a well defined limit point in the link at infinity. 
In other words, $R^\nat_\pi(Y,T)$ represents an object in the Fukaya category of $R(S^2, 2n)^*$, and $R_\pi(Y,T)^*$ represents an object in the  wrapped Fukaya category of $R(S^2, 2n)^*$.


\medskip

Every {\em planar} 2-tangle is isotopic in $D^2\times I$   to $T_\ell(n)$ for some $\ell=0,1$, $n\geq 0$.  Theorem \ref{fukayadudes2} implies  that   the traceless character variety constructions $R^\nat$  and  $R$, with appropriate perturbation $\pi$,     assign, to $T_\ell(n)$,    {\em Lagrangians with multiplicity}  in the pillowcase $P^*$: 
\begin{align}
\label{tracelessfunctor}
T_\ell(n)&\mapsto R^\nat_\pi(D^2\times I,T_\ell(n)) =\{\pm 1\}^n\times L_\ell \\
 T_\ell(n)&\mapsto R_\pi(D^2\times I,T_\ell(n)) =\{\pm 1\}^n\times W_\ell.\nonumber
\end{align}

\medskip
Adding a small unknotted circle $S$ to a tangle $T$ in a 3-manifold $Y$ replaces the fundamental group of $Y\setminus T$ by its free product with the meridian  of $S$.  Thus on the level of traceless character varieties this replaces $R(Y,T)^*$ by $S^2\times R(Y,T)^*$, where $S^2$ corresponds to the 2-sphere of traceless elements of $SU(2)$. Since  $\partial (Y,T\cup S)=\partial (Y,T)$, for dimension reasons  the restriction $R(Y,T\cup S)\to R(\partial Y, \partial T)$ can never be a Lagrangian immersion. This reasoning applies verbatim to $R_\pi(Y,T)$ and $R_\pi^\nat(Y,T)$. 

The following lemma  and corollary  explain how to choose one more perturbation curve to obtain a new perturbation $\pi'=\pi\cup (C,f)$  so that $R_{\pi'}(Y,T')= S^0\times R_\pi(Y,T)$ and $ R^\nat_{\pi'}(Y,T')= S^0\times R^\nat_\pi(Y,T)$. In Morse theory terminology,  Morse-Bott critical 2-spheres are perturbed into pairs of Morse critical points.   Moreover this process fixes the restriction to the boundary.

\begin{lemma} Let $(Y,T)=(D^3, I)$ be a trivial 1-tangle in the ball.   Let $S\subset Y\setminus T$ be a small unknotted circle 
missing $T$.   Let $C\subset Y\setminus T\cup S$ be the boundary of a disk in $Y$ which intersects $S$ and $T$ transversely in one point each. Let $0<|\ep|<\frac\pi 2$ and denote the perturbation data $\{(C, \ep\sin x)\}$ by $\pi'$.
Then $R_\pi(Y,T\cup S)$ consists of precisely two points.    

If $\mu_T$ and $\mu_S$ denote meridians   for the $T$ and  $S$, respectively,  then these two conjugacy classes  are distinguished by the property that one of them sends $\mu_T$ and $\mu_S$ to the same matrix and the other sends $\mu_T$ and $\mu_S$ to a matrix and its negative.
\end{lemma}
\begin{proof} Identify $SU(2)$ with the unit quaternions, so that traceless matrices correspond to purely imaginary unit quaternions.  Also recall that that the inverse of a purely imaginary quaternion equals its negative, and that the purely imaginary quaternions form a conjugacy class in the unit quaternions. 

Figure \ref{onearcfig} illustrates the arc $T$, the circle $S$, the perturbation curve $C$, and their meridians $\mu_T,\mu_S$ and $c$. The fundamental group $\pi_1(Y\setminus (T\cup S\cup C))$ has the presentation
$$\langle\mu_T,\mu_S,c \mid [c,\mu_S], [c, \mu_T\mu_S^{-1}]\rangle.$$ 
It is simpler but equivalent to write this 
$$\langle \mu_T, \mu_S, c \mid [c,\mu_T]=[c,\mu_S]=1\rangle.$$  
The longitude $\lambda_C$ of $C$ is represented by $\mu_T\mu_S^{-1}$.

  \begin{figure} 
\begin{center}
\def\svgwidth{1.7in}
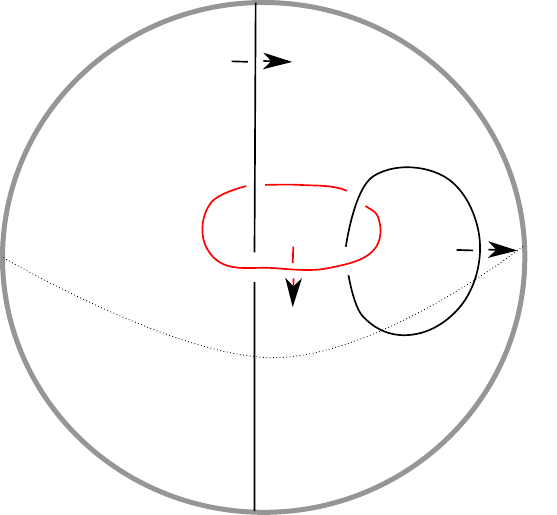
 \caption{\label{onearcfig} }
\end{center}
\end{figure}

The space $R_{\pi'}(Y,T\cup S)$ consists of conjugacy classes of $SU(2)$ representations $\rho$ of $\pi_1(Y\setminus (T\cup S\cup C))$ 
which send $\mu_T$ and $\mu_S$ to purely imaginary unit quaternions and which satisfy the perturbation condition: 

$$\rho(\lambda_C)=e^{\alpha Q} ~\mbox{\text for some purely imaginary quaternion } Q ~\mbox{\text implies } \rho(c)=e^{\ep\sin(\alpha) Q}.$$

If $[\rho]\in R_{\pi'}(Y,T\cup S)$, fix a representative and write $\rho(\lambda_C)=e^{\alpha Q}$ with $\alpha\in [0,\pi]$ and $Q$ a purely imaginary unit quaternion. This determines $Q$ and $\alpha$ uniquely unless $\sin\alpha=0$. Then $$\rho(\mu_S)=\rho(\lambda_C)^{-1}\rho(\mu_T)=e^{-\alpha Q}\rho(\mu_T).$$

 Suppose  that $\sin\alpha\ne 0$. Then $[c, \mu_S]=1$ implies
$\rho(c)\rho(\mu_S)=\rho(\mu_S)\rho(c)
$, and hence
$$e^{\ep\sin\alpha Q}\rho(\mu_T)=\rho(\mu_T)e^{\ep\sin\alpha Q}.$$
Since $\rho(\mu_T)$ is a purely imaginary quaternion, and $\ep\sin\alpha\ne 0$, this implies that $Q=\pm \rho(\mu_T)$. 
Then $\rho(\mu_S)=\pm e^{-\alpha Q}Q=\pm e^{(-\alpha +\tfrac\pi 2)Q}$. But since $\rho(\mu_S)$ is a purely imaginary quaternion, this implies $\cos(\alpha+\tfrac \pi 2)=0$, contradicting $\sin\alpha\ne 0$. 

Hence $\sin\alpha=0$ and thus $\rho(\mu_S)=\pm\rho(\mu_T)$, and $\rho(c)=1$. This shows that there are at most two conjugacy classes in $ R_{\pi'}(Y,T\cup S)$.   But, the two assignments 
$$\mu_S,\mu_T\mapsto \bbi, c\mapsto 1, \mbox{ and } \mu_S\mapsto \bbi,\mu_T\mapsto -\bbi, c\mapsto 1$$
define non-conjugate representations satisfying the perturbation condition,  and hence there are precisely two conjugacy classes, completing the proof.
\end{proof}

\begin{corollary}\label{double}
 Let $(Y,T)$ be a tangle in a 3-manifold, and $\pi$ perturbation data. Let $D\subset Y$ be a 3-ball meeting $R$ in a single unknotted arc, and let  $T'$ be obtained from $T$ by adding a small unknotted circle in $D$ as in Figure \ref{onearcfig}. Then, with  $\pi'=(C,\ep \sin x)$ $0<|\ep|<\tfrac \pi 2$,
 $$R_{\pi\cup \pi'}(Y,T')=R_{\pi}(Y,T)\times \{\pm 1\}$$
and the restriction  $R_{\pi\cup \pi'}(Y, T')\to R(\partial Y, \partial T)$ coincides with the restriction  $R_\pi(Y,T)\to R(\partial Y, \partial T)$ on each of the two copies. The projection $R_{\pi\cup \pi'}(Y,T')\to  \{\pm 1\}$ is given by $[\rho]\mapsto  \rho(\mu_T\mu_S)$.
\end{corollary}

In \cite{HHK1}, the space $R(S^2,4)$ is identified with the pillowcase $P$.  Explicitly, giving $\pi_1(S^2\setminus \{p_a,p_b,p_c,p_d\})$ the presentation 
\begin{equation}\label{presentation}
\langle a,b,c,d \mid  ba=cd\rangle
\end{equation}
 with $a,b,d,a$ the four meridians around the punctures, the pair $(e^{\bbi\gamma}, e^{\bbi\theta})$ corresponds to the traceless representation 
\begin{equation}\label{pillowcord}
a\mapsto\bbi, ~b\mapsto e^{\gamma\bbk}\bbi,~c\mapsto e^{\theta\bbk}\bbi.
\end{equation}
Since $d=c^{-1}ba$, $d$ is mapped to $e^{(\gamma-\theta)\bbk}\bbi$.
Thus $R(S^2,4)$ is homeomorphic to a 2-sphere, with four orbifold points of order 2, corresponding to the four conjugacy classes of abelian representations sending $a,b,c$ to $\bbi, \pm \bbi, \pm\bbi$ respectively.

\medskip

Consider the three tangles $T_\times,T_0,$ and $T_1$ illustrated in Figure \ref{tripleballfig}.  
The traceless character varieties  $R(D^3, T_0)$ and $R^\nat_\pi(D^3, T_0)$ are identified explicitly in \cite{HHK1}.  

Proposition 6.1 of \cite{HHK1} shows that the space $R(D^3,T_0)$ is homeomorphic to the  interval $[0,\pi]$ and the restriction to the boundary properly embeds it in the pillowcase via the map
\begin{equation}
\label{RT0}[0,\pi]\to R(S^2,4),~ t\mapsto \big(a\mapsto\bbi, ~b\mapsto e^{t\bbk}\bbi,~c\mapsto e^{t\bbk}\bbi, ~d\mapsto \bbi \big).
\end{equation}

 Theorem 7.1  of \cite{HHK1} shows that there exists a 1-parameter family of holonomy perturbations along a single curve, $\pi_\delta$, defined for small  $\delta>0$ so that $R^\nat_{\pi_\delta}(D^3, T_0)$ is a smooth circle immersing to the pillowcase with one double point via the map
 \begin{equation}
 \label{RnatT0}\RR/2\pi\ZZ \to R(S^2,4),~ t\mapsto \big(a\mapsto\bbi, ~b\mapsto e^{\tfrac\pi 2 + t +\delta\sin t}\bbi, ~c\mapsto e^{\tfrac\pi 2 + t -\delta\sin t}\bbi, ~d\mapsto e^{-2\delta\sin t}\bbi).\end{equation}

Let $T_i(n)$ denote any tangle in the 3-ball obtained by adding an $n$ component unlink in a small ball to $T_i$, where $i\in \{\times, 0, 1\}$, as in Figure \ref{tripleballfig}.

\begin{theorem}
 \label{fukayadudes2}There exists a holonomy perturbation $\pi=\{(C_i,\ep\sin x)\}_{i=1}^n$ so that 
 $$R_\pi(D^3, T_i(n))=[0,\pi]\times\{\pm 1\}^n.$$ 
 The perturbation $\pi$ can be enlarged by one perturbation curve, $\pi'=\pi\cup \{(C,\delta\sin x)\}$ so that 
 $$R_{\pi'}^\nat(D^3, T_i(n))=(\RR/2\pi\ZZ)\times\{\pm 1\}^n.$$ 
 
 The restrictions to the boundary are the same on each path component and are given by the following formulas  (with $d$ determined by $d=c^{-1}ba$)

  \begin{align} \label{RT}   &R_\pi(D^3,T_i(n))\to R(S^2,4):  \\  ~(t,\sigma)\mapsto
& \begin{cases}  a\mapsto\bbi, ~b\mapsto e^{t\bbk}\bbi,~c\mapsto e^{t\bbk}\bbi  & \text{ if } i=0,\\ 
a\mapsto\bbi, ~b\mapsto -e^{-t\bbk}\bbi,~c\mapsto \bbi    & \text{ if } i=\times.\\
   a\mapsto\bbi, ~b\mapsto -\bbi,~c\mapsto -e^{-t\bbk}\bbi    & \text{ if } i=1,
\end{cases}\nonumber 
\end{align}

  \begin{align} \label{RT2}
&  R^\nat_{\pi'}(D^3,T_i(n))\to R(S^2,4):  
\\     ~(t,\sigma)\mapsto 
&  \begin{cases}
a\mapsto\bbi, ~b\mapsto e^{\tfrac\pi 2 + t +\delta\sin t}\bbi, ~c\mapsto e^{\tfrac\pi 2 + t -\delta\sin t}\bbi& \text{ if } i=0,\\
a\mapsto\bbi,~ b\mapsto -e^{(-\tfrac\pi 2 -t +\delta\sin t)\bbk}\bbi, ~c\mapsto e^{2\delta\sin t\bbk}\bbi& \text{ if } i=\times,\\
a\mapsto\bbi,~ b\mapsto -e^{-2\delta\sin t \bbk}\bbi, ~c\mapsto -e^{(-\tfrac\pi 2 -t -\delta\sin t)\bbk}\bbi  & \text{ if } i=1. 
\end{cases}\nonumber 
\end{align}

 where $\sigma\in \{\pm \}^n$ and $\delta$ is the parameter for the holonomy perturbation $\pi'$.
\end{theorem}
\begin{proof} From Corollary \ref{double} it suffices to establish the formulas when $n=0$.

  The pairs $(D^3, T_0)$, $(D^3, T_1)$ and $(D^3, T_\times)$ are homeomorphic.  There is a homeomorphism $h:D^3\to D^3$ order 3 satisfying $h(T_0)=T_1, h(T_1)=T_\times,$ and $h(T_\times)=T_0$. 
  
  Since $h(D^3, T_\times)=(D^3,T_0)$, $h$ induces, by precomposition, a homeomorphism $h^*:R(D^3,T_0)\to R(D^3,T_\times)$. It follows that the image of the restriction map $R(D^3,T_\times)\to R(S^2)$ is the same as the image of the composite 
$$R(D^3,T_0)\to R(S^2,4)\xrightarrow{(h|_\partial)^*}R(S^2,4).$$  Similarly, $h^2$ induces, by composition, a homeomorphism $(h^2)^*:R(D^3,T_0)\to R(D^3,T_1)$.  Hence the image of the restriction map $R(D^3,T_1)\to R(S^2)$ is the same as the image of the composite 
$$R(D^3,T_0)\to R(S^2,4)\xrightarrow{(h^2|_\partial)^*}R(S^2,4).$$

   Taking care with base points one can arrange that the induced action of the restriction of $h$ to the boundary 4-punctured 2-sphere  acts on the presentation (\ref{presentation})  by 
$$h(a)=c^{-1},~ h(b)=a, ~ h(c)=b^{-1},~h(d)=cdc^{-1}.$$

The induced action of $h$ on the representation $\rho$  of Equation (\ref{pillowcord}) is given by
$$h^*(\rho): a\mapsto \rho(c^{-1})=(e^{\theta\bbk}\bbi)^{-1}=e^{(\theta+\pi)\bbk} \bbi, b\mapsto \bbi, c\mapsto e^{(\gamma+\pi)\bbk}\bbi,~ d\mapsto e^{(\gamma+\theta)\bbk}\bbi.$$
 This is conjugate, by $e^{-(\frac{\theta+\pi}{2})\bbk}$, to  
 \begin{equation}
\label{actonP}
a\mapsto\bbi, ~b\mapsto e^{-(\theta+\pi)\bbk}\bbi,~c\mapsto e^{(\gamma-\theta)\bbk}\bbi, d\mapsto e^{\gamma\bbk}\bbi.
\end{equation}
More succinctly, $(h|_\partial)^*$ transforms the pillowcase coordinates $(\gamma,\theta)$ to $(-\theta-\pi, \gamma-\theta)$.  Hence $(h^2|_\partial)^*$ transforms $(\gamma,\theta)$ to $(\theta-\gamma-\pi, -\gamma-\pi)$. 
  
   Equation (\ref{RT0}) shows that in these coordinates, the restriction $ R(D^3,T_0)\to R(S^2,4)$ is given by $  [0,\pi]\ni t\mapsto (t,t)$. This is taken to $(-t-\pi,0)$ by $(h|_\partial)^*$ and to $(-\pi, -t-\pi)$ by $(h^2|_\partial)^*$.  The formulas of Equation (\ref{RT}) follow.
  
Similarly, Equation  (\ref{RnatT0})  shows that in these coordinates, the restriction $ R^\nat_{\pi'}(D^3,T_0)\to R(S^2,4)$ is given by $  \RR/(2\pi \ZZ)\ni t\mapsto (\tfrac \pi 2 +t +\delta\sin t, \tfrac \pi 2 +t -\delta\sin t)$.  This is taken to 
$(-\tfrac \pi 2 -t +\delta\sin t-\pi,2\delta\sin t)$ by $(h|_\partial)^* $ and to $(-2\delta\sin t -\pi, 
-\tfrac \pi 2 -t-\delta\sin t-\pi)$ by $(h^2|_\partial)^*$. The formulas of Equation (\ref{RT2}) follow.
  \end{proof}

\bibliographystyle{amsplain} 
\bibliography{HHHK1revisedperreferee}

\end{document}

%% file: figures/exacttriples.pdf_tex
\begingroup%
  \makeatletter%
  \providecommand\color[2][]{%
    \errmessage{(Inkscape) Color is used for the text in Inkscape, but the package 'color.sty' is not loaded}%
    \renewcommand\color[2][]{}%
  }%
  \providecommand\transparent[1]{%
    \errmessage{(Inkscape) Transparency is used (non-zero) for the text in Inkscape, but the package 'transparent.sty' is not loaded}%
    \renewcommand\transparent[1]{}%
  }%
  \providecommand\rotatebox[2]{#2}%
  \ifx\svgwidth\undefined%
    \setlength{\unitlength}{400.39984275bp}%
    \ifx\svgscale\undefined%
      \relax%
    \else%
      \setlength{\unitlength}{\unitlength * \real{\svgscale}}%
    \fi%
  \else%
    \setlength{\unitlength}{\svgwidth}%
  \fi%
  \global\let\svgwidth\undefined%
  \global\let\svgscale\undefined%
  \makeatother%
  \begin{picture}(1,0.44553406)%
    \put(0,0){\includegraphics[width=\unitlength,page=1]{exacttriples.pdf}}%
    \put(0.10410707,0.26891146){\color[rgb]{0,0,0}\makebox(0,0)[lb]{\smash{$\color{red}W_0$}}}%
    \put(0.10780051,0.25262159){\color[rgb]{0,0,0}\makebox(0,0)[lb]{\smash{}}}%
    \put(0.41258142,0.36216652){\color[rgb]{0,0,0}\makebox(0,0)[lb]{\smash{$\color{blue} W_1$}}}%
    \put(2.74726777,-0.10491583){\color[rgb]{0,0,0}\makebox(0,0)[lt]{\begin{minipage}{0\unitlength}\raggedright \end{minipage}}}%
    \put(0.1580409,-0.00075773){\color[rgb]{0,0,0}\makebox(0,0)[lb]{\smash{$\color{green}W_\times$}}}%
    \put(0,0){\includegraphics[width=\unitlength,page=2]{exacttriples.pdf}}%
    \put(0.71208841,0.27504022){\color[rgb]{0,0,0}\makebox(0,0)[lb]{\smash{$\color{red}L_0$}}}%
    \put(1.01198227,0.3523604){\color[rgb]{0,0,0}\makebox(0,0)[lb]{\smash{$\color{blue} L_1$}}}%
    \put(0.77460252,0.00046797){\color[rgb]{0,0,0}\makebox(0,0)[lb]{\smash{$\color{green}L_\times$}}}%
    \put(0,0){\includegraphics[width=\unitlength,page=3]{exacttriples.pdf}}%
  \end{picture}%
\endgroup%

%% file: figures/Heraldlemma3.pdf_tex
\begingroup%
  \makeatletter%
  \providecommand\color[2][]{%
    \errmessage{(Inkscape) Color is used for the text in Inkscape, but the package 'color.sty' is not loaded}%
    \renewcommand\color[2][]{}%
  }%
  \providecommand\transparent[1]{%
    \errmessage{(Inkscape) Transparency is used (non-zero) for the text in Inkscape, but the package 'transparent.sty' is not loaded}%
    \renewcommand\transparent[1]{}%
  }%
  \providecommand\rotatebox[2]{#2}%
  \ifx\svgwidth\undefined%
    \setlength{\unitlength}{441.38944371bp}%
    \ifx\svgscale\undefined%
      \relax%
    \else%
      \setlength{\unitlength}{\unitlength * \real{\svgscale}}%
    \fi%
  \else%
    \setlength{\unitlength}{\svgwidth}%
  \fi%
  \global\let\svgwidth\undefined%
  \global\let\svgscale\undefined%
  \makeatother%
  \begin{picture}(1,1.13808223)%
    \put(0,0){\includegraphics[width=\unitlength,page=1]{Heraldlemma3.pdf}}%
    \put(0.19575381,0.38200199){\color[rgb]{0,0,0}\makebox(0,0)[lb]{\smash{$\color{black}L_0$}}}%
    \put(0.206896,0.33464784){\color[rgb]{0,0,0}\makebox(0,0)[lb]{\smash{$\color{blue}L_0'$}}}%
    \put(0.20828877,0.27754424){\color[rgb]{0,0,0}\makebox(0,0)[lb]{\smash{$\color{red}L_0''$}}}%
    \put(0.73336318,0.31375623){\color[rgb]{0,0,0}\makebox(0,0)[lb]{\smash{$\color{black}L_1$}}}%
    \put(0.79603781,0.25804541){\color[rgb]{0,0,0}\makebox(0,0)[lb]{\smash{$\color{blue}L_1'$}}}%
    \put(0.81971489,0.20233458){\color[rgb]{0,0,0}\makebox(0,0)[lb]{\smash{$\color{red}L_1''$}}}%
    \put(0.53355634,0.44657073){\color[rgb]{0,0,0}\makebox(0,0)[lb]{\smash{$\color{green}W_0'''$}}}%
    \put(0,0){\includegraphics[width=\unitlength,page=2]{Heraldlemma3.pdf}}%
    \put(0.96003656,0.53524164){\color[rgb]{0,0,0}\makebox(0,0)[lb]{\smash{$\color{green}W_1'''$}}}%
  \end{picture}%
\endgroup%

%% file: figures/Generators.pdf_tex
\begingroup%
  \makeatletter%
  \providecommand\color[2][]{%
    \errmessage{(Inkscape) Color is used for the text in Inkscape, but the package 'color.sty' is not loaded}%
    \renewcommand\color[2][]{}%
  }%
  \providecommand\transparent[1]{%
    \errmessage{(Inkscape) Transparency is used (non-zero) for the text in Inkscape, but the package 'transparent.sty' is not loaded}%
    \renewcommand\transparent[1]{}%
  }%
  \providecommand\rotatebox[2]{#2}%
  \ifx\svgwidth\undefined%
    \setlength{\unitlength}{296.25961914bp}%
    \ifx\svgscale\undefined%
      \relax%
    \else%
      \setlength{\unitlength}{\unitlength * \real{\svgscale}}%
    \fi%
  \else%
    \setlength{\unitlength}{\svgwidth}%
  \fi%
  \global\let\svgwidth\undefined%
  \global\let\svgscale\undefined%
  \makeatother%
  \begin{picture}(1,1.09498151)%
    \put(0,0){\includegraphics[width=\unitlength]{Generators.pdf}}%
    \put(0.79347852,0.94957186){\color[rgb]{0,0,0}\makebox(0,0)[lb]{\smash{$\color{green}a_1$}}}%
    \put(0.65606713,1.01547781){\color[rgb]{0,0,0}\makebox(0,0)[lb]{\smash{$\color{red}a_0$}}}%
    \put(0.13728509,0.39824425){\color[rgb]{0,0,0}\makebox(0,0)[lb]{\smash{$L_0$}}}%
    \put(0.66023703,0.2613259){\color[rgb]{0,0,0}\makebox(0,0)[lb]{\smash{$L_1$}}}%
    \put(0.28144328,0.14452491){\color[rgb]{0,0,0}\makebox(0,0)[lb]{\smash{$\color{blue}L_0'$}}}%
    \put(0.94131374,0.08747573){\color[rgb]{0,0,0}\makebox(0,0)[lb]{\smash{$\color{blue}L_1'$}}}%
    \put(0.51763444,0.63483212){\color[rgb]{0,0,0}\makebox(0,0)[lb]{\smash{$\color{red}b_0$}}}%
    \put(0.32354355,0.39825348){\color[rgb]{0,0,0}\makebox(0,0)[lb]{\smash{$\color{red}c_0$}}}%
    \put(-0.00000001,0.02662295){\color[rgb]{0,0,0}\makebox(0,0)[lb]{\smash{$\color{red}d_0$}}}%
    \put(0.80249375,-0.04183623){\color[rgb]{0,0,0}\makebox(0,0)[lb]{\smash{$\color{green}d_1$}}}%
    \put(0.81916836,0.34902199){\color[rgb]{0,0,0}\makebox(0,0)[lb]{\smash{$\color{green}c_1$}}}%
    \put(0.82987386,0.60636027){\color[rgb]{0,0,0}\makebox(0,0)[lb]{\smash{$\color{green}b_1$}}}%
    \put(0.81548557,1.07808801){\color[rgb]{0,0,0}\makebox(0,0)[lb]{\smash{$\color{magenta}p_{01}$}}}%
    \put(0.94550162,0.99057717){\color[rgb]{0,0,0}\makebox(0,0)[lb]{\smash{$\color{magenta}p_{10}$}}}%
    \put(0.76672948,0.82930719){\color[rgb]{0,0,0}\makebox(0,0)[lb]{\smash{$\color{magenta}q_{01}$}}}%
    \put(0.73047497,0.67803845){\color[rgb]{0,0,0}\makebox(0,0)[lb]{\smash{$\color{magenta}q_{10}$}}}%
  \end{picture}%
\endgroup%

%% file: figures/parallel.pdf_tex
\begingroup%
  \makeatletter%
  \providecommand\color[2][]{%
    \errmessage{(Inkscape) Color is used for the text in Inkscape, but the package 'color.sty' is not loaded}%
    \renewcommand\color[2][]{}%
  }%
  \providecommand\transparent[1]{%
    \errmessage{(Inkscape) Transparency is used (non-zero) for the text in Inkscape, but the package 'transparent.sty' is not loaded}%
    \renewcommand\transparent[1]{}%
  }%
  \providecommand\rotatebox[2]{#2}%
  \ifx\svgwidth\undefined%
    \setlength{\unitlength}{511.32127176bp}%
    \ifx\svgscale\undefined%
      \relax%
    \else%
      \setlength{\unitlength}{\unitlength * \real{\svgscale}}%
    \fi%
  \else%
    \setlength{\unitlength}{\svgwidth}%
  \fi%
  \global\let\svgwidth\undefined%
  \global\let\svgscale\undefined%
  \makeatother%
  \begin{picture}(1,1.16104694)%
    \put(0,0){\includegraphics[width=\unitlength,page=1]{parallel.pdf}}%
    \put(0.50951585,0.59636521){\color[rgb]{0,0,0}\makebox(0,0)[lb]{\smash{$\alpha$}}}%
    \put(0.28370772,0.35511714){\color[rgb]{0,0,0}\makebox(0,0)[lb]{\smash{$\beta$}}}%
    \put(0.58671531,0.72374411){\color[rgb]{0,0,0}\makebox(0,0)[lb]{\smash{$\gamma$}}}%
    \put(0.53267573,0.40336676){\color[rgb]{0,0,0}\makebox(0,0)[lb]{\smash{$\color{green} W_0$}}}%
    \put(0,0){\includegraphics[width=\unitlength,page=2]{parallel.pdf}}%
    \put(0.78743364,0.83954323){\color[rgb]{0,0,0}\makebox(0,0)[lb]{\smash{$\tau$}}}%
    \put(0.86077308,0.7121643){\color[rgb]{0,0,0}\makebox(0,0)[lb]{\smash{$\rho$}}}%
    \put(0.87621291,0.33967723){\color[rgb]{0,0,0}\makebox(0,0)[lb]{\smash{$\sigma$}}}%
    \put(0.89744276,0.51337587){\color[rgb]{0,0,0}\makebox(0,0)[lb]{\smash{$\color{blue} W_1$}}}%
    \put(0.05982954,0.13509892){\color[rgb]{0,0,0}\makebox(0,0)[lb]{\smash{$L_0$}}}%
    \put(0.60987506,0.06368951){\color[rgb]{0,0,0}\makebox(0,0)[lb]{\smash{$L_1$}}}%
  \end{picture}%
\endgroup%

%% file: figures/tripleball.pdf_tex
\begingroup%
  \makeatletter%
  \providecommand\color[2][]{%
    \errmessage{(Inkscape) Color is used for the text in Inkscape, but the package 'color.sty' is not loaded}%
    \renewcommand\color[2][]{}%
  }%
  \providecommand\transparent[1]{%
    \errmessage{(Inkscape) Transparency is used (non-zero) for the text in Inkscape, but the package 'transparent.sty' is not loaded}%
    \renewcommand\transparent[1]{}%
  }%
  \providecommand\rotatebox[2]{#2}%
  \ifx\svgwidth\undefined%
    \setlength{\unitlength}{407.62836914bp}%
    \ifx\svgscale\undefined%
      \relax%
    \else%
      \setlength{\unitlength}{\unitlength * \real{\svgscale}}%
    \fi%
  \else%
    \setlength{\unitlength}{\svgwidth}%
  \fi%
  \global\let\svgwidth\undefined%
  \global\let\svgscale\undefined%
  \makeatother%
  \begin{picture}(1,0.35030897)%
    \put(0,0){\includegraphics[width=\unitlength]{tripleball.pdf}}%
    \put(0.2157433,0.18488831){\color[rgb]{0,0,0}\makebox(0,0)[lb]{\smash{$a$}}}%
    \put(0.15623304,0.29251312){\color[rgb]{0,0,0}\makebox(0,0)[lb]{\smash{$b$}}}%
    \put(0.03847874,0.23806774){\color[rgb]{0,0,0}\makebox(0,0)[lb]{\smash{$c$}}}%
    \put(0.08785957,0.11271624){\color[rgb]{0,0,0}\makebox(0,0)[lb]{\smash{$d$}}}%
    \put(0.12331249,0.00509138){\color[rgb]{0,0,0}\makebox(0,0)[lb]{\smash{$T_\times$}}}%
    \put(0.46897838,0.01015604){\color[rgb]{0,0,0}\makebox(0,0)[lb]{\smash{$T_0$}}}%
    \put(0.85516191,0.00888991){\color[rgb]{0,0,0}\makebox(0,0)[lb]{\smash{$T_1$}}}%
  \end{picture}%
\endgroup%

%% file: figures/dotted.pdf_tex
\begingroup%
  \makeatletter%
  \providecommand\color[2][]{%
    \errmessage{(Inkscape) Color is used for the text in Inkscape, but the package 'color.sty' is not loaded}%
    \renewcommand\color[2][]{}%
  }%
  \providecommand\transparent[1]{%
    \errmessage{(Inkscape) Transparency is used (non-zero) for the text in Inkscape, but the package 'transparent.sty' is not loaded}%
    \renewcommand\transparent[1]{}%
  }%
  \providecommand\rotatebox[2]{#2}%
  \ifx\svgwidth\undefined%
    \setlength{\unitlength}{496.47866211bp}%
    \ifx\svgscale\undefined%
      \relax%
    \else%
      \setlength{\unitlength}{\unitlength * \real{\svgscale}}%
    \fi%
  \else%
    \setlength{\unitlength}{\svgwidth}%
  \fi%
  \global\let\svgwidth\undefined%
  \global\let\svgscale\undefined%
  \makeatother%
  \begin{picture}(1,0.3485729)%
    \put(0,0){\includegraphics[width=\unitlength]{dotted.pdf}}%
    \put(0.3095315,0.2736953){\color[rgb]{0,0,0}\makebox(0,0)[lb]{\smash{$=0$}}}%
    \put(0.61100955,0.27889319){\color[rgb]{0,0,0}\makebox(0,0)[lb]{\smash{$=0$}}}%
    \put(0.95277125,0.28539056){\color[rgb]{0,0,0}\makebox(0,0)[lb]{\smash{$=1$}}}%
    \put(0.27704465,0.09176889){\color[rgb]{0,0,0}\makebox(0,0)[lb]{\smash{$=$}}}%
    \put(0.63959799,0.07617518){\color[rgb]{0,0,0}\makebox(0,0)[lb]{\smash{$+$}}}%
  \end{picture}%
\endgroup%

%% file: figures/move12.pdf_tex
\begingroup%
  \makeatletter%
  \providecommand\color[2][]{%
    \errmessage{(Inkscape) Color is used for the text in Inkscape, but the package 'color.sty' is not loaded}%
    \renewcommand\color[2][]{}%
  }%
  \providecommand\transparent[1]{%
    \errmessage{(Inkscape) Transparency is used (non-zero) for the text in Inkscape, but the package 'transparent.sty' is not loaded}%
    \renewcommand\transparent[1]{}%
  }%
  \providecommand\rotatebox[2]{#2}%
  \ifx\svgwidth\undefined%
    \setlength{\unitlength}{754.28837891bp}%
    \ifx\svgscale\undefined%
      \relax%
    \else%
      \setlength{\unitlength}{\unitlength * \real{\svgscale}}%
    \fi%
  \else%
    \setlength{\unitlength}{\svgwidth}%
  \fi%
  \global\let\svgwidth\undefined%
  \global\let\svgscale\undefined%
  \makeatother%
  \begin{picture}(1,0.45752354)%
    \put(0,0){\includegraphics[width=\unitlength]{move12.pdf}}%
    \put(0.00228496,0.24552294){\color[rgb]{0,0,0}\makebox(0,0)[lb]{\smash{$=$}}}%
    \put(0.2859798,0.24682429){\color[rgb]{0,0,0}\makebox(0,0)[lb]{\smash{$+$}}}%
    \put(-0.00096842,0.07764848){\color[rgb]{0,0,0}\makebox(0,0)[lb]{\smash{$=$}}}%
    \put(0.26896463,0.07098554){\color[rgb]{0,0,0}\makebox(0,0)[lb]{\smash{$+$}}}%
    \put(0.51926677,0.07628855){\color[rgb]{0,0,0}\makebox(0,0)[lb]{\smash{$+$}}}%
    \put(0.76426591,0.07204614){\color[rgb]{0,0,0}\makebox(0,0)[lb]{\smash{$+$}}}%
  \end{picture}%
\endgroup%

%% file: figures/cubetref.pdf_tex
\begingroup%
  \makeatletter%
  \providecommand\color[2][]{%
    \errmessage{(Inkscape) Color is used for the text in Inkscape, but the package 'color.sty' is not loaded}%
    \renewcommand\color[2][]{}%
  }%
  \providecommand\transparent[1]{%
    \errmessage{(Inkscape) Transparency is used (non-zero) for the text in Inkscape, but the package 'transparent.sty' is not loaded}%
    \renewcommand\transparent[1]{}%
  }%
  \providecommand\rotatebox[2]{#2}%
  \ifx\svgwidth\undefined%
    \setlength{\unitlength}{1779.42217489bp}%
    \ifx\svgscale\undefined%
      \relax%
    \else%
      \setlength{\unitlength}{\unitlength * \real{\svgscale}}%
    \fi%
  \else%
    \setlength{\unitlength}{\svgwidth}%
  \fi%
  \global\let\svgwidth\undefined%
  \global\let\svgscale\undefined%
  \makeatother%
  \begin{picture}(1,0.49131338)%
    \put(0,0){\includegraphics[width=\unitlength,page=1]{cubetref.pdf}}%
    \put(0.16278693,0.35844971){\color[rgb]{0,0,0}\makebox(0,0)[lb]{\smash{(13)}}}%
    \put(0.17749208,0.22113922){\color[rgb]{0,0,0}\makebox(0,0)[lb]{\smash{(13)}}}%
    \put(0.14770711,0.12719484){\color[rgb]{0,0,0}\makebox(0,0)[lb]{\smash{(13)}}}%
    \put(0.46176037,0.44115446){\color[rgb]{0,0,0}\makebox(0,0)[lb]{\smash{(9)}}}%
    \put(0.39878802,0.32017478){\color[rgb]{0,0,0}\makebox(0,0)[lb]{\smash{(7)}}}%
    \put(0.51889504,0.32501188){\color[rgb]{0,0,0}\makebox(0,0)[lb]{\smash{(9)}}}%
    \put(0.51985115,0.1241601){\color[rgb]{0,0,0}\makebox(0,0)[lb]{\smash{(9)}}}%
    \put(0.39672066,0.13211893){\color[rgb]{0,0,0}\makebox(0,0)[lb]{\smash{(7)}}}%
    \put(0.44953415,0.03510707){\color[rgb]{0,0,0}\makebox(0,0)[lb]{\smash{(7)}}}%
    \put(0.75796096,0.37593741){\color[rgb]{0,0,0}\makebox(0,0)[lb]{\smash{(9)}}}%
    \put(0.73361624,0.26217304){\color[rgb]{0,0,0}\makebox(0,0)[lb]{\smash{(12)}}}%
    \put(0.75777554,0.1375448){\color[rgb]{0,0,0}\makebox(0,0)[lb]{\smash{(7)}}}%
  \end{picture}%
\endgroup%

%% file: figures/pillowtangle1.pdf_tex
\begingroup%
  \makeatletter%
  \providecommand\color[2][]{%
    \errmessage{(Inkscape) Color is used for the text in Inkscape, but the package 'color.sty' is not loaded}%
    \renewcommand\color[2][]{}%
  }%
  \providecommand\transparent[1]{%
    \errmessage{(Inkscape) Transparency is used (non-zero) for the text in Inkscape, but the package 'transparent.sty' is not loaded}%
    \renewcommand\transparent[1]{}%
  }%
  \providecommand\rotatebox[2]{#2}%
  \newcommand*\fsize{\dimexpr\f@size pt\relax}%
  \newcommand*\lineheight[1]{\fontsize{\fsize}{#1\fsize}\selectfont}%
  \ifx\svgwidth\undefined%
    \setlength{\unitlength}{428.93450303bp}%
    \ifx\svgscale\undefined%
      \relax%
    \else%
      \setlength{\unitlength}{\unitlength * \real{\svgscale}}%
    \fi%
  \else%
    \setlength{\unitlength}{\svgwidth}%
  \fi%
  \global\let\svgwidth\undefined%
  \global\let\svgscale\undefined%
  \makeatother%
  \begin{picture}(1,0.91782179)%
    \lineheight{1}%
    \setlength\tabcolsep{0pt}%
    \put(0,0){\includegraphics[width=\unitlength,page=1]{pillowtangle1.pdf}}%
    \put(0.71452133,0.78527573){\color[rgb]{0,0,0}\makebox(0,0)[lt]{\lineheight{1.25}\smash{\begin{tabular}[t]{l}$\color{red} W_0$\end{tabular}}}}%
    \put(0.84501133,0.45242309){\color[rgb]{0,0,0}\makebox(0,0)[lt]{\lineheight{1.25}\smash{\begin{tabular}[t]{l}$\color{blue} W_1$\end{tabular}}}}%
    \put(0.52791153,0.10232493){\color[rgb]{0,0,0}\makebox(0,0)[lt]{\lineheight{1.25}\smash{\begin{tabular}[t]{l}$\color{violet} R^\nat_\pi(T)$\\\\\end{tabular}}}}%
  \end{picture}%
\endgroup%

%% file: figures/2bigons.pdf_tex
\begingroup%
  \makeatletter%
  \providecommand\color[2][]{%
    \errmessage{(Inkscape) Color is used for the text in Inkscape, but the package 'color.sty' is not loaded}%
    \renewcommand\color[2][]{}%
  }%
  \providecommand\transparent[1]{%
    \errmessage{(Inkscape) Transparency is used (non-zero) for the text in Inkscape, but the package 'transparent.sty' is not loaded}%
    \renewcommand\transparent[1]{}%
  }%
  \providecommand\rotatebox[2]{#2}%
  \ifx\svgwidth\undefined%
    \setlength{\unitlength}{276.67268047bp}%
    \ifx\svgscale\undefined%
      \relax%
    \else%
      \setlength{\unitlength}{\unitlength * \real{\svgscale}}%
    \fi%
  \else%
    \setlength{\unitlength}{\svgwidth}%
  \fi%
  \global\let\svgwidth\undefined%
  \global\let\svgscale\undefined%
  \makeatother%
  \begin{picture}(1,1.0435981)%
    \put(0,0){\includegraphics[width=\unitlength,page=1]{2bigons.pdf}}%
  \end{picture}%
\endgroup%

%% file: figures/mu2_1.pdf_tex
\begingroup%
  \makeatletter%
  \providecommand\color[2][]{%
    \errmessage{(Inkscape) Color is used for the text in Inkscape, but the package 'color.sty' is not loaded}%
    \renewcommand\color[2][]{}%
  }%
  \providecommand\transparent[1]{%
    \errmessage{(Inkscape) Transparency is used (non-zero) for the text in Inkscape, but the package 'transparent.sty' is not loaded}%
    \renewcommand\transparent[1]{}%
  }%
  \providecommand\rotatebox[2]{#2}%
  \ifx\svgwidth\undefined%
    \setlength{\unitlength}{498.75165804bp}%
    \ifx\svgscale\undefined%
      \relax%
    \else%
      \setlength{\unitlength}{\unitlength * \real{\svgscale}}%
    \fi%
  \else%
    \setlength{\unitlength}{\svgwidth}%
  \fi%
  \global\let\svgwidth\undefined%
  \global\let\svgscale\undefined%
  \makeatother%
  \begin{picture}(1,1.13103305)%
    \put(0,0){\includegraphics[width=\unitlength,page=1]{mu2_1.pdf}}%
    \put(0.59178033,1.11579971){\color[rgb]{0,0,0}\makebox(0,0)[lb]{\smash{$\mu^2(a_0,a_0)=a_0$}}}%
    \put(0,0){\includegraphics[width=\unitlength,page=2]{mu2_1.pdf}}%
    \put(0.75176621,1.04476417){\color[rgb]{0,0,0}\makebox(0,0)[lb]{\smash{$\mu^2(a_0,p_{01})=p_{01}$}}}%
    \put(0,0){\includegraphics[width=\unitlength,page=3]{mu2_1.pdf}}%
    \put(0.68145273,0.83273488){\color[rgb]{0,0,0}\makebox(0,0)[lb]{\smash{$\mu^2(a_1,a_1)=a_1$}}}%
    \put(0,0){\includegraphics[width=\unitlength,page=4]{mu2_1.pdf}}%
    \put(0.89949709,0.97408771){\color[rgb]{0,0,0}\makebox(0,0)[lb]{\smash{$\mu^2(a_1,p_{10})=p_{10}$}}}%
    \put(0,0){\includegraphics[width=\unitlength,page=5]{mu2_1.pdf}}%
    \put(0.73232837,0.74285469){\color[rgb]{0,0,0}\makebox(0,0)[lb]{\smash{$\mu^2(q_{01},a_1)=q_{01}$}}}%
    \put(0,0){\includegraphics[width=\unitlength,page=6]{mu2_1.pdf}}%
    \put(0.16654999,0.33632756){\color[rgb]{0,0,0}\makebox(0,0)[lb]{\smash{$\mu^2(c_0,b_0)=d_0$}}}%
    \put(0,0){\includegraphics[width=\unitlength,page=7]{mu2_1.pdf}}%
    \put(0.69656839,0.26154324){\color[rgb]{0,0,0}\makebox(0,0)[lb]{\smash{$\mu^2(b_1,c_1)=d_1$}}}%
    \put(0,0){\includegraphics[width=\unitlength,page=8]{mu2_1.pdf}}%
  \end{picture}%
\endgroup%

%% file: figures/mu3T.pdf_tex
\begingroup%
  \makeatletter%
  \providecommand\color[2][]{%
    \errmessage{(Inkscape) Color is used for the text in Inkscape, but the package 'color.sty' is not loaded}%
    \renewcommand\color[2][]{}%
  }%
  \providecommand\transparent[1]{%
    \errmessage{(Inkscape) Transparency is used (non-zero) for the text in Inkscape, but the package 'transparent.sty' is not loaded}%
    \renewcommand\transparent[1]{}%
  }%
  \providecommand\rotatebox[2]{#2}%
  \ifx\svgwidth\undefined%
    \setlength{\unitlength}{511.32127176bp}%
    \ifx\svgscale\undefined%
      \relax%
    \else%
      \setlength{\unitlength}{\unitlength * \real{\svgscale}}%
    \fi%
  \else%
    \setlength{\unitlength}{\svgwidth}%
  \fi%
  \global\let\svgwidth\undefined%
  \global\let\svgscale\undefined%
  \makeatother%
  \begin{picture}(1,1.16104694)%
    \put(0,0){\includegraphics[width=\unitlength,page=1]{mu3T.pdf}}%
    \put(0.68321454,0.54232567){\color[rgb]{0,0,0}\makebox(0,0)[lb]{\smash{$\mu^3(q_{01},q_{10},\alpha)=\alpha$}}}%
    \put(0,0){\includegraphics[width=\unitlength,page=2]{mu3T.pdf}}%
    \put(-0.0424596,0.38213683){\color[rgb]{0,0,0}\makebox(0,0)[lb]{\smash{$\mu^3(c_0,b_0,\beta)=\beta$}}}%
    \put(0,0){\includegraphics[width=\unitlength,page=3]{mu3T.pdf}}%
    \put(0.25282797,0.82410336){\color[rgb]{0,0,0}\makebox(0,0)[lb]{\smash{$\mu^3(b_0,p_{01},\gamma)=\alpha$}}}%
  \end{picture}%
\endgroup%

%% file: figures/fishy.pdf_tex
\begingroup%
  \makeatletter%
  \providecommand\color[2][]{%
    \errmessage{(Inkscape) Color is used for the text in Inkscape, but the package 'color.sty' is not loaded}%
    \renewcommand\color[2][]{}%
  }%
  \providecommand\transparent[1]{%
    \errmessage{(Inkscape) Transparency is used (non-zero) for the text in Inkscape, but the package 'transparent.sty' is not loaded}%
    \renewcommand\transparent[1]{}%
  }%
  \providecommand\rotatebox[2]{#2}%
  \ifx\svgwidth\undefined%
    \setlength{\unitlength}{321.58200001bp}%
    \ifx\svgscale\undefined%
      \relax%
    \else%
      \setlength{\unitlength}{\unitlength * \real{\svgscale}}%
    \fi%
  \else%
    \setlength{\unitlength}{\svgwidth}%
  \fi%
  \global\let\svgwidth\undefined%
  \global\let\svgscale\undefined%
  \makeatother%
  \begin{picture}(1,0.78002926)%
    \put(0,0){\includegraphics[width=\unitlength,page=1]{fishy.pdf}}%
    \put(0.71201991,0.00175289){\color[rgb]{0,0,0}\makebox(0,0)[lb]{\smash{$\color{green} f(A)$}}}%
    \put(0.40724126,0.44964523){\color[rgb]{0,0,0}\makebox(0,0)[lb]{\smash{$\color{blue} f(B)$}}}%
  \end{picture}%
\endgroup%

%% file: figures/fishy2.pdf_tex
\begingroup%
  \makeatletter%
  \providecommand\color[2][]{%
    \errmessage{(Inkscape) Color is used for the text in Inkscape, but the package 'color.sty' is not loaded}%
    \renewcommand\color[2][]{}%
  }%
  \providecommand\transparent[1]{%
    \errmessage{(Inkscape) Transparency is used (non-zero) for the text in Inkscape, but the package 'transparent.sty' is not loaded}%
    \renewcommand\transparent[1]{}%
  }%
  \providecommand\rotatebox[2]{#2}%
  \ifx\svgwidth\undefined%
    \setlength{\unitlength}{565.63906212bp}%
    \ifx\svgscale\undefined%
      \relax%
    \else%
      \setlength{\unitlength}{\unitlength * \real{\svgscale}}%
    \fi%
  \else%
    \setlength{\unitlength}{\svgwidth}%
  \fi%
  \global\let\svgwidth\undefined%
  \global\let\svgscale\undefined%
  \makeatother%
  \begin{picture}(1,0.32126568)%
    \put(0,0){\includegraphics[width=\unitlength,page=1]{fishy2.pdf}}%
    \put(0.79100832,0.3024609){\color[rgb]{0,0,0}\makebox(0,0)[lb]{\smash{$\color{blue} C$}}}%
    \put(0.6944803,0.12319466){\color[rgb]{0,0,0}\makebox(0,0)[lb]{\smash{$U$}}}%
    \put(0,0){\includegraphics[width=\unitlength,page=2]{fishy2.pdf}}%
    \put(0.36034497,0.11046573){\color[rgb]{0,0,0}\makebox(0,0)[lb]{\smash{$\xrightarrow{f}$}}}%
    \put(0,0){\includegraphics[width=\unitlength,page=3]{fishy2.pdf}}%
    \put(0.26593848,0.21335815){\color[rgb]{0,0,0}\makebox(0,0)[lb]{\smash{$\color{green} A$}}}%
    \put(0,0){\includegraphics[width=\unitlength,page=4]{fishy2.pdf}}%
    \put(0.21820486,0.01605916){\color[rgb]{0,0,0}\makebox(0,0)[lb]{\smash{$x$}}}%
    \put(0.19062544,0.27912447){\color[rgb]{0,0,0}\makebox(0,0)[lb]{\smash{$y$}}}%
    \put(0.52157854,0.15607779){\color[rgb]{0,0,0}\makebox(0,0)[lb]{\smash{$f(x)$}}}%
    \put(0.52794303,0.04682088){\color[rgb]{0,0,0}\makebox(0,0)[lb]{\smash{$f(y)$}}}%
    \put(0,0){\includegraphics[width=\unitlength,page=5]{fishy2.pdf}}%
    \put(0.02939188,0.12743763){\color[rgb]{0,0,0}\makebox(0,0)[lb]{\smash{$\color{blue}f^{-1}(C)$}}}%
    \put(0.18669148,0.15592982){\color[rgb]{0,0,0}\makebox(0,0)[lb]{\smash{$D'$}}}%
  \end{picture}%
\endgroup%

%% file: figures/griddy.pdf_tex
\begingroup%
  \makeatletter%
  \providecommand\color[2][]{%
    \errmessage{(Inkscape) Color is used for the text in Inkscape, but the package 'color.sty' is not loaded}%
    \renewcommand\color[2][]{}%
  }%
  \providecommand\transparent[1]{%
    \errmessage{(Inkscape) Transparency is used (non-zero) for the text in Inkscape, but the package 'transparent.sty' is not loaded}%
    \renewcommand\transparent[1]{}%
  }%
  \providecommand\rotatebox[2]{#2}%
  \ifx\svgwidth\undefined%
    \setlength{\unitlength}{177.02128853bp}%
    \ifx\svgscale\undefined%
      \relax%
    \else%
      \setlength{\unitlength}{\unitlength * \real{\svgscale}}%
    \fi%
  \else%
    \setlength{\unitlength}{\svgwidth}%
  \fi%
  \global\let\svgwidth\undefined%
  \global\let\svgscale\undefined%
  \makeatother%
  \begin{picture}(1,0.70977889)%
    \put(0,0){\includegraphics[width=\unitlength,page=1]{griddy.pdf}}%
    \put(0.13618446,0.42275713){\color[rgb]{0,0,0}\makebox(0,0)[lb]{\smash{$C$}}}%
    \put(0.13252426,0.35321254){\color[rgb]{0,0,0}\makebox(0,0)[lb]{\smash{$C'$}}}%
    \put(0.13618434,0.2873284){\color[rgb]{0,0,0}\makebox(0,0)[lb]{\smash{$C''$}}}%
    \put(0.13252417,0.21659692){\color[rgb]{0,0,0}\makebox(0,0)[lb]{\smash{$C'''$}}}%
    \put(0.13618446,0.12993796){\color[rgb]{0,0,0}\makebox(0,0)[lb]{\smash{$C''''$}}}%
    \put(0.47658639,0.63871106){\color[rgb]{0,0,0}\makebox(0,0)[lb]{\smash{$C$}}}%
    \put(0.53881033,0.64237077){\color[rgb]{0,0,0}\makebox(0,0)[lb]{\smash{$C'$}}}%
    \put(0,0){\includegraphics[width=\unitlength,page=2]{griddy.pdf}}%
  \end{picture}%
\endgroup%

%% file: figures/twisty.pdf_tex
\begingroup%
  \makeatletter%
  \providecommand\color[2][]{%
    \errmessage{(Inkscape) Color is used for the text in Inkscape, but the package 'color.sty' is not loaded}%
    \renewcommand\color[2][]{}%
  }%
  \providecommand\transparent[1]{%
    \errmessage{(Inkscape) Transparency is used (non-zero) for the text in Inkscape, but the package 'transparent.sty' is not loaded}%
    \renewcommand\transparent[1]{}%
  }%
  \providecommand\rotatebox[2]{#2}%
  \ifx\svgwidth\undefined%
    \setlength{\unitlength}{173.10444629bp}%
    \ifx\svgscale\undefined%
      \relax%
    \else%
      \setlength{\unitlength}{\unitlength * \real{\svgscale}}%
    \fi%
  \else%
    \setlength{\unitlength}{\svgwidth}%
  \fi%
  \global\let\svgwidth\undefined%
  \global\let\svgscale\undefined%
  \makeatother%
  \begin{picture}(1,1.14079231)%
    \put(0,0){\includegraphics[width=\unitlength,page=1]{twisty.pdf}}%
    \put(0.51574714,0.02444385){\color[rgb]{0,0,0}\makebox(0,0)[lb]{\smash{$C$}}}%
    \put(0.40269324,0.0138422){\color[rgb]{0,0,0}\makebox(0,0)[lb]{\smash{$C'$}}}%
    \put(-0.04891442,0.01037653){\color[rgb]{0,0,0}\makebox(0,0)[lb]{\smash{$C^{(n)}$}}}%
    \put(0.09303883,0.0138422){\color[rgb]{0,0,0}\makebox(0,0)[lb]{\smash{$C^{(n-1)}$}}}%
    \put(-0.0467926,-0.08912895){\color[rgb]{0,0,0}\makebox(0,0)[lt]{\begin{minipage}{1.37604785\unitlength}\raggedright \end{minipage}}}%
    \put(0,0){\includegraphics[width=\unitlength,page=2]{twisty.pdf}}%
  \end{picture}%
\endgroup%

%% file: figures/onearc.pdf_tex
\begingroup%
  \makeatletter%
  \providecommand\color[2][]{%
    \errmessage{(Inkscape) Color is used for the text in Inkscape, but the package 'color.sty' is not loaded}%
    \renewcommand\color[2][]{}%
  }%
  \providecommand\transparent[1]{%
    \errmessage{(Inkscape) Transparency is used (non-zero) for the text in Inkscape, but the package 'transparent.sty' is not loaded}%
    \renewcommand\transparent[1]{}%
  }%
  \providecommand\rotatebox[2]{#2}%
  \ifx\svgwidth\undefined%
    \setlength{\unitlength}{267.02370605bp}%
    \ifx\svgscale\undefined%
      \relax%
    \else%
      \setlength{\unitlength}{\unitlength * \real{\svgscale}}%
    \fi%
  \else%
    \setlength{\unitlength}{\svgwidth}%
  \fi%
  \global\let\svgwidth\undefined%
  \global\let\svgscale\undefined%
  \makeatother%
  \begin{picture}(1,0.92544077)%
    \put(0,0){\includegraphics[width=\unitlength]{onearc.pdf}}%
    \put(0.52912348,0.76655348){\color[rgb]{0,0,0}\makebox(0,0)[lb]{\smash{$\mu_T$}}}%
    \put(0.85921344,0.38739624){\color[rgb]{0,0,0}\makebox(0,0)[lb]{\smash{$\mu_S$}}}%
    \put(0.49343813,0.32494678){\color[rgb]{0,0,0}\makebox(0,0)[lb]{\smash{$\color{red}c$}}}%
    \put(0.26891746,0.50486061){\color[rgb]{0,0,0}\makebox(0,0)[lb]{\smash{$\color{red}C$}}}%
    \put(0.70181286,0.64977164){\color[rgb]{0,0,0}\makebox(0,0)[lb]{\smash{$S$}}}%
    \put(0.49364802,0.12787271){\color[rgb]{0,0,0}\makebox(0,0)[lb]{\smash{$T$}}}%
  \end{picture}%
\endgroup%